\documentclass[a4paper,reqno]{amsart}
\pdfoutput=1

\usepackage[english]{babel}
\usepackage{lmodern}
\usepackage[T1]{fontenc}
\usepackage[utf8]{inputenc}

\DeclareUnicodeCharacter{211D}{$\mathbb{R}$}
\DeclareUnicodeCharacter{1D54A}{$\Sp$}
\DeclareUnicodeCharacter{03B1}{$\alpha$}
\DeclareUnicodeCharacter{0393}{$\Gamma$}

\usepackage{amsmath}
\usepackage{amssymb}
\usepackage{amsthm}
\usepackage{mathtools}
\usepackage{thmtools}
\usepackage{graphicx}
\usepackage{caption}

\usepackage[normalem]{ulem}

\usepackage{hyperref}
\hypersetup{unicode,colorlinks=true,linkcolor=blue,citecolor=blue,pdfauthor={Bronsard-Merlet-Pegon},pdftitle={Non-ball and non-radial minimizers in the generalized liquid drop model with screened interactions}}

\usepackage[capitalize,nameinlink,noabbrev]{cleveref}
\usepackage{csquotes}
\usepackage[shortlabels]{enumitem}

\usepackage[top=3cm, bottom=3cm, inner=3cm, outer=3cm,headsep=0.75cm,footskip=1.25cm]{geometry}

\usepackage{microtype}

\usepackage{xcolor}

\usepackage[backend=biber,style=numeric,isbn=false,alldates=year,url=false,doi=false,eprint=false,giveninits=true,maxbibnames=50]{biblatex}
\renewbibmacro{in:}{} 
\AtEveryBibitem{\clearlist{location}} 
\AtEveryBibitem{\clearfield{urlyear}} 
\AtEveryBibitem{\clearfield{pagetotal}} 

\renewbibmacro{volume+number+eid}{%
    \textbf{
	\printfield{volume}%
    \setunit{\adddot}%
    \printfield{number}%
    \setunit{\addcomma\space}%
    \printfield{eid}}
}
\AtBeginBibliography{%

}

\DeclareFieldFormat
  [article,inproceedings,incollection,report,unpublished,online]
  {title}{#1\addcomma}
\DeclareFieldFormat
  [book,inbook,incollection,unpublished]
  {date}{(#1)}
\DeclareFieldFormat
  [book,inbook]
  {title}{\textit{#1}\addcomma}
\DeclareFieldFormat
  [article,inbook,book,incollection,inproceedings,report,unpublished,online]
  {pages}{pp. #1}

\DeclareDelimFormat[bib]{nametitledelim}{\space:\space}

\DeclareNameAlias{labelname}{given-family} 

\newcommand{\eps}{\varepsilon}

\newcommand{\N}{\mathbb{N}}
\newcommand{\R}{\mathbb{R}}

\newcommand{\Sp}{\mathbb{S}}
\newcommand{\cC}{\mc{C}}
\newcommand{\cE}{\mc{E}}
\newcommand{\cI}{\mc{I}}

\newcommand{\ind}{{\mathbf{1}}}

\newcommand{\be}{\begin{equation}}
\newcommand{\ee}{\end{equation}}

\newcommand{\lt}{\left}
\newcommand{\rt}{\right}
\newcommand{\mc}{\mathcal}
\newcommand{\oo}{\infty}

\newcommand{\sm}{\setminus}
\newcommand{\sub}{\subset}
\newcommand{\dw}{\downarrow}
\newcommand{\up}{\uparrow}

\newcommand{\ov}{\overline} 
\newcommand{\longto}{\longrightarrow}
\newcommand{\void}{\varnothing}

\newcommand{\te}{\theta}

\newcommand{\I}{\mc{I}}

\newcommand{\h}{\mc{H}}

\DeclareMathOperator*{\argmin}{arg\,min}
\DeclareMathOperator*{\argth}{tanh^{-1}}
\DeclareMathOperator*{\argsh}{sinh^{-1}}

\DeclareMathOperator{\cyl}{cyl}
\DeclareMathOperator{\ball}{ball}
\DeclareMathOperator{\shell}{shell}

\def\XXint#1#2#3{{\setbox0=\hbox{$#1{#2#3}{\int}$}
     \vcenter{\hbox{$#2#3$}}\kern-.5\wd0}}

\crefname{equation}{}{}
\crefname{figure}{Figure}{Figures}

\declaretheorem[name=Theorem]{mainthm}

\declaretheorem[name=Theorem, numbered=no]{extthm}
\declaretheorem[name=Proposition,numberlike=mainthm]{mainprp}
\declaretheorem[name=Conjecture,numberlike=mainthm]{mainconj}
\declaretheorem[name=Hypothesis,style=remark]{mainhypothesis}

\declaretheorem[name=Theorem,numberwithin=section]{thm}
\declaretheorem[name=Lemma,numberlike=thm]{lem}
\declaretheorem[name=Proposition,numberlike=thm]{prp}
\declaretheorem[name=Corollary,numberlike=thm]{cor}
\declaretheorem[name=Definition,numberlike=thm,style=definition]{dfn}
\declaretheorem[name=Remark,numberlike=thm,style=remark]{rmk}

\crefname{equation}{}{}
\crefname{enumi}{}{}
\crefname{thm}{Theorem}{Theorems}
\crefname{mainthm}{Theorem}{Theorems}
\crefname{mainprp}{Proposition}{Propositions}
\crefname{lem}{Lemma}{Lemmas}
\crefname{prp}{Proposition}{Propositions}
\crefname{cor}{Corollary}{Corollaries}
\crefname{dfn}{Definition}{Definitions}
\crefname{rmk}{Remark}{Remarks}
\crefname{conj}{Conjecture}{Conjectures}
\crefname{mainconj}{Conjecture}{Conjectures}
\crefname{mainhypothesis}{Hypothesis}{Hypotheses}
\crefname{ex}{Example}{Examples}

\title[Symmetry breaking in the liquid drop model with screened interactions]{Symmetry breaking in the liquid drop model with screened interactions}
\author{L. Bronsard}
\address{McMaster University, Canada}
\email{bronsard@mcmaster.ca}
\author{B. Merlet}
\author{M. Pegon}
\address{Univ. Lille, CNRS, UMR 8524, Inria - Laboratoire Paul Painlev\'e, F-59000 Lille}
\email{benoit.merlet@univ-lille.fr}
\email{marc.pegon@univ-lille.fr}
\date{}

\begin{document}


\begin{abstract}
We investigate liquid drop models with screened Riesz-type interactions, focusing in
particular on truncated Coulomb and Yukawa potentials in three dimensions.
While in the classical Gamow model with Coulomb interaction the only minimizers are
balls, as recently proved by Chodosh and Gianocca (2026), we show that this is no longer true if the
interaction is screened. In the case of truncated Coulomb and Yukawa potentials, we prove that for
some values of the parameter, there exist connected minimizers which are not balls.
For truncated Coulomb potentials, we further obtain genuine symmetry breaking by showing existence
of non-radial connected minimizers.
This gives the first evidence of minimizers which are not balls in three dimensions, and the first
evidence of non-radial minimizers altogether, in the class of radial kernels.
Our approach relies on a comparison of the energy-per-mass ratios of balls, core-shells and cylinders.
\end{abstract}

\maketitle

\tableofcontents

\section*{Introduction}
In this paper, we are interested in generalizations of Gamow's liquid drop model for the atomic
nucleus:
\be\label{minpb}
e_G(m) = \inf~ \lt \{ P(E)+\int_{E\times E} G(x-y)\,dx\,dy ~:~ E\sub\R^n \text{ s.t. }
|E|=m\rt\},
\ee
where $|E|$ denotes the Lebesgue measure (or volume, or \textsl{mass}) of a measurable set $E\sub\R^n$, and where $P(E)$ is its perimeter. We assume $n\ge2$ and that $G$ is a radial, nonnegative, locally integrable kernel.
We introduce the functional $\cE_G$ defined by
\[
\cE_G(E) = P(E)+\int_{E\times E} G(x-y)\,dx\,dy.
\]

Gamow's liquid drop model refers to the specific case $n=3$ and the Coulomb interaction
$G(x)=1/|x|$, while the terminology \textsl{generalized} liquid drop model is commonly used for
arbitrary dimensions $n\ge 2$ and Riesz kernels 
\[
G(x)=R_\alpha(x)=|x|^{-\alpha},
\]
where $\alpha\in(0,n)$. Here we are particularly interested in further generalizations where the Riesz interaction
is screened:
\begin{enumerate}[1.]
\item either abruptly, by considering truncations of Riesz kernels such as 
\[
R_{\alpha,\kappa}(x) =
\ind_{|x|<\kappa}|x|^{-\alpha}
,\]
\item or smoothly, by considering the kernels 
\[
Y_{\alpha,\kappa}(x) =
e^{-|x|/\kappa}|x|^{-\alpha}.
\] 
%
%
\end{enumerate}

In both cases, the characteristic range of the interaction is $\kappa$. In the latter case, for $n=3$ and $\alpha=1$ we obtain the physically-relevant potential known as
\textsl{Yukawa potential} $Y_{1,\kappa}(x)= e^{-|x|/\kappa}/|x|$, also called \textsl{screened
Coulomb potential}. It has been introduced early in quantum mechanics to model some particle
interactions, first  between neutrons and protons, see the seminal article~\cite{Yuk1935},
republished in~\cite{Yuk1955}. In the mathematical community, such potential has also been suggested
in~\cite{KMN2016} to model diblock copolymer melts and the energy $\cE_G$ with a Yukawa potential is
studied in~\cite{Fal2018} as a model generating periodic patterns.\medskip

The aim of this paper is to investigate the effect of the screening in terms of existence and
shape of minimizers for~\cref{minpb}.\medskip

Let us first present a selection of results of the literature in the absence of screening, that is, for the liquid drop model with the Riesz interaction $G(x)=R_\alpha(x)=|x|^{-\alpha}$, $0<\alpha<n$.

\begin{itemize}[leftmargin=*]
\item In any dimension $n\ge2$, for any parameter $\alpha\in(0,n)$, there exists a critical volume (or
\textsl{mass}) $m_{\ball}$ below which the balls of mass $m$ are the minimizers of~\cref{minpb}. We
refer to~\cite{KM2013} for the planar case, to~\cite{KM2014,Jul2014,BC2014} for many other cases and
to~\cite{FFM+2015} for the full range of $\alpha$ and any dimension. See also \cite{FL2015} for a
general criterion of existence which applies to the generalized liquid drop model.
\item In any dimension  $n\ge2$, for any parameter $\alpha\in (0,2)$ (and even $\alpha\in(0,2]$ for $n\ge
3$), there exists a critical mass $m_\mathrm{ex}\ge m_{\ball}$ for which minimizers exist for masses
$m\leq m_\mathrm{ex}$ and no minimizer exists masses for larger than $m_\mathrm{ex}$. We refer to~\cite{KM2013} for the planar case,
\cite{LO2014} for the Coulomb case, \cite{KM2014} for the cases $\alpha\in(0,2)$ in dimension $n\geq
3$, and \cite{FN2021} for the proof in the case $\alpha=2$, which is inspired from \cite{FKN2016}.
\item\label{point:ex_geq_crit} It is established in~\cite{FN2021} that in any dimension $n\geq 2$, $m_\mathrm{ex}\ge
m_\mathrm{crit}$ where  $m_\mathrm{crit}$ is the critical mass for which the energy of one ball of mass $m$ is twice the energy of a ball of mass $m/2$. 
\item In dimension $n=2$, there exists a critical parameter $\alpha_*>0$ such that for $0<\alpha<\alpha_*$, $m_{\ball}=m_\mathrm{ex}=m_\mathrm{crit}$, that is, the ball is the unique minimizer for
masses not larger than $m_\mathrm{crit}$, and there is no minimizer for larger masses. This is
given in~\cite[Theorem~2.7]{KM2013}.
\item In the Coulomb case $n=3$ and $\alpha=1$, \cite{CR2025} established the first
explicit lower bound $m_{\ball}\geq 1$.
\end{itemize}

These results support the following statement, which has long been conjectured in the literature
(see \cite{CP2011,FL2015,CR2025}), and which has been established only very recently
in~\cite{CG2026}:
\begin{extthm}[Chodosh--Gianocca \cite{CG2026}]
If $n=3$ and $\alpha=1$ (the case of the Coulomb interaction),
$m_{\ball}=m_\mathrm{ex}=m_\mathrm{crit}$, that is, the ball of mass $m$ is the unique minimizer (up
to translations) for masses not larger than $m_{\mathrm{crit}}$, and there is no minimizer for
masses larger than $m_{\mathrm{crit}}$.
\end{extthm}

Before~\cite{CG2026}, minimality of balls in the Coulomb case was known only in the perturbative
small-mass regime, quantitatively up to the bound of~\cite{CR2025}.

The corresponding statement is expected to hold in arbitrary dimension and for any $\alpha\in(0,n)$
(see e.g. the introduction of~\cite{FN2021}), and remains open apart from the cases mentioned above:
\begin{mainconj}\label{conj:gamow2}
Let $n\ge 2$ and $0<\alpha<n$, then $m_{\ball}=m_\mathrm{ex}=m_\mathrm{crit}$.
\end{mainconj}
In this work, we prove that while \cref{conj:gamow2} is expected to remain true for all Riesz
kernels, its analogue fails for kernels that decay faster at infinity, in particular for some
truncated Coulomb potentials and some Yukawa potentials.


\begin{mainthm}[Non-ball minimizers]\label{mainthm:thm1}
Let $n=3$.
Considering Problem~\cref{minpb} with either the truncated Coulomb potential $G=R_{1,\kappa}$
with $\kappa=11/10$ or the Yukawa potential $G=Y_{1,\kappa}$ with $\kappa=28/5$, there exists a mass
$m>0$ such that the minimization problem admits a connected minimizer which is not a ball.
Furthermore, within the class of sets of mass $m$, the global minimizers are neither balls nor any
(finite or countable) disjoint union of balls.\footnote{In the truncated Coulomb case, kernels are compactly supported and thus disconnected minimizers are possible.}
\end{mainthm}

The above result covers both potentials. In the truncated Coulomb case, we can go further and rule
out radial minimizers altogether with the same $\kappa=11/10$.

\begin{mainthm}[Non-radial minimizers]\label{mainthm:thm2}
Let $n=3$ and consider Problem~\cref{minpb} with the truncated Coulomb potential $G=R_{1,\kappa}$
with $\kappa=11/10$. There exists a mass $m>0$ such that the minimization problem admits a
connected minimizer which is neither a ball nor a core-shell $B_R\sm\ov B_r$, $0<r<R$. Furthermore,
within the class of sets of mass $m$, the global minimizers are neither balls, core-shells, nor any
(finite or countable) disjoint union of such sets. In particular, they are not radial.
\end{mainthm}

To our knowledge, these are the first non-radial minimizers in the generalized liquid drop model for
a radial and radially nonincreasing kernel, and in fact for a radial kernel of any kind (see the
comments on \cite{MS2019} below).

We expect the same phenomenon to occur for Yukawa potentials, but the computations seem untractable
(see \cref{rmk:yukawa}), so we state it as a conjecture.

\begin{mainconj}\label{conj:yukawa}
There exists $\kappa>0$ such that the conclusion of \cref{mainthm:thm2} holds with the Yukawa
potential $G=Y_{1,\kappa}$ in place of the truncated Coulomb potential.
\end{mainconj}

In dimension $n=2$ the existence of non-ball minimizers for kernels $Y_{1,\kappa}$
follows from the recent work~\cite{MNS2025} (see Corollary~2.5 therein).
The authors consider a precise asymptotic regime of $m_j\to+\oo$ and
$\kappa_j\uparrow1/\sqrt{2}$ for which the (properly rescaled) energy $\Gamma$-converges to a convex
combination of the perimeter and of the elastica functional.
If the elastica functional dominates the perimeter term, minimizers are annuli of the form
$B_R(x)\sm\ov{ B_r}(y)$. As a consequence there exists $\ov\kappa<1/\sqrt{2}$ such that for
$\ov\kappa<\kappa<1/\sqrt{2}$ there is some large mass $m(\kappa)$ such that minimizers
of~\cref{minpb} with $n=2$ and $G=Y_{1,\kappa}$ exist but  are not disks. Moreover, their work
suggests that these minimizers are annuli.
Our results differ from theirs in two ways. First, we consider the
three-dimensional case and potentials that are screened analogues of the standard
Coulomb interaction. Second, and more importantly, the minimizers produced in~\cite{MNS2025} are
expected to be radial, so that no symmetry is broken, whereas \cref{mainthm:thm2} produces
minimizers which are not radial.

Let us also mention that the effect of screening on the general question of existence of minimizers
in the liquid drop model has already been studied by the second and third authors in
\cite{Peg2021,MP2021,GMP2024}. Therein it is proven that if the interaction kernel $G$ is locally
integrable and if its first moment, that is, $\int_{\R^n} |x|G(x)\,dx$, is below an explicit critical
threshold, then for arbitrarily large masses the ball is the unique minimizer of \cref{minpb}, up to
translation.  This is in sharp contrast with the case of Riesz kernels $G=R_\alpha$, where existence
is expected to fail for large masses.
In the Yukawa case in dimension $n=3$, these results have recently been refined
in~\cite{BDT2026}.
Minimizers are shown to exist at every mass above an explicit critical screening length, whereas the
mass was previously required to be large, and at every screening rate below an explicit mass
threshold.
Note that existence of minimizers for all masses can be recovered for Riesz kernels $G=R_\alpha$ by
introducing an additional confining term, such as $\int_{E} |x|^\beta\,dx$ with $\beta>\alpha$
\cite{GO2018}, or $\int_{E} -V(x)\,dx$ with $V(x)\gg |x|^{-\alpha}$ at infinity \cite{ABCT2017,ABCT2020}.

Finally, let us also point out that minimizers which are not disks have been
observed in a related problem in \cite[Theorem~2.8]{MS2019}, in dimension $n=2$ and for a radial but
non-monotonic kernel. The comparison made therein is with collections of disks only, so
that radial minimizers such as annuli are not excluded.
Non-radial, axially-symmetric \textsl{stable} critical points in dimension $n=3$ have been described in
\cite{Fra2019}, but these were not minimizers.
Note that non-radial minimizers naturally appear when anisotropy is introduced in the energy. For
instance, \cite{CNT2022} focuses on anisotropic perimeters, while \cite{KS2022} considers the
repulsive non-local energy $\|\partial_1 \ind_E\|_{H^{-1}}^2$ which penalizes variations only in one
direction.
\medskip

Let us sketch the proof of \cref{mainthm:thm1}. The first ingredient is the notion of generalized minimizers. Given a kernel $G$ and a mass $m>0$, we relax the problem~\cref{minpb} and consider
\be
\label{genminpb}
\hat e_G(m) = \inf~ \lt \{\sum_{i\in \N} \cE_G(E^i) ~:~ E^i\sub\R^n \text{ s.t. }
\sum_i|E^i|=m\rt\}.
\ee
We will see that in the settings we consider, minimizers for~\cref{genminpb} always exist and that
we can assume that they are sequences $(E^i)_{i\in\N}$ of connected sets. Observe that
for each  $i\in\N$, $E^i$ is a minimizer of $\cE_G$ among sets of mass $m_i\coloneqq|E^i|$. 

We now reason by contradiction.
\begin{mainhypothesis}\label{Intro_hyp}
For any mass $m>0$:
\begin{itemize}[leftmargin=*]
\item either~\cref{minpb} admits no minimizer,
\item or \cref{minpb} admits a ball or a (finite or countable) disjoint union of balls as minimizer.\footnote{When $G$ is compactly supported, disjoint unions of balls are possible minimizers.}
\end{itemize}
\end{mainhypothesis}

Assuming \cref{Intro_hyp}, we see that for any $m>0$, 
\[
\hat e_G(m) = \inf  \lt \{\sum_i \cE_G(B_{r_i}) ~:~ (B_{r_i})_i \text{ sequence of balls with }
\sum_i|B_{r_i}|=m\rt\}.
\]
Considering the energy per unit of mass, it follows (see \cref{prp_genminball}) that for $m>0$
\be\label{intro_def_eGball}
\rho^{\ball}_G\coloneqq \inf \lt\{ \frac{\cE_G(B_R)}{|B_R|} ~:~ B_R\sub \R^n \rt\}\le\dfrac{\hat e_G(m)}m.
\ee
By definition  $\rho^{\ball}_G$ is the best energy-per-mass ratio for balls.

We now try to contradict~\cref{intro_def_eGball} by considering long cylinders as candidates for the energy-per-mass ratio. Namely, denoting  $C_{l,L}\coloneqq B_l^{n-1}\times (0,L)$ we set 
\[
\rho^{\cyl}_G\coloneqq \inf_{l>0} \lim_{L\up\oo} \dfrac{\cE_G(C_{l,L})}{|C_{l,L}|}.
\]
Using $(E^1,E^2,\dots)=(C_{l,L},\void,\void,\dots)$ in~\cref{genminpb}  we see that $\hat e_G(|C_{l,L}|)\le \cE_G(C_{l,L})$ and we deduce
\be\label{intro_contradiction_step_2}
\inf_{m>0} \dfrac{\hat e_G(m)}m \le \inf_{l,L>0} \dfrac{\cE_G(C_{l,L})}{|C_{l,L}|} \le \rho^{\cyl}_G.
\ee
The key is now the following result.
\begin{mainprp}\label{mainprp_prp1}
Let $n=3$ and $G$ be either $R_{1,\kappa}$ with $\kappa=11/10$ or $Y_{1,\kappa}$ with $\kappa=28/5$.
Then we have
\[
\rho^{\cyl}_G < \rho_G^{\ball}.
\]
\end{mainprp}
\cref{mainprp_prp1} and~\cref{intro_contradiction_step_2} imply
\[
\inf_{m>0} \dfrac{\hat e_G(m)}m  < \rho_G^{\ball},
\]
which contradicts~\cref{intro_def_eGball}.\\
It follows that \cref{Intro_hyp} were wrong and there
exists some mass $m>0$ for which~\cref{minpb} admits a minimizer $E$ and for which neither a ball
nor any finite or countable disjoint union of balls is a minimizer.
To complete the proof of \cref{mainthm:thm1}, one must still address the question of connectedness.
This is actually a minor technical point, and we refer the reader to the final step of the proof of
\cref{prp_ecyleball_nonsphere} for the concluding argument.

The proof of \cref{mainthm:thm2} follows the same strategy: considering instead the
infimum energy-per-mass ratio of core-shells
\[
\rho^{3,\shell}_G
\coloneqq \inf \lt\{ \frac{\cE_G(B_R\setminus B_r)}{|B_R\setminus B_r|} ~:~ B_r
\subset B_R\subset \R^3 \text{ with }0\leq r <R \rt\}
\]
we prove
\begin{mainprp}\label{mainprp_prp2}
Let $n=3$ and $\kappa=11/10$. Then we have
\[
\rho^{\cyl}_{R_{1,\kappa}} < \rho_{R_{1,\kappa}}^{\shell} = \rho_{R_{1,\kappa}}^{\ball}.
\]
\end{mainprp}

The equality in the previous proposition follows from the precise computation of
$\rho_{R_{1,\kappa}}^{\shell}$, which turns out to be attained by balls for $\kappa=11/10$.
It might well happen that $\rho_{R_{1,\kappa}}^{\shell}<\rho^{\ball}_{R_{1,\kappa}}$ for other
values of this parameter.

We are left to establish \cref{mainprp_prp1,mainprp_prp2} and thus to estimate the quantities
$\rho^{\cyl}_G$, $\rho_G^{\ball}$ and $\rho_{R_{1,\kappa}}^{\shell}$, where $G=R_{1,\kappa}$ or
$Y_{1,\kappa}$.
For this we use a slicing argument to rewrite the self-interaction of a \textsl{convex} set
$E\sub\R^n$ as an integral over all one-dimensional slices of $E$ (see \cref{prp_1Dslices}), where the integrand depends only
on the length of each slice:
\[
\int_{E\times E} G(x-y)\,dx\,dy
= \int_{\Sp^{n-1}} \int_{\sigma^\perp}
S_G\lt(\h^1(E\cap [y+\R\sigma])\rt)\,dy\,d\h^{n-1}(\sigma).
\]
Here $\h^d$ is the $d$-dimensional Hausdorff measure in $\R^n$, and $S_G$ is a measurable nonnegative function which depends
only on $G$ and $n$. While elementary, the computations are quite tedious. They lead to reasonably
simple closed formulas for $\rho^{\cyl}_G$, $\rho_G^{\ball}$ and $\rho_{R_{1,\kappa}}^{\shell}$.
For the Yukawa case, we use a symbolic mathematics library to evaluate them with the precision
needed for comparison. We refer the reader to the interactive Jupyter Notebook provided as
supplementary material.
\medskip

By contrast with \cref{mainprp_prp1}, we show the opposite inequality $ \rho^{\cyl}_G >
\rho_G^{\ball}$  in the Riesz case $G=R_\alpha$ in arbitrary dimension $n\ge 3$ and for integer
parameters $1<\alpha<n$.
\begin{mainprp}\label{mainprp_prp3}
Let $n\ge 3$ and $\alpha\in (1,n)$ be an integer. Then we have
\[
\rho_{R_\alpha}^{\ball}<\rho_{R_\alpha}^{\cyl}.
\]
\end{mainprp}
This is consistent with~\cref{conj:gamow2}. However, we observe that the ratio
\[
\frac{\rho_{R_\alpha}^{\cyl}}{\rho_{R_\alpha}^{\ball}}>1
\]
is surprisingly close to $1$, even in low dimensions, see \cref{rmk_closeto1}, which illustrates how
delicate the conjecture is.

\vspace{5mm}

\subsection*{Outline of the paper}

The structure of the paper is as follows. In~\cref{S_gamow}, we recall a few facts concerning
generalized minimizers of Gamow-type problems and prove
\cref{prp_ecyleball_nonsphere,prp_ecyleshell_nonshell}, which
explain how \cref{mainthm:thm1,mainthm:thm2} are implied respectively by \cref{mainprp_prp1,mainprp_prp2}.
In~\cref{Ss_slicing} we establish a slicing formula to express the self-interaction energy of a set $E$ in terms of
the self-interaction energies of its one dimensional slices.
The three remaining sections are successively devoted to the cases of the Riesz interactions, the
truncated Coulomb interaction, and the Yukawa interaction. More precisely, we prove
\cref{mainprp_prp3} in \cref{S_riesz}, the truncated Coulomb and the Yukawa cases of
\cref{mainprp_prp1} respectively in \cref{subsec:truncrieszballvscyl} and
\cref{subsec:yukawaballvscyl}, and \cref{mainprp_prp2} in \cref{subsec:nonradial}.

\subsection*{Notation}

\begin{itemize}[leftmargin=*]
\item We denote by $\h^k$ the $k$-dimensional Hausdorff measure.
\item Given a measurable set $E\sub\R^n$ we write $|E|=\mc{L}^n(E)$ to denote its volume (or mass) and $P(E)$ to denote its perimeter.
\item We denote by $B_r^n$ the $n$-dimensional ball of radius $r$ centered at $0$, or simply $B_r$ when
there is no ambiguity on the dimension. We recall the formulas 
\[
|B_1^n|=
\frac{\pi^{n/2}}{\Gamma\lt(1+\frac{n}{2}\rt)}.
\]
\item By an abuse of notation we let $|\Sp^k|=\h^k(\Sp^k)$ where $\Sp^k\sub\R^{k+1}$ is the unit sphere of dimension $k$. Recall that we have $|\Sp^{n-1}|=P(B_1^n) =n|B_1^n|$.
\item We denote by $C_{l,L}^n$ the ball-based cylinder $C_{l,L}^n\coloneqq B_l^{n-1}\times (0,L)$ (or simply
$C_{l,L}$ when there is no ambiguity).
\item For a generic kernel $G$, we set
\[
\I_G^n(E) \coloneqq \int_{E\times E} G(x-y)\,dx\,dy,\qquad\text{ for every measurable set } E.
\]
With this notation the total energy writes
\[
\cE^n_G=P+\I_G^n.
\]
\item We often drop the superscript $n$ when there is no risk for confusion.
\item For radial kernels we let $G(|x|)\coloneqq G(x)$, by an abuse of notation.
\item The main part of the article consists in the comparison of the following quantities:
\begin{equation}\label{eq:defrhoball}
\rho^{n,\ball}_G\coloneqq \inf \lt\{ \frac{\cE_G(B_R)}{|B_R|} ~:~ B_R\sub \R^n \rt\},
\end{equation}
\begin{equation}\label{eq:defrhoshell}
\rho^{n,\shell}_G\coloneqq \inf \lt\{ \frac{\cE_G(B_R\setminus B_r)}{|B_R\setminus B_r|} ~:~ B_r
\subset B_R\subset \R^n \text{ with }0\leq r <R \rt\}
\end{equation}
and
\begin{equation}\label{eq:defrhosigmacyl}
\rho^{n,\cyl}_G\coloneqq \inf_{l>0}\sigma_G^{n,\cyl}(l)\qquad\text{where}\qquad\sigma_G^{n,\cyl}(l)\coloneqq  \lim_{L\up\oo} \dfrac{\cE_G(C_{l,L})}{|C_{l,L}|}.
\end{equation}
\item In \cref{S_riesz,,S_truncated_riesz,,S_yukawa} we consider respectively the kernels
\[
R_\alpha(x)=|x|^{-\alpha},\qquad R_{\alpha,\kappa}(x)=\ind_{|x|<\kappa}|x|^{-\alpha}\quad\text{and}\quad Y_{\alpha,\kappa}(x) =
e^{-|x|/\kappa}|x|^{-\alpha}.
\]  
\item In~\cref{S_riesz,S_truncated_riesz} we use the notation 
\[
\beta\coloneq n+1-\alpha,
\]
where $n$ is the ambient dimension and where $0<\alpha<n$ is the parameter appearing in the potentials $ \I^n_{R_\alpha}$, $\I_{{R}_{\alpha,\kappa}}^n$.
\end{itemize}

\section{Basic facts on generalizations of Gamow-type problems}\label{S_gamow}

In this section, unless stated otherwise, $G$ is a nonnegative, radial, locally integrable kernel
on~$\R^n$.

\subsection{Minimizers and generalized minimizers}

First of all, if $G$ is compactly supported, in particular if $G=R_{\alpha,\kappa}$ then the problem
always admits at least one minimizer, see~\cite[Chapter~5]{Rig2000a}. More precisely, we have the
following.

\begin{prp}
Let $G$ be a nonnegative integrable kernel with compact support in $\R^n$ (not necessarily
radial). Then for every $m>0$, Problem~\cref{minpb} admits a minimizer.
In addition, every minimizer has a finite number of components with finite perimeter $E^i$,
$i=1,\dotsc,N$, which are connected in the measure theoretic sense (as introduced
in~\cite{ACMM2001}). 
Finally, each connected component $E^i$ is itself a minimizer of \cref{minpb}
within the class of measurable sets of mass $|E^i|$.
\end{prp}

For many other kernels, such as $R_\alpha$ or $Y_{\alpha,\kappa}$, there is always existence of
minimizers in a generalized sense.
While various equivalent notions of generalized minimizers exist, here we adopt the
one in~\cite{GNR2024}, which is similar to~\cite[Definition~4.3]{KMN2016} except it does not require
to have finitely many components.

\begin{dfn}[Generalized minimizers]\label{def:genmin}~
\begin{enumerate}[(1),leftmargin=*,widest=1]
\item We call \textsl{generalized set of mass $m$} a finite or countable family of measurable sets
$E^i\sub\R^n$, $i\ge1$ such that $\sum_i |E^i|=m$.
\item We extend the definition of $\cE_G$ to generalized sets by  
\[
\cE_G((E^i))\coloneqq \sum_i \cE_G(E^i).
\]
\item As in~\cref{genminpb} we denote
\[
\hat e_G(m) \coloneq \inf~ \lt \{\cE_G((E^i)) ~:~ (E^i) \text{ generalized set of mass }m\rt\}.
\]
\item A generalized set $(E^i)$ of mass $m$ is a generalized minimizer of Problem~\cref{minpb} if it minimizes $\cE_G$ among generalized sets of mass $m$.
\end{enumerate}
\end{dfn}

The existence of generalized minimizers for these kernels is well-established. Here, we simply state
the existence result in a form that matches our notation and link it to the existing literature.

\begin{prp}\label{prp_existmingen}
Let $G$ be either $R_\alpha$, $R_{\alpha,\kappa}$ or $Y_{\alpha,\kappa}$. Then for any mass $m$:
\begin{enumerate}[(i),leftmargin=*,widest=1]
\item Problem~\cref{minpb} admits a generalized minimizer $E=(E^i)_{i\in I}$. Besides, $I$ is a finite set 
and each component $E^i$ is a classical minimizer of \cref{minpb} with mass $|E^i|$, that
is,
\[
\cE_G(E^i) = \min \lt\{\cE(F):F\sub\R^n,\ |F|=|E^i|\rt\}.
\]
\item In addition, if $G$ is either $R_\alpha$ or $Y_{\alpha,\kappa}$, then any component $E^i$ is connected (in the
measure theoretic sense of~\cite{ACMM2001}).
\end{enumerate}
\end{prp}

\begin{proof}~\\
\textit{We prove (i).} In \cite{NP2021}, existence of generalized minimizers is obtained for a wide class
of kernels which includes $R_\alpha$, $R_{\alpha,\kappa}$ and $Y_{\alpha,\kappa}$. However, the
definition of generalized minimizers differs slightly from the one in the present paper. Precisely,
\cite[Proposition~2.1]{NP2021} implies that for $G$ being either $R_\alpha$, $R_{\alpha,\kappa}$ or
$Y_{\alpha,\kappa}$, there exists $E=\sqcup_{i=1}^H E^i$ a minimizer of
\begin{equation}\label{defgenminNP2021}
\tilde{e}_G(m) = \inf \left\{ \sum_{i=1}^H \mathcal{E}_G(E^i) ~:~ H\in\N^*,~
E=\bigsqcup_{i=1}^H E^i~\text{ s.t. }~ \sum_{i=1}^H |E^i|=m \right\}.
\end{equation}
Here $\sqcup_{i=1}^H E^i$ indicates that the sets $E^i$ are pairwise disjoint in the
measure-theoretic sense, that is, $|E^i\cap E^j|=0$ for $i\ne j\in\{1,\dotsc,H\}$.  It is
established in \cite[Proposition~2.1]{NP2021} that for $(E^i)_{1\le i\le H}$ such a generalized
minimizer, each $E^i$ is a classical minimizer of Problem~\cref{minpb} in the class of sets with
mass~$|E^i|$.

By contrast, in \cref{def:genmin}, the collection $(E^i)$ may be countable and each $E^i$ is treated as existing in an independent copy of $\R^n$.
We obviously have  $\hat e_G(m)\le \tilde{e}_G(m)$ and if $\hat e_G(m)= \tilde{e}_G(m)$, then any generalized minimizers in the sense of
\cref{defgenminNP2021} is a generalized minimizer in the sense of \cref{genminpb}.
Therefore, in view of the results of~\cite{NP2021} there just remains to establish that $\tilde e_G(m)\le \hat{e}_G(m)$.\\
This inequality follows by approximating $\hat{e}_G(m)$ with a finite sub-collection of sets $E^j$,
intersecting each with a large ball, and translating them sufficiently far apart so that their
mutual interaction becomes arbitrarily small. Details are left to the reader.

\noindent
\textit{We prove (ii).}
It remains to justify the connectedness of the components for $G = R_\alpha$ and $G
= Y_{\alpha,\kappa}$. Suppose, by contradiction, that a component $E^j$ of a generalized minimizer
$(E^i)$ is not indecomposable (i.e., not connected in the measure-theoretic sense). Since these
kernels are strictly positive, the interaction energy between any two disjoint subsets of $E^j$ is
strictly positive. Consequently, by placing these subsets into independent copies of $\mathbb{R}^n$,
one would strictly decrease the total energy while preserving the total mass. This contradicts the
optimality of $(E^i)$, allowing us to conclude that all components of a generalized minimizer must
be connected.

\end{proof}


When $G$ is compactly supported, the problem always admits minimizers by
~\cite[Chapter~5]{Rig2000a}, and a possible generalized minimizer is just $(E,\void,\void,\dots)$
where $E$ is a classical minimizer. If $E$ is not connected, another generalized minimizer is
$(E^1,\cdots,E^k)$ where the $E^i$'s are the connected components of $E$. We favor the latter in
order to work  within a unified framework whether the kernel is compactly supported (as the
$R_{\alpha,\kappa}$'s) or not (as the $Y_{\alpha,\kappa}$'s).

\begin{dfn}\label{dfn_compactcase}
When the kernel  $G$ is compactly supported, we modify the definition of generalized minimizers by
requiring  that each component of a generalized minimizer is connected in the measure theoretic sense
of~\cite{ACMM2001}. That way, the last conclusion of \cref{prp_existmingen} also holds for $G=R_{\alpha,\kappa}$.
\end{dfn}

In the rest of this section, we drop the $n$ superscript and simply write $\rho_G^{\ball}$,
$\rho_G^{\shell}$, $\rho_G^{\cyl}$ and $\sigma_G^{\cyl}(l)$ for the quantities defined 
in \cref{eq:defrhoball}, \cref{eq:defrhoshell} and \cref{eq:defrhosigmacyl}.


\begin{prp}\label{prp_genminball}
If $E=(B_{r_i})_{i\in I}$ is a generalized set of mass $m$ made of a collection of balls, then
\[
\frac{\cE_G(E)}{m} \ge \rho_G^{\ball}.
\]
\end{prp}

\begin{proof}
The proof is elementary. By definition of $\rho_G^{\ball}$ and of the energy of generalized
minimizers, we have
\[
\cE_G(E)
= \sum_{i\in I} \frac{\cE_G(B_{r_i})}{|B_{r_i}|}|B_{r_i}|
\ge \rho_G^{\ball}\sum_{i\in I} |B_{r_i}|
=m \rho_G^{\ball}.
\]
\end{proof}

\begin{rmk}
Notice that in fact
\be\label{rmkgenmin:asymp}
\lim_{m\to\oo} \inf_E \,\frac{\cE_G(E)}{m} = \rho_G^{\ball},
\ee
where the infimum is taken over generalized sets $E$ of mass $m$ and made exclusively of balls.
Indeed, let $\eps>0$ and let $B_r$ such that $\cE_G(B_r)/|B_r|\le \rho_G^{\ball}+\eps$. For $m>0$, we consider the generalized set $E$ of mass $m$ made of $N(m) =\lt\lfloor\frac{m}{|B_{r}|}\rt\rfloor$ balls of radius $r$ and one ball of radius $s<r$ so that $|B_s|=m-N(m)|B_r|$. We have 
\[
\dfrac{\cE_G(E)}{m}
= \dfrac{N(m)\cE_G(B_r) + \cE_G(B_s)}{m}  
\le \dfrac{N(m) + 1}{m} \mc \cE_G(B_r) 
=   \dfrac{N(m) + 1}{m/|B_r|}(\rho_G^{\ball}+\eps).
\]
Sending $m$ to $\oo$ and then $\eps$ to $0$ we obtain
\[
\lim_{m\to\oo} \inf_E \,\frac{\cE_G(E)}{m} \le \rho_G^{\ball}.
\] 
As the converse inequality follows from \cref{prp_genminball}, we conclude  that~\cref{rmkgenmin:asymp} holds true.
\end{rmk}

Let us now consider the energy of cylinders. 
\begin{lem}\label{lem_cylinf}
There holds for $l>0$,
\be\label{def_sigma_G^cyl_2}
 \sigma_G^{\cyl}(l) = \dfrac{n-1}{l} + \dfrac1{|B_l^{n-1}|} \int_{B_l^{n-1}\times B_l^{n-1}}\int_{\R}  G(x'-y',t) \,\,dt\,dx'\,dy'.
\ee
\end{lem}

\begin{proof}
Recall the definition~\cref{eq:defrhosigmacyl} of $ \sigma_G^{\cyl}(l)$. First, we have $P(C_{l,L}) = L l^{n-2} |\Sp^{n-2}| + 2 |B_l^{n-1}|$. Dividing by $|C_{l,L}|$ and using the identities  $|C_{l,L}|=|B_l^{n-1}|\,L$ and $|B^{n-1}|=|\Sp^{n-2}|/(n-1)$, we get 
\be\label{proof_lem_cylinf_1}
\dfrac{P(C_{l,L})}{|C_{l,L}|} = \lt(\dfrac{n-1}{l}+ \frac{2}{L}\rt)\ \longto\ \dfrac{n-1}{l}\quad\text{as }L\up\oo.
\ee
Then we compute
\begin{align*}
\dfrac{\cI_G(C_{l,L})}{|C_{l,L}|}  
&=\dfrac1{|B_l|^{n-1} L} \int_{(0, L)^2}
\int_{[B_l^{n-1}]^2} G(x'-y',r-s) \,dx'\,dy'\,dr\, ds\\
&=\dfrac1{|B_l|^{n-1}} \dfrac1L\int_0^L \lt[
\int_{[B_l^{n-1}]^2} \lt(\int_0^L G(x'-y',r-s)\,dr\rt) \,dx'\,dy' \rt] \,ds.
\end{align*}
We perform the change of variable $r=s+t$ in the inner integral and $s=\xi L$ in the outer integral to get 
\begin{align*}
\dfrac{\cI_G(C_{l,L})}{|C_{l,L}|}  
&=\dfrac1{|B_l^{n-1}|} \int_0^1 \lt[
\int_{[B_l^{n-1}]^2} \lt(\int_{-\xi L}^{(1-\xi)L} G(x'-y',t)\,dt\rt) \,dx'\,dy' \rt] \,d\xi.
\end{align*}
Since $G\ge0$, we get by the monotone convergence theorem,
\[
\lim_{L\up\oo} \dfrac{\cI_G(C_{l,L})}{|C_{l,L}|}= \dfrac1{|B_l|^{n-1}}\int_{[B_l^{n-1}]^2} \lt(\int_\R G(x'-y',r)\,dt\rt) \,dx'\,dy'.
\]
Together with~\cref{proof_lem_cylinf_1}, this proves~\cref{def_sigma_G^cyl_2}.
\end{proof}

\begin{rmk}
Note that the quantity $\sigma_G^{\cyl}(l)$ is infinite for $G=R_\alpha$ with $\alpha\le 1$. Indeed, 
\[
\int_{\R}  \dfrac1{|(x'-y',t)|^\alpha} \,dt = +\oo.
\]
In particular, $\sigma_G^{\cyl}(l)=+\oo$ for the Coulomb potential ($G=R_1$) in $\R^3$.
\end{rmk}

\begin{prp}\label{prp_ecyleball_nonsphere}
Assume that $G$ is such that Problem~\cref{minpb} admits a generalized minimizer for every mass. We have that if
\[
\rho_G^{\cyl}<\rho_G^{\ball}
\]
then there exists a mass $m>0$ for which Problem~\cref{minpb} admits a connected minimizer and for
which neither balls nor any (finite or countable) disjoint union of balls are minimizing in the class of sets of mass $m$.
\end{prp}

Similarly, we have:

\begin{prp}\label{prp_ecyleshell_nonshell}
Assume that $G$ is such that Problem~\cref{minpb} admits a generalized minimizer for every mass. We have that if
\[
\rho_G^{\cyl}<\rho_G^{\shell}
\]
then there exists a mass $m>0$ for which Problem~\cref{minpb} admits a connected minimizer and for
which neither balls, core-shells nor any (finite or countable) disjoint union of such sets are
minimizing in the class of sets of mass $m$.
\end{prp}

Since the proofs of \cref{prp_ecyleball_nonsphere,prp_ecyleshell_nonshell} are similar, we only prove
the first one.

\begin{proof}[Proof of \cref{prp_ecyleball_nonsphere}]
Recall that in \cref{dfn_compactcase}, when $G$ is compactly supported, a generalized minimizer
is a collection of connected sets in the measure theoretic sense.\\
We proceed in two steps. In a first step we prove an auxiliary version of the
statement, in which the connectedness requirement is omitted. In a second step, we use this result
to deduce the full statement, possibly for a different value of the mass.\\[5pt]
\textit{Step 1.}
We argue by contradiction and assume that for every mass $m$, either Problem~\cref{minpb} has no
minimizer, or it admits a disjoint union of balls as a minimizer.
Then for every generalized minimizer, we may replace any possible non-ball component of mass $m_i$
by a collection of balls with total mass $m_i$, and build another generalized minimizer whose
components are all balls. Thus, we can assume that for every mass $m$, there is a generalized
minimizer made only of balls.
We deduce from \cref{prp_genminball}  that for every  $m>0$ we have
\be\label{proof_prp_ecyleball_nonsphere_1}
\dfrac{e_G(m)}m \ge  \rho_G^{\ball}.
\ee
Next, by definition of $\rho_G^{\cyl}$ and $\sigma_G^{\cyl}$ (see \cref{eq:defrhosigmacyl}) and since by assumption $\rho_G^{\cyl}<\rho_G^{\ball}$ there exists $l>0$ such that
\[
\sigma_G^{\cyl}(l) < \rho_G^{\ball}.
\]
By definition of $\sigma_G^{\cyl}(l)$, we have for some $L=L(m)$ large enough
\[
\dfrac{e_G(|C_{l,L}|)}{|C_{l,L}|} \le \dfrac{\cE_G(C_{l,L})}{|C_{l,L}|}<  \rho_G^{\ball} .
\]
This contradicts~\cref{proof_prp_ecyleball_nonsphere_1} and concludes this step.\\[5pt]
\textit{Step 2.} We deduce the full statement with connectedness.\\
By the previous step, we can find some mass $m$ for which Problem~\cref{minpb} admits a minimizer
$E$ and such that neither balls nor any disjoint union of balls are minimizing in the class of sets of mass $m$.
If $E$ is connected, then there is nothing to do. If $E$ is not connected, we claim that there
exists a connected component $F$ of $E$ with mass $m_F=|F|>0$ such that: $F$ is a minimizer of
\cref{minpb} among sets of mass $m_F$, and for this mass neither balls nor any disjoint union of
balls are minimizers.
By contradiction, assume that for every connected component $F_i$ (which is necessarily a
minimizer among sets of mass $m_i=|F_i|$) of $E$, the ball of mass $m_i$ or a disjoint union of balls with
total mass $m_i$ is a minimizer. Then we may replace each $F_i$ by a disjoint union of balls and
build a minimizer of mass $m$ which is a disjoint union of balls: a contradiction. This concludes
the proof.
\end{proof}

\subsection{A slicing formula}
\label{Ss_slicing}
We  present useful  formulas which express the 
interaction energy $\cI_G(E)$ through the interaction energies of the one-dimensional slices of $E$.
Similar formulas appear in the literature (see~\cite{MP2021,CR2025} or in~\cite{Lud2014} for
anisotropic perimeters). As they depend on $n$ we put back the superscripts $n$.

\begin{dfn}\label{dfn_Esigmay} 
Let $E\sub\R^n$, for $y\in\R^n$ and $\sigma\in \Sp^{n-1}$ we set:
\[
E_{\sigma,y}= \big\{ s\in\R ~:~ y+s\sigma\in E\}.
\] 
\end{dfn}

\begin{prp}[Interaction energy using 1D-slices]\label{prp_1Dslices}
Writing $G(|x|)\coloneqq G(x)$, we have for every measurable set $E\sub\R^n$, 
\be\label{prp_1Dslices_id1}
\cI_G^n(E) = \dfrac1{2} \int_{\Sp^{n-1}}\int_{\sigma^\perp} \int_{E_{\sigma,y'}\times
E_{\sigma,y'}} |s-t|^{n-1}G(s-t)\, ds\, dt\, d\h^{n-1}(y')\, d\h^{n-1}(\sigma).
\ee
Moreover, when $E$ is a bounded convex set, there holds
\be\label{eq_defFG}
\cI_G^n(E) = \dfrac1{2} \int_{\Sp^{n-1}}\int_{\sigma^\perp} 
S_G^n(\h^1(E_{\sigma,y'}))\,d\h^{n-1}(y')\, d\h^{n-1}(\sigma),
\ee
where we have set 
\be\label{eq_defSG}
S_G^n(L) \coloneqq \int_0^L \int_0^L |s-t|^{n-1}G(s-t)\,ds\, dt.
\ee
\end{prp}

\begin{proof}
The first identity follows from simple changes of variables and Fubini. One can
reproduce almost verbatim the proof of \cite[Proposition~3.1]{MP2021} (albeit replacing occurrences of $E\times(\R^n\setminus E)$ by $E\times E$). Here are the details.  We perform the change of variable $y=x+r\sigma$ with $\sigma\in\Sp^{n-1}$ and $r>0$ and use Fubini to get 
\begin{align*}
\cI_G^n(E)&=\int_{E\times E} G(y-x)\,dy\,dx\\
&=\int_{\Sp^{n-1}} \int_{(0,+\oo)}\int_{\R^n}\ind_E(x)\ind_E(x+r\sigma) r^{n-1}G(r)\, dx\, dr\,d\h^{n-1}(\sigma).
\end{align*}
Using the change of variable $r=-\tilde r$, $\sigma=-\tilde \sigma$, we also have 
\[
\cI_G^n(E)=\int_{\Sp^{n-1}} \int_{(-\oo,0)}\int_{\R^n}\ind_E(x)\ind_E(x+r\sigma) |r|^{n-1}G(|r|)\, dx\, dr\,d\h^{n-1}(\sigma).
\]
Averaging the last two identities we get 
\begin{align*}
\cI_G^n(E)&=\dfrac12\int_{\Sp^{n-1}} \int_\R \lt[\int_{\R^n}\ind_E(x)\ind_E(x+r\sigma) |r|^{n-1}G(|r|)\, dx\rt]\, dr\,d\h^{n-1}(\sigma).
\end{align*}
Then, we perform the change of variable in the inner integral, $x=y'+s\sigma$ for $y'\in\sigma^\perp$ and $s\in\R$. By Fubini we have 
\begin{multline*}
\cI_G^n(E)\\
=\dfrac12\int_{\Sp^{n-1}} \int_{\sigma^{\perp}} \lt[\int_{\R\times \R}
\ind_E(y'+s\sigma)\ind_E(y'+(r+s)\sigma) |r|^{n-1}G(|r|)\, ds\,dr\rt]\, d\h^{n-1}(y')\, d\h^{n-1}(\sigma).
\end{multline*}
Using the change of variable $r=t-s$, the term into square brackets rewrites as
\begin{align*}
\int_{\R\times \R}
\ind_E(y'+s\sigma)&\ind_E(y'+(r+s)\sigma) |r|^{n-1}G(|r|)\, ds\,dr\\
&=\int_{\R\times \R}\ind_E(y'+s\sigma)\ind_E(y'+t\sigma) |s-t|^{n-1}G(|s-t|)\, ds\,dt\\
&=\int_{E_{\sigma,y'}\times
E_{\sigma,y'}} |s-t|^{n-1}G(|s-t|)\, ds\, dt.
\end{align*}
This yields~\cref{prp_1Dslices_id1}.\medskip 

Since the intersection of a bounded convex set with any line is a (possibly empty) segment,~\cref{eq_defFG} is a direct consequence of~\cref{prp_1Dslices_id1}. 
\end{proof}

As an immediate application, using  the symmetries of the ball, we obtain a simpler expression of the energy of a ball in terms of the function $S_G$.

\begin{cor}\label{cor:rieszball}
For every $R>0$, there holds
\[
\cI_G^n(B_R) = \dfrac1{2}|\Sp^{n-1}||\Sp^{n-2}|\int_0^R S_G^n\big(2\sqrt{R^2-r^2}\big)r^{n-2}\,dr.
\]
\end{cor}
\begin{proof}
For $\sigma\in \Sp^{n-1}$, $y'\in\sigma^\perp$, we have for $E=B_R$,
\[
E_{\sigma,y'}=\begin{cases}\lt[- \sqrt{R^2-|y'|^2}, \sqrt{R^2-|y'|^2}\rt]&\text{if }|y'|<R,\\
\qquad\void&\text{if }|y'|\ge R.
\end{cases}
\]
It follows that 
\[
\h^1(E_{\sigma,y'})=2\sqrt{(R^2-|y'|^2)_+},
\]
with the notation $a_+\coloneq \max(a,0)$. We deduce form~\cref{eq_defFG},
\[
\cI_G^n(B_R)=\dfrac12 \int_{\Sp^{n-1}}\int_{\sigma^\perp} 
S_G^n\lt(2\sqrt{(R^2-|y'|^2)_+}\rt)\,d\h^{n-1}(y')\,d\h^{n-1}(\sigma ).
\]
The inner integral does not depend on $\sigma$ and we can write 
\[
\cI_G^n(B_R)=\dfrac{|\Sp^{n-1}|}2 \int_{\R^{n-1}} 
S_G^n\lt(2\sqrt{(R^2-|y'|^2)_+}\rt)\,dy'.
\]
Finally  using polar coordinates we write  $y'=r\tau$ with $r>0$, $\tau\in \Sp^{n-2}$ and we get the desired identity.
\end{proof}

\begin{rmk}\label{cor:rieszcyl}
We can apply the formula of the corollary to the interaction energy in the
expression~\cref{eq:defrhosigmacyl} of $\sigma_{G}^{n,\cyl}(l)$ in~\cref{lem_cylinf}. More precisely,  setting
\[
G^{\cyl}(x') = \int_{\R} G(x'+te_n)\,dt,\qquad \text{ for }x'\in\R^{n-1},
\]
we have
\[
\sigma_{G}^{n,\cyl}(l)
=\dfrac{n-1}{l}+\dfrac1{|B_l^{n-1}|}\cI_{G^{\cyl}}^{n-1}(B_l^{n-1}).
\]
This is used in~\cref{S_riesz_cylinder,S_truncated_riesz_cylinder}.
\end{rmk}

\section{Riesz potentials}\label{S_riesz}

In this section, we treat the case $G(x)=R_\alpha(x)=|x|^{-\alpha}$ for $\alpha\in(0,n)$ and $n\ge 2$. 
We use the notation
\be\label{def_beta}
\beta=\beta(\alpha,n)\coloneqq n+1-\alpha\ \in (1,n+1).
\ee

\subsection{Energy-per-mass ratio of balls}

We first compute the function $S_G^n$ defined in 
\cref{eq_defSG} in \cref{prp_1Dslices} in the case of the Riesz potentials.

\begin{lem}
For every $\alpha\in (0,n)$ and every $L>0$, we have
\[
S_{R_\alpha}^n(L) = \frac{2}{\beta(\beta-1)} L^\beta.
\]
Thus, for every convex set $E$, we have
\[
\I^n_{R_\alpha}(E) = \dfrac1{\beta(\beta-1)}\int_{\Sp^{n-1}}\int_{\sigma^\perp}
\h^1(E_{\sigma,y})^\beta\,dy\, d\h^{n-1}(\sigma).
\]
\end{lem}

\begin{proof}
A simple computation gives
\begin{multline*}
\int_0^L \int_0^L |s-t|^{n-1-\alpha}\,ds\, dt
= \int_0^L \lt( \int_{-t}^{L-t}  |r|^{n-1-\alpha}\,dr\rt)\,dt\\ 
= \int_0^L \lt[\int_0^t r^{n-1-\alpha}\, dr+ \int_0^{L-t} r^{n-1-\alpha}\,dr\rt]\, dt
=\dfrac{2L^{n+1-\alpha}}{(n-\alpha)(n+1-\alpha)} .
\end{multline*}
This proves the formula for $S_{R_\alpha}^n(L)$. The formula for $\I_{R_\alpha}(E)$ is then an application of \cref{prp_1Dslices}.
\end{proof}

\begin{rmk}
Notice that the quantity
\[
\int_{\Sp^{n-1}}\int_{\sigma^\perp}
\h^1(E_{\sigma,y})^\beta\,dy\, d\h^{n-1}(\sigma)
\]
makes sense even for $n\le \alpha<n+3$ (that is, $1\ge \beta>-2$).  For $\beta=1$ we have  by Fubini, 
\[
\dfrac1{|\Sp^{n-1}|} \int_{\Sp^{n-1}}\int_{\sigma^\perp}
\h^1(E_{\sigma,y})\,dy\, d\h^{n-1}(\sigma) = |E|.
\]
For $\beta=0$, by Crofton's formula (see \cite[Theorem~3.2.26]{Fed1996} or
\cite[§3.16]{Mor2016}),
\[
\dfrac1{|B^{n-1}|} \int_{\Sp^{n-1}}\int_{\sigma^\perp}
\h^1(E_{\sigma,y})^0\,dy\, d\h^{n-1}(\sigma) = P(E),
\]
(here we adopt the convention $\h^1(E_{\sigma,y})^0 = 0$ if $E_{\sigma,y}=\void$).
\end{rmk}

We deduce an expression for the Riesz energy of balls.

\begin{lem}\label{lem_Riesz_nrj_of_balls}
For every $\alpha\in(0,n)$ and every $R>0$, we have
\[
\I^n_{R_\alpha}(B_R) =
\frac{2^{\beta}\pi^{n-\tfrac1{2}}}{\beta-1}
\frac{\Gamma\lt(\frac{\beta}{2}\rt)}{
\Gamma\lt(\frac{n}{2}\rt)\Gamma\lt(\frac{n+\beta+1}{2}\rt)}  R^{n+\beta-1}.
\]
\end{lem}

\begin{proof}
First of all, notice that by a scaling we have
\[
\I^n_{R_\alpha}(B_R) = R^{2n-\alpha} \I^n_{R_\alpha}(B_1)
= R^{n+\beta-1} \I^n_{R_\alpha}(B_1).
\]
Next by \cref{cor:rieszball} we have
\be\label{proof_lem_Riesz_nrj_of_balls_1}
\I^n_{R_\alpha}(B_1)
= \dfrac1{2}|\Sp^{n-1}||\Sp^{n-2}|
\int_0^1 \frac{2(2\sqrt{R^2-r^2})^\beta}{\beta(\beta-1)} r^{n-2}\,dr.
\ee
Now, recall that the Beta-function (or Euler integral of the first kind) is defined as
\[
B(x,y) = \int_0^1 t^{x-1}(1-t)^{y-1}\,dt,\qquad\forall x,y\in\mathbb{C}\text{ s.t. }
\mathrm{Re}(x),~\mathrm{Re}(y)>0,
\]
and satisfies
\[
B(x,y) = \frac{\Gamma(x)\Gamma(y)}{\Gamma(x+y)}.
\]
Thus, making the change of variable $t=r^2$, we find
\[
\int_0^1 (1-r^2)^{\frac{\beta}{2}} r^{n-2}\,dr
=\dfrac1{2} B\lt(1+\frac{\beta}{2},\frac{n-1}{2}\rt)
= \dfrac1{2}
\frac{\Gamma\lt(1+\frac{\beta}{2}\rt)\Gamma\lt(\frac{n-1}{2}\rt)}{
\Gamma\lt(\frac{n+\beta+1}{2}\rt)}.
\]
Putting this in~\cref{proof_lem_Riesz_nrj_of_balls_1}, we compute
\[
\begin{aligned}
\I^n_{R_\alpha}(B_1)
&= \frac{2^{\beta-1}n(n-1)|B_1^n||B_1^{n-1}|}{\beta(\beta-1)}
\frac{\Gamma\lt(1+\frac{\beta}{2}\rt)\Gamma\lt(\frac{n-1}{2}\rt)}{
\Gamma\lt(\frac{n+\beta+1}{2}\rt)}\\
&=\frac{2^{\beta-1}}{\beta(\beta-1)}
\, \frac{\pi^{n-\tfrac1{2}}n(n-1)}{\Gamma\lt(1+\frac{n}{2}\rt)\Gamma\lt(1+\frac{n-1}{2}\rt)}
\,  \frac{\Gamma\lt(1+\frac{\beta}{2}\rt)\Gamma\lt(\frac{n-1}{2}\rt)}{
\Gamma\lt(\frac{n+\beta+1}{2}\rt)}\\
&=\frac{2^{\beta}\pi^{n-\tfrac1{2}}}{\beta-1}\, \frac{\Gamma\lt(\frac{\beta}{2}\rt)}{
\Gamma\lt(\frac{n}{2}\rt)\Gamma\lt(\frac{n+\beta+1}{2}\rt)}\,,
\end{aligned}
\]
where we used the fact that $\Gamma(1+z)=z\Gamma(z)$.
\end{proof}

Lastly, we give an expression of the optimal energy-per-mass ratio of balls in terms of the
Riesz energy of the unit ball.

\begin{lem}\label{lem_eballriesz}
For any $\alpha\in(0,n)$ we have
\[
\rho_{R_\alpha}^{\ball}
= \frac{n\beta}{\beta-1}\lt(\frac{(\beta-1)\I_{R_\alpha}(B_1)}{n|B_1|}\rt)^{1/\beta},
\]
where $\rho_{R_\alpha}^{\ball}$ is defined by~\cref{eq:defrhoball} with $G=R_\alpha$.
\end{lem}

\begin{proof}
We have
\[
\frac{\cE_{R_\alpha}(B_R)}{|B_R|}
= \frac{n}{R} + \frac{R^{n-1+\beta}}{|B_1| R^n}\I_{R_\alpha}(B_1)
= n\lt(\dfrac1{R} + R^{\beta-1}\frac{\I_{R_\alpha}(B_1)}{n|B_1|}\rt).
\]
Since $\beta>1$, the infimum over $R$ is reached at the critical point
\[
R_* = \lt(\frac{n|B_1|}{(\beta-1)\I^n_{R_\alpha}(B_1)}\rt)^{1/\beta}
\]
and the value of the infimum is
\[
\rho_{R_\alpha}^{n,\ball}
= \lt(\frac{n\beta}{\beta-1}\rt)\dfrac1{R_*},
\]
which gives the result.
\end{proof}

\subsection{Energy-per-mass ratio of infinite cylinders}
\label{S_riesz_cylinder}
For a nonnegative measurable kernel $G$, the kernel $G^{\cyl}$ is defined as in \cref{cor:rieszcyl} by
\[
G^{\cyl}(x') \coloneqq \int_{\R} G(x',t)\,dt\qquad \text{ for a.e. }x'\in\R^{n-1}.
\]

Elementary computations lead to the following  lemma.

\begin{lem}\label{lem_ecylriesz}
For every $\alpha\in(1,n)$, we have
\[
R_\alpha^{\cyl}(x') = c_{\alpha} R_{\alpha-1}(x')
\qquad\text{ with }\qquad
c_{\alpha} \coloneqq
\frac{\sqrt{\pi}\Gamma\lt(\frac{\alpha-1}{2}\rt)}{\Gamma\lt(\frac{\alpha}{2}\rt)}.
\]
Then the energy-per-mass ratio of infinite cylinders of radius $l$ defined by~\cref{eq:defrhosigmacyl} is
\[
\sigma_{R_\alpha}^{n,\cyl}(l)
=\frac{n-1}{l}+\frac{c_{\alpha}l^{\beta-1}}{|B_1^{n-1}|}\I_{R_{\alpha-1}}^{n-1}(B_1^{n-1})
\]
and
\[
\rho_{R_\alpha}^{n,\cyl}
=
\frac{(n-1)\beta}{\beta-1}\lt(\frac{c_{\alpha}(\beta-1)\cI_{R_{\alpha-1}}^{n-1}(B_1^{n-1})}{(n-1)|B_1^{n-1}|}\rt)^{1/\beta}.
\]
\end{lem}

\begin{proof}
By a change of variable, for $x'\in\R^{n-1}\setminus\{0\}$, we have
\[
R_\alpha^{\cyl}(x')
= \int_{\R} \dfrac{dt}{\lt(|x'|^2+t^2\rt)^{\alpha/2}}\,
=  \dfrac{c_{\alpha}}{|x'|^{\alpha-1}},
\]
where we denote 
\[
c_{\alpha}\coloneqq\int_{\R} \dfrac{dt}{\lt(1+t^2\rt)^{\alpha/2}}\,.
\]
Thus, by \cref{cor:rieszball}, we have
\[
\sigma_{R_\alpha}^{n,\cyl}(l)
=\frac{n-1}{l}+\frac{c_{\alpha}}{|B_l^{n-1}|}\cI^{n-1}_{R_\alpha^{\cyl}}(B_l^{n-1})
=\frac{n-1}{l}+\frac{c_{\alpha}l^{\beta-1}}{|B^{n-1}|}\cI_{R_{\alpha-1}}^{n-1}(B_1^{n-1}).
\]
Since $\beta>1$, the infimum over $l$ is reached at the critical point
\[
l_* \coloneqq
\lt(\frac{(n-1)|B_1^{n-1}|}{c_{\alpha}(\beta-1)\cI_{R_{\alpha-1}}^{n-1}(B_1^{n-1})}\rt)^{1/\beta}.
\]
The minimum is then
\[
\rho_{R_\alpha}^{n,\cyl}
=\lt(\frac{(n-1)\beta}{\beta-1}\rt)\dfrac1{l_*}.
\]
We still have to express $c_{\alpha}$ in terms of the $\Gamma$ function.
Using the changes of variable $s=t^2$ and then $u=\dfrac1{1+s}$, we compute
\[
c_{\alpha}
=\int_0^\oo \dfrac{ds}{\sqrt{s}(1+s)^{\alpha/2}} \,
=\int_0^1 u^{\frac{\alpha-3}{2}}(1-u)^{-\tfrac1{2}}\,du=
B\lt(\dfrac12,\dfrac{\alpha - 1}2\rt)
=\frac{\sqrt{\pi}\Gamma\lt(\tfrac{\alpha-1}{2}\rt)}{\Gamma\lt(\tfrac{\alpha}{2}\rt)},
\]
where $B(x,y)$ is the Beta function, and
where we used the identities $B(x,y)=\Gamma(x)\Gamma(y)/\Gamma(x+y)$ and   $\Gamma(1/2)=\sqrt{\pi}$. This ends the proof of the lemma.
\end{proof}

\subsection{Energy-per-mass ratio of balls versus infinite cylinders}

\begin{prp}\label{prop_Riesz_rapport>1}
For every $n\ge2$ and $\alpha\in(1,n)$ we have
\be\label{prop_Riesz_rapport>1_1}
\frac{\rho_{R_\alpha}^{\cyl}}{\rho_{R_\alpha}^{\ball}}
=\frac{n-1}{n}\lt(\frac{\Gamma\lt(\frac{n-\beta}{2}\rt)}{\Gamma\lt(\frac{n-\beta+1}{2}\rt)}
\frac{\Gamma\lt(\frac{n+\beta+1}{2}\rt)}
{\Gamma\lt(\frac{n+\beta}{2}\rt)}\rt)^{1/\beta}.
\ee
If in addition $\alpha$ is an integer then we have
\[
\frac{\rho_{R_\alpha}^{\cyl}}{\rho_{R_\alpha}^{\ball}}
> 1.
\]
\end{prp}

\begin{rmk}
In view of some numerical computations using formula~\cref{eq_expratioriesz} below, we believe that $\rho_{R_\alpha}^{\cyl}>\rho_{R_\alpha}^{\ball}$
whenever $\alpha\in(1,n)$, even if $\alpha$ is not an integer. This remains however to be established.
\end{rmk}

\begin{proof}[Proof of \cref{prop_Riesz_rapport>1}]
Recall that the $\beta$ as defined in~\cref{def_beta} associated with the kernel
$R_{\alpha-1}^{n-1}$ is the same for the kernel $R_\alpha^n$.
Using \cref{lem_eballriesz} and \cref{lem_ecylriesz} we compute
\[
\begin{aligned}
\tau
&\coloneqq{\rho_{R_\alpha}^{\cyl}}/{\rho_{R_\alpha}^{\ball}}\\
&=\lt(\frac{(n-1)\beta}{\beta-1}\lt[\frac{c_{\alpha}(\beta-1)\cI_{R_{\alpha-1}}^{n-1}(B_1^{n-1})}{(n-1)|B_1^{n-1}|}\rt]^{1/\beta}\rt)
\lt(\frac{\beta-1}{n\beta}\lt[\frac{n|B_1^n|}{(\beta-1)\I^n_{R_\alpha}(B_1)}\rt]^{1/\beta}\rt)\\
&=\frac{n-1}{n}\lt[c_{\alpha}\lt(\frac{n|B_1^n|}{(n-1)|B_1^{n-1}|}\rt)
\lt(\frac{\cI_{R_{\alpha-1}}^{n-1}(B_1^{n-1})}{\I^n_{R_\alpha}(B_1)}\rt)\rt]^{1/\beta}\\
&=\frac{n-1}{n}\lt[c_{\alpha}\lt(\frac{\sqrt{\pi}\Gamma\lt(\frac{n-1}{2}\rt)}{\Gamma\lt(\frac{n}{2}\rt)}\rt)
\lt(\frac{\Gamma\lt(\frac{n}{2}\rt)\Gamma\lt(\frac{n+\beta+1}{2}\rt)}
{\pi\Gamma\lt(\frac{n-1}{2}\rt)\Gamma\lt(\frac{n+\beta}{2}\rt)}\rt)\rt]^{1/\beta}\\
&=\frac{n-1}{n}\lt(\frac{c_{\alpha}}{\sqrt{\pi}}
\frac{\Gamma\lt(\frac{n+\beta+1}{2}\rt)}
{\Gamma\lt(\frac{n+\beta}{2}\rt)}\rt)^{1/\beta}.
\end{aligned}
\]
Recalling  $\beta=n+1-\alpha$ and the  expression of $c_{\alpha}$ given in \cref{lem_ecylriesz}, this gives~\cref{prop_Riesz_rapport>1_1}. 

Now, we assume that $\alpha$ (and thus $\beta$) is an integer and we set  $p\coloneqq\beta-1\in (0,n-1)$. We get from~\cref{prop_Riesz_rapport>1_1}
\[
\begin{aligned}
\frac{\Gamma\lt(\frac{n+\beta}{2}\rt)}
{\Gamma\lt(\frac{n-\beta}{2}\rt)}
&= \frac{\Gamma\lt(\frac{n+1+p}{2}\rt)}
{\Gamma\lt(\frac{n-(1+p)}{2}\rt)}\\
&= \dfrac{(n-1+p)(n-1+p-2)(n-1+p-4)\dotsm (n-1-p)}{2^{p+1}}
\end{aligned}
\]
and similarly
\[
\frac{\Gamma\lt(\frac{n+\beta+1}{2}\rt)}
{\Gamma\lt(\frac{n-\beta+1}{2}\rt)}
= \dfrac{(n+p)(n+p-2)(n+p-4)\dotsm (n-p)}{2^{p+1}}.
\]
We obtain
\be\label{eq_expratioriesz}
\tau^\beta
=\lt(\frac{n-1}{n}\rt)^{p+1}\frac{(n+p)(n+p-2)\dotsm (n-p)}{(n-1+p)(n-1+p-2)\dotsm
(n-1-p)}.
\ee
If $p$ is odd, we can regroup the terms by pairs to get 
\[
\tau^\beta  = \prod_{k=0}^{(p-1)/2} c_{n,2k+1},
\]
with  
\[
c_{n,j}\coloneqq\lt(\frac{n-1}{n}\rt)^2 \frac{(n+j)}{n-1+j} \,  \frac{(n-j)}{n-1-j}
=\frac{1-(j/n)^2}{1- (j/(n-1))^2}.
\]
If $p$ is even, after a similar pairing, we are left with the harmless extra factor:
\[
\lt(\frac{n-1}{n}\rt)\lt(\frac{n}{n-1}\rt) = 1.
\]
We find
\[
\tau^\beta= \prod_{k=1}^{p/2} c_{n,2k}.
\]
Since $ c_{n,j}> 1$ for $j\ne0$, we deduce, in both cases, that $\tau^\beta >1$.
\end{proof}

\begin{rmk}\label{rmk_closeto1}
Interestingly, the quotient $\rho_{R_\alpha}^{\cyl} / \rho_{R_\alpha}^{\ball}$ is quite close to 1. Indeed, for $n$ large, using \cref{eq_expratioriesz} and elementary calculus, we deduce that 
\[
\frac{\rho_{R_\alpha}^{\cyl}}{\rho_{R_\alpha}^{\ball}}  = 1+ \dfrac{\beta^2}{3n^3}+O\lt(\dfrac{\beta^2}{n^4}\rt)\qquad\text{ uniformly with respect to }\beta\in(1,n+1).
\]
For small values of $n$, using \cref{eq_expratioriesz},  in the case $\alpha=n-2$ (that is,
$\beta=3$, or equivalently $p=2$)  we obtain for $n=4,5,6,7$,
\begin{align*}
\frac{\rho_{R_2}^{4,\cyl}}{\rho_{R_2}^{4,\ball}}
= \lt(\frac{27}{20}\rt)^{\frac1{3}}\simeq 1.11,
&\qquad\quad
\frac{\rho_{R_3}^{5,\cyl}}{\rho_{R_3}^{5,\ball}}
= \lt(\frac{28}{25}\rt)^{\frac1{3}}\simeq 1.038,\\
\frac{\rho_{R_4}^{6,\cyl}}{\rho_{R_4}^{6,\ball}}
= \lt(\frac{25}{23}\rt)^{\frac1{3}}\simeq 1.028,
&\qquad\quad
\frac{\rho_{R_5}^{7,\cyl}}{\rho_{R_5}^{7,\ball}}
= \lt(\frac{405}{392}\rt)^{\frac1{3}}\simeq 1.011.
\end{align*}
\end{rmk}

\section{Truncated Coulomb potentials}\label{S_truncated_riesz}

We now consider the truncated Riesz potentials 
\[
R_{\alpha,\kappa}(x)\coloneqq 
\ind_{|x|<\kappa}|x|^{-\alpha},
\]
for $n\ge 2$ and $\alpha\in(0,n)$. We restrict the analysis to the Coulomb case $n=3$, $\alpha=1$ only when required. We still use the notation $\beta=\beta(\alpha,n)\coloneqq n+1-\alpha$ of~\cref{def_beta}.

\subsection{Energy-per-mass ratio of balls}

First we compute the interaction energy on 1D slices as defined in~\cref{eq_defSG}.  
\begin{lem}\label{lem_SRiesztrunc}
For every $\alpha\in (0,n)$, $\kappa>0$ and every $L>0$, we have
\[
S_{R_{\alpha,\kappa}}(L) =
\frac{2}{\beta(\beta-1)}
\begin{cases}
\displaystyle \qquad L^\beta&\quad\text{ if } L\le \kappa\\
\displaystyle \kappa^{\beta-1}(\beta L-(\beta-1) \kappa)&\quad\text{ for }L>\kappa. 
\end{cases}
\]
\end{lem}

\begin{proof}
For $L\le \kappa$ we have $|s-t|<\kappa$ on $(0,L)\times (0,L)$ so that the truncation has no effect and $S_{R_{\alpha,\kappa}}(L)=S_{R_\alpha}(L)$. 
Assuming  $L<\kappa$, we compute
\begin{multline*}
S_{R_{\alpha,\kappa}}(L)
=\int_0^L \int_0^L |s-t|^{n-1-\alpha}\ind_{|s-t|<\kappa}\,ds\, dt
=2 \int_0^L \lt[\int_0^t r^{n-1-\alpha}\ind_{r<\kappa}\,dr \rt]\, dt\\
=\dfrac2{n-\alpha}
 \lt(
 \int_0^\kappa t^{n-\alpha}\, dt 
+
(L-\kappa)  \kappa^{n-\alpha} 
\rt)
=\dfrac2{n-\alpha}\lt(\dfrac{\kappa^{n-\alpha+1}}{n-\alpha+1} + (L-\kappa)  \kappa^{n-\alpha} \rt).
\end{multline*}
We conclude by inserting the identity $\alpha=n+1-\beta$ and simplifying.
\end{proof}

Next, we deduce an expression of the truncated Riesz energy of balls in terms of the incomplete Beta
function.

\begin{lem}\label{lem_rieszcut}
Let $\alpha\in(0,n)$, $\kappa>0$ and $R>0$. Writing
\[
B(z; x, y) = \int_0^z t^{x-1} (1-t)^{y-1} \,dt,
\qquad\forall z\in\R,~x>0,~y>0
\]
for the incomplete Beta function and setting $\lambda=\frac{\kappa}{2R}$, we have
\[
\begin{aligned}
\frac{\I_{{R}_{\alpha,\kappa}}(B_R)}{|B_R|}
&=
\frac{n(n-1)|B_1^{n-1}|\kappa^{\beta-1}}{\beta(\beta-1)}
\lt[
\dfrac1{\lambda^{\beta-1}} B\lt( \lambda^2; \frac{\beta+2}{2},\frac{n-1}{2}\rt)\rt.\\
&\qquad\qquad\qquad
\lt.
-\frac{2(\beta-1)\lambda}{n-1} (1-\lambda^2)^{\frac{n-1}{2}} 
+ \beta B\lt(1-\lambda^2; \frac{n-1}{2}, \frac{3}{2}\rt)
\rt],
\end{aligned}
\]
if $\kappa\le 2R$, and $\I_{R_{\alpha,\kappa}}(B_R) =\I_{R_{\alpha}}(B_R)$ if $\kappa>2R$.
\end{lem}

\begin{proof}
In view of the expression of $S_{R_{\alpha,\kappa}}$ in \cref{lem_SRiesztrunc}, it is obvious that
$\I_{{R}_{\alpha,\kappa}}(B_R)=\I_{R_\alpha}(B_R)$ if $\kappa>2R$.
If $\kappa\le 2R$, by \cref{lem_SRiesztrunc} and \cref{cor:rieszball}, we have
\[
\begin{aligned}
\I_{{R}_{\alpha,\kappa}}(B_R)
&= \dfrac1{2}|\Sp^{n-1}||\Sp^{n-2}|\int_0^R
S_{\alpha,\kappa}^n\big(2\sqrt{R^2-r^2}\big)r^{n-2}\,dr\\
&=\frac{|\Sp^{n-1}||\Sp^{n-2}|}{\beta(\beta-1)}
\lt(
2^\beta\int_{\sqrt{R^2-(\kappa/2)^2}}^R (R^2-r^2)^{\frac{\beta}{2}} r^{n-2}\,dr\rt.\\
&\qquad\qquad
\lt.+\kappa^{\beta-1}\int_0^{\sqrt{R^2-(\kappa/2)^2}} \lt(2\beta\sqrt{R^2-r^2}-(\beta-1)\kappa\rt) r^{n-2}\,dr
\rt)
\end{aligned}
\]
thus 
\begin{multline*}
\I_{{R}_{\alpha,\kappa}}(B_R)
= \frac{|\Sp^{n-1}||\Sp^{n-2}|}{\beta(\beta-1)}
\lt(
2^\beta\int_{\sqrt{R^2-(\kappa/2)^2}}^R (R^2-r^2)^{\frac{\beta}{2}} r^{n-2}\,dr\rt.\\
\lt.  -\frac{(\beta-1)\kappa^\beta}{n-1}\lt(R^2-(\kappa/2)^2\rt)^{(n-1)/2}
+2\beta \kappa^{\beta-1}\int_0^{\sqrt{R^2-(\kappa/2)^2}} \sqrt{R^2-r^2}\, r^{n-2}\,dr
\rt).
\end{multline*}
Using the changes of variables $s=r/R$ followed by $t=s^2$, we compute
\begin{multline*}
\int_0^{\sqrt{R^2-(\kappa/2)^2}} \sqrt{R^2-r^2} r^{n-2}\,dr
= R^n \int_0^{\sqrt{1-(\kappa/(2R))^2}} \sqrt{1-s^2} s^{n-2}\,ds\\
= \frac{R^n}{2} \int_0^{1-\lt(\frac{\kappa}{2R}\rt)^2} (1-t)^{\frac1{2}}
t^{\frac{n-3}{2}}\,dt
= \frac{R^n}{2} B\lt( 1-\lt(\frac{\kappa}{2R}\rt)^2; \frac{n-1}{2}, \frac{3}{2}\rt)
\end{multline*}
and
\[
\begin{aligned}
\int_{\sqrt{R^2-(\kappa/2)^2}}^R (R^2-r^2)^{\frac{\beta}{2}} r^{n-2}\,dr
&=R^{n+\beta-1}\int_{\sqrt{1-(\kappa/(2R))^2}}^1 (1-s^2)^{\frac{\beta}{2}} s^{n-2}\,ds\\
&=\frac{R^{n+\beta-1}}{2} \int_{1-\lt(\frac{\kappa}{2R}\rt)^2}^1 (1-t)^{\frac{\beta}{2}}
t^{\frac{n-3}{2}}\,dt\\
&=\frac{R^{n+\beta-1}}{2} \int_0^{\lt(\frac{\kappa}{2R}\rt)^2} u^{\frac{\beta}{2}}
(1-u)^{\frac{n-3}{2}}\,du\\
&=\frac{R^{n+\beta-1}}{2}B\lt( \lt(\frac{\kappa}{2R}\rt)^2; \frac{\beta+2}{2}, \frac{n-1}{2}\rt).
\end{aligned}
\]
Hence, setting $\lambda=\frac{\kappa}{2R}$, we find
\[
\begin{aligned}
&\I_{{R}_{\alpha,\kappa}}(B_R)\\
&\quad=\frac{n(n-1)|B_1^n||B_1^{n-1}|}{\beta(\beta-1)}
\cdot\lt[
2^{\beta-1} R^{n+\beta-1}
B\lt( \lt(\frac{\kappa}{2R}\rt)^2; \frac{\beta+2}{2}, \frac{n-1}{2}\rt)
\rt.\\
&\qquad
\lt.-\frac{(\beta-1)\kappa^\beta
R^{n-1}}{n-1}\lt(1-\lt(\frac{\kappa}{2R}\rt)^2\rt)^{\frac{n-1}{2}}
+ \beta\kappa^{\beta-1} R^n B\lt(1-\lt(\frac{\kappa}{2R}\rt)^2; \frac{n-1}{2},
\frac{3}{2}\rt)\rt]\\
&\quad= \frac{n(n-1)|B_1^n||B_1^{n-1}|}{2^n\beta(\beta-1)} \kappa^{n+\beta-1}
\lt[
\dfrac1{\lambda^{n+\beta-1}} B\lt(\lambda^2; \frac{\beta+2}{2},\frac{n-1}{2} \rt)\rt.\\
&\qquad\qquad\qquad\qquad
\lt.
-\frac{2(\beta-1)}{(n-1)\lambda^{n-1}} (1-\lambda^2)^{\frac{n-1}{2}} 
+ \frac{\beta}{\lambda^n} B\lt(1-\lambda^2; \frac{n-1}{2}, \frac{3}{2}\rt)
\rt],
\end{aligned}
\]
which gives the result using that $\kappa^n = (2\lambda R)^n$ and $|B_R| = |B_1| R^n$.
\end{proof}

Lastly, we obtain the optimal energy-per-mass ratio in the truncated Riesz case, in dimension
$n=3$ for every $\alpha\in(0,3)$.

\begin{prp}\label{prp_eballriesztrunc}
Let $n=3$, $\alpha\in(0,3)$, $\beta=4-\alpha$, $\kappa>0$ and let us set
\[
\kappa_{\min} \coloneqq \lt(\frac{\beta}{\pi}\rt)^{1/\beta},
\qquad
\kappa_{\max} \coloneqq \lt(\frac{\beta(\beta+2)}{2\pi}\rt)^{1/\beta},
\qquad 
\lambda_{*,\alpha,\kappa} =
\lt[\frac{\beta+2}{\pi}\lt(\frac{\pi}{\beta}-\dfrac1{\kappa^\beta}\rt)\rt]^{1/2}
\]
and
\[
f_{\alpha,\kappa}(\lambda) \coloneqq \frac{6\lambda}{\kappa} + 6\pi\kappa^{\beta-1}
\lt( \frac{\lambda^3}{3(\beta+2)} -\frac{\lambda}{\beta}+\frac{2}{3(\beta-1)}\rt).
\]
There holds:
\be\label{prp_eballriesztrunc_1}
\rho^{3,\ball}_{R_{\alpha,\kappa}}=
\begin{cases}
f_{\alpha,\kappa}(0)=\frac{4\pi\kappa^{\beta-1}}{\beta-1}
&\text{if }0<\kappa\le\kappa_{\min},\medskip\\
\quad f_{\alpha,\kappa}(\lambda_{*,\alpha,\kappa})
&\text{if }\kappa_{\min}<\kappa<\kappa_{\max},\medskip\\
\qquad \rho^{3,\ball}_{R_\alpha}
&\text{if }\kappa\ge\kappa_{\max}.
\end{cases}
\ee

\end{prp}

\begin{proof}
We write $\lambda=\kappa/(2R)$.
Since $n=3$,
\[
\frac{P(B_R)}{|B_R|} = \frac{3}{R} = \frac{6 \lambda}{\kappa},
\]
by \cref{lem_rieszcut} we have for $R\ge\kappa/2$ (that is, $\lambda\in(0,1]$),
\begin{multline}\label{proof_prp_eballriesztrunc_1}
\frac{\cE_{R_{\alpha,\kappa}}(B_R)}{|B_R|}
= \frac{6}{\kappa}\lambda + \frac{6\pi\kappa^{\beta-1}}{\beta(\beta-1)}\times\medskip
\\
\lt[
\dfrac1{\lambda^{\beta-1}} B\lt( \lambda^2; \frac{\beta+2}{2}, 1\rt)
-\frac{2(\beta-1)\lambda}{n-1} (1-\lambda^2) 
+ \beta B\lt(1-\lambda^2; 1, \frac{3}{2}\rt)
\rt]\\
 \eqqcolon f_{\alpha,\kappa}(\lambda).
\end{multline}
 When $\kappa> 2R$ we simply have
\[
\frac{\cE_{R_{\alpha,\kappa}}(B_R)}{|B_R|}
=\frac{\cE_{R_\alpha}(B_R)}{|B_R|},
\]
Hence,
\be\label{proof_prp_eballriesztrunc_2}
\rho^{3,\ball}_{R_{\alpha,\kappa}}  =\min\lt(\inf_{(0,1]}f_{\alpha,\kappa},\ \inf_{R\le \kappa/2}\frac{\cE_{R_\alpha}(B_R)}{|B_R|}\rt).
\ee

In the rest of the proof, we simply write $f$ for $f_{\alpha,\kappa}$ to lighten the
notation. We first look for the infimum of the function $f$ defined on $[0,1]$. We compute the terms involving $\beta$-functions as follows:
\[
B\lt(\lambda^2; \frac{\beta+2}{2}, 1\rt)
= \int_0^{\lambda^2} t^{\frac{\beta}{2}}\,dt
= \frac{2 \lambda^{\beta+2}}{\beta+2}
\]
and
\[
B\lt(1-\lambda^2; 1,\frac{3}{2}\rt)
= \int_0^{1-\lambda^2} (1-t)^{\frac1{2}}\,dt
= \frac{2}{3}(1-\lambda^3).
\]
Substituting these identities in \cref{proof_prp_eballriesztrunc_1}, we compute
\[
\begin{aligned}
f(\lambda)
&= \frac{6\lambda}{\kappa}+ \frac{6\pi\kappa^{\beta-1}}{\beta(\beta-1)}
\lt( \frac{2\lambda^3}{\beta+2}-(\beta-1)(\lambda-\lambda^3)+\frac{2\beta}{3}(1-\lambda^3)
\rt)\\
&= \frac{6\lambda}{\kappa}+ 6\pi\kappa^{\beta-1}
\lt( \frac{\lambda^3}{3(\beta+2)} -\frac{\lambda}{\beta}+\frac{2}{3(\beta-1)}\rt).
\end{aligned}
\]
This function is smooth, and after computation, we find
\[
f'(\lambda) = \frac{6}{\kappa}
\lt[\lt(1-\frac{\pi\kappa^\beta}{\beta}\rt)+\frac{\pi\kappa^\beta}{\beta+2}
\lambda^2\rt].
\]
There are two cases:\medskip

\noindent
\textit{Case 1, $\kappa^\beta\le \beta/\pi=\kappa_{\min}^\beta$.}\\  
In that case, $f$ is increasing on $(0,1)$ so that the minimum is reached at $\lambda=0$. We have
\[
\inf_{(0,1]} f= f(0) = \frac{4\pi\kappa^{\beta-1}}{\beta-1}.
\]
Note that $\lambda=0$ corresponds to $R\to+\oo$, which is expected.\footnote{When the range of the
interaction is small enough, the  energy-per-mass ratio of $B_R$ improves as $R$ increases. Indeed, the contribution of the interaction energy is  constant up to  $O(
\kappa/R)$ and the contribution of the perimeter decreases in $O(1/R)$.}
In view of~\cref{{proof_prp_eballriesztrunc_2}} we now have to compare this value with
\[
Q\coloneqq\inf_{R\le\kappa/2}~ \frac{\cE_{R_\alpha}(B_R)}{|B_R|}.
\]
For this we remark that
 \be\label{proof_prp_eballriesztrunc_3}
f(0)
= \inf_{R\ge\kappa/2} \frac{\cE_{R_{\alpha,\kappa}}(B_R)}{|B_R|}.
\ee
By the assumption $\kappa\le\kappa_{\min}$, one checks that
\[
R_{*,\alpha} = \argmin_{R>0}~ 
\frac{\cE_{R_\alpha}(B_R)}{|B_R|}
=\lt(\frac{n|B_1^n|}{(\beta-1)\I_{R_\alpha}(B_1)}\rt)^{1/\beta}
\]
satisfies $R_{*,\alpha}\ge \kappa/2$, therefore, 
\[
\inf_{R>0} \frac{\cE_{R_\alpha}(B_R)}{|B_R|} =  \frac{\cE_{R_\alpha}(B_{R_{*,\alpha}})}{|B_{R_{*,\alpha}}|} = Q.
\]
With~\cref{proof_prp_eballriesztrunc_3}, we get $f(0)\le Q$ so that $\rho^{3,\ball}_{R_{\alpha,\kappa}} =f(0)= \frac{4\pi\kappa^{\beta-1}}{\beta-1}$. This proves the first case of~\cref{prp_eballriesztrunc_1}.
\medskip
%

\noindent
\textit{Case 2, $\kappa^\beta_{\min}=\frac{\beta}{\pi}<\kappa^\beta$.} Denoting
\[
\lambda_{*,\alpha,\kappa} =
\lt[\frac{\beta+2}{\pi}\lt(\frac{\pi}{\beta}-\dfrac1{\kappa^\beta}\rt)\rt]^{1/2}
\]
the positive root of $f'$, we consider two subcases:
\begin{enumerate}[(a)]
\item \textit{Assume $\lambda_{*,\alpha,\kappa}\ge 1$.} A simple computation shows that this is equivalent to
\[
\kappa^\beta \ge \kappa_{\max}^\beta = \frac{\beta(\beta+2)}{2\pi} = (2R_{*,\alpha})^\beta .
\]
In that case $f$ is decreasing on $(0,1)$ and the minimum of $f$ is
reached in $\lambda=1$. Its value is 
\[
f(1)=\frac{\cE_{R_{\alpha,\kappa}}(B_{\kappa/2})}{|B_{\kappa/2}|}= \frac{\cE_{R_\alpha}(B_{\kappa/2})}{|B_{\kappa/2}|}.
\]
There remains to compare $f(1)$ with $\cE_{R_\alpha}(B_R)/|B_R|$ for $R\le\kappa/2$. Since $R_{*,\alpha}\le\kappa/2$, we have 
\[
f(1)= \frac{\cE_{R_\alpha}(B_{\kappa/2})}{|B_{\kappa/2}|} 
\ge\inf_{0<R<\kappa/2} \frac{\cE_{R_\alpha}(B_R)}{|B_R|} 
 = \frac{\cE_{R_\alpha}(B_{R_{*,\alpha}})}{|B_{R_{*,\alpha}}|}
=\rho^{3,\ball}_{R_\alpha}.
\]
Hence $\rho^{3,\ball}_{R_{\alpha,\kappa}}=\rho^{3,\ball}_{R_\alpha}$ which proves the third case of~\cref{prp_eballriesztrunc_1}.\medskip

\item \textit{Assume  $\lambda_{*,\alpha,\kappa}<1$.}   We have
\[
\kappa_{\min}^\beta < \kappa^\beta < \kappa_{\max}^\beta.
\]
The function $f$ is decreasing on $(0,\lambda_*)$ and increasing on $(\lambda_*,1)$,
so that its minimum is reached in $\lambda_*$.
Since $\kappa < \kappa_{\max} = 2R_{*,\alpha}$, we have
\[
f(\lambda_*)
= \inf_{R>\kappa/2} \frac{\cE_{\alpha,\kappa}(B_R)}{|B_R|}
\ \le\  \inf_{R>\kappa/2} \frac{\cE_{\alpha}(B_R)}{|B_R|}
=  \frac{\cE_{\alpha}(B_{R_{*,\alpha}})}{|B_{R_{*,\alpha}}|}
=\rho^{3,\ball}_{R_{\alpha}}.
\]
We conclude that $\rho^{3,\ball}_{R_{\alpha,\kappa}}=f(\lambda_*)$ which is the second case of~\cref{prp_eballriesztrunc_1}.
\end{enumerate}
The proposition is established.
\end{proof}

\subsection{Energy-per-mass ratio of infinite cylinders}
\label{S_truncated_riesz_cylinder}
In the rest of the section we work in dimension $n=3$. To simplify notation we write $B_R$ for balls in
$\R^3$ and $B_R'$ for balls in $\R^2$.

We compute the energy-per-mass ratio for the infinite cylinder of radius $l$.

\begin{prp}\label{prp_ecyltruncriesz}
Let $n=3$, $\alpha=1$, $\kappa>0$ and $l\leq \kappa/2$. Setting $\lambda=\frac{\kappa}{2l}$ and
\[
g(\ell)= 
\frac{4\pi}{3}\lt[\ell-\arcsin(\ell)+2\ell\lt(1-\sqrt{1-\ell^2}\rt)
+\ell^3\argth\lt(\sqrt{1-\ell^2}\rt)\rt],\quad\forall\ell<1,
\]
we have
\[
\I^2_{R^{\cyl}_{1,\kappa}}(B_l^2)
= 2\kappa^2 l^2\lambda \int_0^{1} g\lt(\lambda^{-1}\sqrt{1-r^2}\rt)\,dr.
\]
As a consequence, the energy-per-mass ratio of infinite cylinders of radius $l$  is
\begin{equation}\label{fcyl_integral}
\sigma_{R_{1,\kappa}}^{3,\cyl}(l)
=\frac{4\lambda}{\kappa}+\frac{2\lambda \kappa^2}{\pi} \int_0^{1} g\lt(\lambda^{-1}\sqrt{1-r^2}\rt)\,dr.
\end{equation}
\end{prp}

\begin{proof}
For $x'\in\R^{n-1}\setminus\{0\}$, by the definition of $G^{\cyl}$ given in
\cref{cor:rieszcyl} and a change of variable we have
\[
R^{\cyl}_{\alpha,\kappa}(x')
= \int_{\R} R_{\alpha,\kappa}(x',t)\,dt
= \int_{\R} \frac{\ind_{(0,\kappa)}(\sqrt{|x'|^2+t^2})}{(|x'|^2+t^2)^{\alpha/2}}\,dt
= \frac{c_{\alpha,\kappa}(|x'|)}{|x'|^{\alpha-1}},
\]
where
\[
c_{\alpha,\kappa}(l)\coloneqq
\begin{cases}
\displaystyle 2\int_0^{\sqrt{(\kappa/l)^2-1}}\dfrac1{(1+t^2)^{\alpha/2}}\,dt&\qquad\text{ if
}\kappa> l,\\[12pt]
\displaystyle \qquad0&\qquad\text{ otherwise}.
\end{cases}
\]
Since $\alpha=1$, we get
\[
c_{1,\kappa}(l)
= 2\argsh\lt(\sqrt{(\kappa/l)^2-1}\rt) 
= 2\argth\lt(\sqrt{1-(l/\kappa)^2}\rt)
\]
whenever $l<\kappa$, thus
\[
c_{1,\kappa}(l)= 2\argth\lt(\sqrt{1-(l/\kappa)^2}\rt)\ind_{(0,\kappa)}(l).
\]
Let $R\leq\kappa/2$.
By \cref{cor:rieszball} we have
\[
\I^{n-1}_{R^{\cyl}_{1,\kappa}}(B_R)
= \dfrac1{2}|\Sp^{n-2}||\Sp^{n-3}|\int_0^R
S^{n-1}_{R^{\cyl}_{1,\kappa}}\big(2\sqrt{R^2-r^2}\big)r^{n-3}\,dr,
\]
where
\[
\begin{aligned}
S^{n-1}_{R^{\cyl}_{1,\kappa}}(L)
&= \int_0^L \int_0^L |s-t|^{n-2} R^{\cyl}_{1,\kappa}(s-t)\,ds\, dt\\
&= \int_0^L \int_0^L |s-t|^{n-2}c_{1,\kappa}(|s-t|)\,ds\, dt.
\end{aligned}
\]
Let us compute $S^{n-1}_{R^{\cyl}_{1,\kappa}}(L)$ for $n=3$.
Since $R\leq\kappa/2$, we have $2\sqrt{R^2-r^2}\leq\kappa$ for every $r\in[0,R]$, thus we are only
interested in computing $S^{n-1}_{R^{\cyl}_{1,\kappa}}(L)$ for $L\le\kappa$.
In that case, $|s-t|\le \kappa$, and we have
\[
\begin{aligned}
S^{2}_{R^{\cyl}_{1,\kappa}}(L)
&= 2\int_0^L \int_0^L
|s-t|\argth\lt(\sqrt{1-\lt(|s-t|/\kappa\rt)^2}\rt)\,ds\, dt\\
&= 2\int_0^L \int_{-t}^{L-t}
|u|\argth\lt(\sqrt{1-\lt(u/\kappa\rt)^2}\rt)\,du\,dt\\
&= 4\int_0^L \int_{0}^t u\argth\lt(\sqrt{1-\lt(u/\kappa\rt)^2}\rt)\,du\,dt\\
&= 4\kappa^3\int_0^{L/\kappa} \int_{0}^{t} s\argth\lt(\sqrt{1-s^2}\rt)\,ds\, dt.
\end{aligned}
\]
Since $\frac{d}{ds}\lt(\argth\lt(\sqrt{1-s^2}\rt)\rt)=\frac{-1}{s\sqrt{1-s^2}}$
and $\argth(\sqrt{1-t^2})\ \stackrel{t\dw0}\sim\ \frac{-1}{t}$, integrating by parts, we find
\be\label{riesztruncyl:eq1}
\begin{aligned}
\int_{0}^{t} s\argth\lt(\sqrt{1-s^2}\rt)\,ds
&=\lt[ \frac{s^2}{2}\argth(\sqrt{1-s^2})\rt]_0^{t}
+\dfrac1{2}\int_0^{t} \frac{s}{\sqrt{1-s^2}}\,ds\\
&=\frac{t^2}{2}\argth\lt(\sqrt{1-t^2}\rt)
+\dfrac1{2}\lt(1-\sqrt{1-t^2}\rt).
\end{aligned}
\ee
Thus
\[
\begin{aligned}
S^{2}_{R^{\cyl}_{1,\kappa}}(L)
&= 2\kappa^2\int_0^L
\lt(t/\kappa\rt)^2\argth\lt(\sqrt{1-\lt(t/\kappa\rt)^2}\rt)
+\lt(1-\sqrt{1-\lt(t/\kappa\rt)^2}\rt)\,dt\\
&= 2\kappa^3\int_0^{L/\kappa}
t^2\argth\lt(\sqrt{1-t^2}\rt)
+\lt(1-\sqrt{1-t^2}\rt)\,dt.
\end{aligned}
\]
We now set $\ell=L/\kappa$ to lighten notation.
Integrating by parts once again, we compute
\[
\begin{aligned}
\int_0^{\ell} t^2\argth\lt(\sqrt{1-t^2}\rt)
&
=\lt[\frac{t^3}{3}\argth\lt(\sqrt{1-t^2}\rt)\rt]_0^{\ell}
+\dfrac1{3}\int_0^{\ell} \frac{t^2}{\sqrt{1-t^2}}\,dt\\
&
=\frac{\ell^3}{3}\argth\lt(\sqrt{1-\ell^2}\rt)
+\dfrac1{3}\int_0^{\ell} \dfrac1{\sqrt{1-t^2}}-\sqrt{1-t^2}\,dt,
\end{aligned}
\]
thus
\[
S^{2}_{R^{\cyl}_{1,\kappa}}(L)
= \frac{2\kappa^3}{3}\ell^3\argth\lt(\sqrt{1-\ell^2}\rt)
+ 2\kappa^3\ell
+\frac{2\kappa^3}{3}\int_0^{\ell}\frac{dt}{\sqrt{1-t^2}}
-\frac{8\kappa^3}{3}\int_0^{\ell}\sqrt{1-t^2}\,dt.
\]
On the one hand $\arcsin$ is a primitive of $t\mapsto (1-t^2)^{-1/2}$
and on the other hand
\begin{multline*}
\int_0^{\ell} \sqrt{1-t^2}
= \int_0^{\arcsin(\ell)} \cos^2\te\,d\te
 =\dfrac1{2}\int_0^{\arcsin(\ell)} 1+\cos(2\te)\,d\te\\
=\dfrac1{2}\arcsin(\ell)+\dfrac1{4}\arcsin(2\sin\ell)
=\dfrac1{2}\arcsin\lt(\ell\rt)+\dfrac1{2}\ell\sqrt{1-\ell^2}.
\end{multline*}
Hence, for $L\le\kappa$,
\[
S^{2}_{R^{\cyl}_{1,\kappa}}(L)
=\frac{2}{3}\kappa^3
\lt[\ell^3\argth\lt(\sqrt{1-\ell^2}\rt)
+\ell-\arcsin(\ell)+2\ell\lt(1-\sqrt{1-\ell^2}\rt)\rt].
\]
Recalling that
\[
\I^{2}_{R^{\cyl}_{1,\kappa}}(B'_R)
= 2\pi\int_0^R S^{2}_{R^{\cyl}_{1,\kappa}}\big(2\sqrt{R^2-r^2}\big)\,dr.
\]
setting
\[
g(\ell)= 
\frac{4\pi}{3}\lt[\ell-\arcsin(\ell)+2\ell\lt(1-\sqrt{1-\ell^2}\rt)
+\ell^3\argth\lt(\sqrt{1-\ell^2}\rt)\rt]
\qquad\forall \ell\leq 1,
\]
we find
\[
\I^{2}_{R^{\cyl}_{1,\kappa}}(B'_R)
= \kappa^3\int_0^R g\lt(2\sqrt{R^2-r^2}/\kappa\rt)\,dr.
\]
By a change of variable, we have
\[
\kappa^3 \int_0^{R} g\lt(2\sqrt{R^2-r^2}/\kappa\rt)\,dr
= \kappa^3 R \int_0^1 g\lt(\frac{2R}{\kappa}\sqrt{1-r^2}\rt)\,dr
= 2\kappa^2 R^2 \int_0^1 \lambda g\lt(\lambda^{-1}\sqrt{1-r^2}\rt)\,dr,
\]
thus, eventually,
\[
\I^2_{R^{\cyl}_{1,\kappa}}(B'_R)
= 2\kappa^2 R^2\lambda \int_0^{1} g\lt(\lambda^{-1}\sqrt{1-r^2}\rt)\,dr.
\]
Hence by \cref{cor:rieszcyl} we deduce
\[
\sigma_{R_{1,\kappa}}^{3,\cyl}(l)
=\frac{2}{l}+\dfrac1{|B'_l|}\I_{R_{1,\kappa}^{\cyl}}^{2}(B'_l)\\
=\frac{4\lambda}{\kappa}+\frac{2\lambda\kappa^2}{\pi}
\int_0^{1} g\lt(\lambda^{-1}\sqrt{1-r^2}\rt)\,dr,
\]
where $\lambda=\frac{\kappa}{2l}$. This concludes the proof.
\end{proof}

In the next lemma, we compute explicitly the truncated Coulomb energy of $B_{\kappa/2}$, which we
use later to compute the truncated Coulomb energy-per-mass of the infinite cylinder with radius $\kappa/2$.

\begin{lem}\label{lem_rieszcutratio1}
We have
\[
\I^2_{R_{1,\kappa}^{\cyl}}(B_{\kappa/2})
= \pi\kappa^4
\lt(\frac{\pi}{2}
-\frac{17}{12}
+\frac{C}{2}\rt)
\]
where $C$ is the Catalan number defined by~\cref{def_Catalan} below.
\end{lem}

\begin{proof}
Recall that for $R\leq\kappa/2$, we have
\[
\I^2_{R^{\cyl}_{1,\kappa}}(B_R')
= 2\kappa^2 R^2\lambda \int_0^{\min(1,\lambda)} g\lt(\lambda^{-1}\sqrt{1-r^2}\rt)\,dr.
\]
where $\lambda=\kappa/(2R)$, and the function $g$ is given by
\cref{prp_ecyltruncriesz}.
Here since $R=\kappa/2$, we have $\lambda=1$ and thus
\[
\I^2_{R_{1,\kappa}^{\cyl}}(B'_{\kappa/2})
= \frac{\kappa^4}{2} \int_0^1 g\lt(\sqrt{1-r^2}\rt)\,dr.
\]
We compute the integral of each term composing $g(\sqrt{1-r^2})$. We have
\[
\int_0^1 \sqrt{1-r^2}\,dr
= \int_0^{\pi/2} \cos^2\te\,d\te
= \frac{\pi}{4},
\]
\[
\int_0^1 \arcsin(\sqrt{1-r^2})\,dr
= \int_0^{\pi/2} \sin\te\arcsin(\sin\te)\,d\te
= \int_0^{\pi/2} \te\sin\te\,d\te
=1,
\]
then, since $1-\sqrt{1-(\sqrt{1-r^2})^2}=1-r$,
\[
\int_0^1 (1-r)\sqrt{1-r^2}\,dr
=\int_0^{\pi/2} (1-\sin\te)\cos^2\te\,d\te
= \frac{\pi}{4}-\dfrac1{3}.
\]
As for the last term, since $\argth(r) = \dfrac1{2}\log\lt(\frac{1+r}{1-r}\rt)$, using
\cref{cor:catalan} below we have
\[
\begin{aligned}
\int_0^1 (1-r^2)^{3/2}\argth(r)\,dr
= \dfrac1{2}\int_0^{\pi/2} \cos^4\te
\log\lt(\frac{1+\sin\te}{1-\sin\te}\rt)\,d\te
= \frac{3C}{4}-\frac{11}{24},
\end{aligned}
\]
where $C$ denotes the Catalan constant defined by~\cref{def_Catalan} below.
Hence, we obtain 
\[
\I^2_{R_{1,\kappa}^{\cyl}}(B'_{\kappa/2})
= \pi\kappa^4
\lt(\frac{\pi}{2}
-\frac{17}{12}
+\frac{C}{2}\rt)
\]
which concludes the computation.
\end{proof}

The following result was used in the proof of \cref{lem_rieszcutratio1}.

\begin{lem}\label{cor:catalan}
Let us denote the Catalan constant 
\be\label{def_Catalan}
C = \sum_{n=0}^{+\oo} \frac{(-1)^n}{(2n+1)^2}.
\ee
We have
\[
\int_0^{\pi/2} \cos^4\te \log\lt(1+\sin\te\rt)\,d\te
= \pi\lt(\frac{7}{64}-\frac{3\log 2}{16}\rt)-\frac{11}{24} +\frac{3C}{4}
\]
and
\[
\int_{0}^{\pi/2} \cos^4\te \log\lt(1-\sin\te\rt)\,d\te
= \pi\lt(\frac{7}{64}-\frac{3\log 2}{16}\rt)+\frac{11}{24}-\frac{3C}{4}.
\]
\end{lem}

We first prove the following simpler identities:

\begin{lem}\label{lem_catalan}
We have
\begin{equation}\label{catalan:res1}
\int_{-\pi/2}^{\pi/2} \log(1+\sin\theta)\,d\theta
= -\pi\log 2
\end{equation}
and
\begin{equation}\label{catalan:res2}
\int_{0}^{\pi/2} \log(1+\sin\theta)\,d\theta
=2C-\frac{\pi\log 2}{2},
\end{equation}
where $C$ is given by \cref{def_Catalan}.
\end{lem}

\begin{proof}
We compute
\begin{multline}\label{catalan:eq1}
\int_{-\pi/2}^{\pi/2} \log(1+\sin\theta)\,d\theta
= \frac{1}{2}\int_{-\pi/2}^{\pi/2} \log(1+\sin\theta)\,d\theta
+ \frac{1}{2}\int_{-\pi/2}^{\pi/2} \log(1-\sin\theta)\,d\theta\\
= \frac{1}{2}\int_{-\pi/2}^{\pi/2} \log(1-\sin^2\theta)\,d\theta
= 2\int_{0}^{\pi/2} \log(\cos\theta)\,d\theta.
\end{multline}
Notice that
\[
\begin{aligned}
\int_{0}^{\pi/2} \log(\cos\theta)\,d\theta
&=\frac{1}{2}\int_{0}^{\pi/2} \log(\sin\theta)\,d\theta
+\frac{1}{2}\int_{0}^{\pi/2} \log(\cos\theta)\,d\theta\\
&=\frac{1}{2}\int_{0}^{\pi/2} \log(\sin\theta\cos\theta)\,d\theta\\
&=\frac{1}{2}\int_{0}^{\pi/2} \log(\sin(2\theta))\,d\theta-\frac{\pi\log 2}{4}\\
&=\frac{1}{4}\int_{0}^{\pi} \log(\sin\theta)\,d\theta-\frac{\pi\log 2}{4}\\
&=\frac{1}{2}\int_{0}^{\pi/2} \log(\cos\theta)\,d\theta-\frac{\pi\log 2}{4},
\end{aligned}
\]
thus
\begin{equation}\label{catalan:eq2}
\int_{0}^{\pi/2} \log(\cos\theta)\,d\theta = -\frac{\pi\log2}{2}.
\end{equation}
Inserting this into \cref{catalan:eq1} gives \cref{catalan:res1}.
Next, making the change of variable $t=\tan(\theta/2)$ we compute
\begin{multline}\label{catalan:eq3}
\int_{0}^{\pi/2} \log(1+\sin\theta)\,d\theta
=\int_{0}^{\pi/2} \log(1+\cos\theta)\,d\theta\\
=2\int_{0}^{1} \frac{1}{1+t^2}\log\left(\frac{2}{1+t^2}\right)\,dt
=\frac{\pi\log 2}{2}-2\int_{0}^{1} \frac{\log(1+t^2)}{1+t^2}\,dt.
\end{multline}
Making the change of variable $s=1/t$ we find
\[
\int_{0}^{1} \frac{\log(1+t^2)}{1+t^2}\,dt
=\int_1^\infty \frac{\log(1+s^{2})}{1+s^2} \,ds
-2\int_1^\infty \frac{\log s}{1+s^2} \,ds.
\]
Adding $\int_0^1 \frac{\log(1+t^2)}{1+t^2}\,dt$ to both sides of the previous inequality yields
\begin{equation}\label{catalan:eq4}
2\int_{0}^{1} \frac{\log(1+t^2)}{1+t^2}\,dt
=\int_0^\infty \frac{\log(1+t^{2})}{1+t^2} \,dt
-2\int_1^\infty \frac{\log t}{1+t^2} \,dt.
\end{equation}
Using the change of variable $t=\tan\theta$ we find
\begin{equation}\label{catalan:eq5}
\int_0^\infty \frac{\log(1+t^{2})}{1+t^2} \,dt
=-2\int_0^{\pi/2} \log(\cos\theta)\,d\theta
=\pi\log 2,
\end{equation}
where we used \cref{catalan:eq2} for the last equality.
In addition, by the change of variable $t\rightsquigarrow 1/t$ once again, notice that
\begin{equation}\label{catalan:eq6}
\int_1^\infty \frac{\log t}{1+t^2} \,dt
=-\int_0^1 \frac{\log t}{1+t^2} \,dt.
\end{equation}
Gathering \cref{catalan:eq3}, \cref{catalan:eq4}, \cref{catalan:eq5} and \cref{catalan:eq6}we obtain
\[
\int_{0}^{\pi/2} \log(1+\sin\theta)\,d\theta
=-\frac{\pi\log 2}{2}-2\int_0^1 \frac{\log t}{1+t^2}\,dt.
\]
Eventually, we compute
\[
\int_0^1 \frac{\log t}{1+t^2}\,dt
= \sum_{k\geq 0} (-1)^k\int_0^1 t^{2k}\log t\,dt
= \sum_{k\geq 0} \frac{(-1)^{k+1}}{(2k+1)^2}
= -C
\]
from which we deduce \cref{catalan:res2}. This concludes the proof.
\end{proof}

We are now ready to prove \cref{cor:catalan}.

\begin{proof}[Proof of \cref{cor:catalan}]
First of all, by the change of variable $t\mapsto -t$ we have
\begin{equation}\label{corcatalan:eq0}
\int_{0}^{\pi/2} \cos^4\te \log\lt(1-\sin\te\rt)\,d\te
=\int_{-\pi/2}^0 \cos^4\te \log\lt(1+\sin\te\rt)\,d\te
\end{equation}
Then we recall
\[
\cos^4\te = \frac{3}{8}+\dfrac1{2}\cos(2\te)+\dfrac1{8}\cos(4\te).
\]
Integrating by parts, we find that a primitive of $\te\mapsto \cos(2\te)\log(1+\sin\te)$
is 
\[
\theta\mapsto \dfrac1{2}\sin(2\te)\log(1+\sin\te)
+\cos\te+\frac{\te}{2}-\dfrac1{4}\sin(2\te),
\]
thus
\[
\int_0^{\pi/2} \cos(2\te)\log(1+\sin\te)\,d\te
=\frac{\pi}{4}-1
\quad\text{ and }\quad
\int_{-\pi/2}^0 \cos(2\te)\log(1+\sin\te)\,d\te
=1+\frac{\pi}{4}.
\]
Using integration by parts and basic trigonometry, we also find that a primitive of
$\theta\mapsto\cos(4\te)\log(1+\sin\te)$ is
\[
\theta\mapsto \dfrac1{4}\sin(4\te)\log(1+\sin\te)
+\frac{2}{3}\cos^3\te-\cos\te
-\frac{\te}{4}+\dfrac1{4}\sin(2\te)
-\dfrac1{16}\sin(4\te),
\]
thus
\[
\int_0^{\pi/2} \cos(4\te)\log(1+\sin\te)\,d\te
=\dfrac1{3}-\frac{\pi}{8}
\quad\text{ and }\quad
\int_{-\pi/2}^0 \cos(4\te)\log(1+\sin\te)\,d\te
=-\dfrac1{3}-\frac{\pi}{8}.
\]
Using \cref{catalan:res1,catalan:res2}, substituting the above identity and simplifying, we obtain
\[
\begin{aligned}
\int_0^{\pi/2} \cos^4\theta\log(1+\sin\te)\,d\te
&= \frac{3}{8}\lt(2C-\frac{\pi\log 2}{2}\rt)
+\dfrac1{2}\lt(\frac{\pi}{4}-1\rt)
+\dfrac1{8}\lt(\dfrac1{3}-\frac{\pi}{8}\rt)\\
&= \pi\lt(\frac{7}{64}-\frac{3\log 2}{16}\rt)-\frac{11}{24} +\frac{3C}{4}
\end{aligned}
\]
and
\[
\begin{aligned}
\int_{-\pi/2}^{0} \cos^4\theta\log(1+\sin\te)\,d\te
&= \frac{3}{8}\lt(-2C-\frac{\pi\log 2}{2}\rt)
+\dfrac1{2}\lt(1+\frac{\pi}{4}\rt)
+\dfrac1{8}\lt(-\dfrac1{3}-\frac{\pi}{8}\rt)\\
&= \pi\lt(\frac{7}{64}-\frac{3\log 2}{16}\rt)+\frac{11}{24}-\frac{3C}{4}.
\end{aligned}
\]
This proves the lemma in view of \cref{corcatalan:eq0}.
\end{proof}

\subsection{Energy-per-mass ratio of balls versus infinite cylinders}\label{subsec:truncrieszballvscyl}

\begin{prp}
Let $n=3$, $\alpha=1$ and $\kappa=11/10$, then we have
\[
\sigma_{R_{\alpha,\kappa}}^{3,\cyl}(\kappa/2) < 
\rho^{3,\ball}_{R_{\alpha,\kappa}}.
\]
\end{prp}

\begin{proof}
By \cref{prp_ecyltruncriesz} and \cref{lem_rieszcutratio1}, since here $\lambda=1$, we have
\[
\sigma_{R_{1,\kappa}}^{3,\cyl}(\kappa/2)
=\frac{4}{\kappa}+\frac{\I^2_{R_{1,\kappa}^{\cyl}}(B_{\kappa/2})}{\pi(\kappa/2)^2}
=\frac{4}{\kappa}+4\kappa^2 \lt(\frac{\pi}{2} -\frac{17}{12} +\frac{C}{2}\rt)
= \frac{40}{11}+\lt(\frac{11}{5}\rt)^2\lt(\frac{\pi}{2} -\frac{17}{12} +\frac{C}{2}\rt).
\]
By \cref{prp_eballriesztrunc}, we have
\[
\kappa_{\min} = \lt(\frac{3}{\pi}\rt)^{1/3}<\frac{11}{10}=\kappa<
\lt(\frac{15}{2\pi}\rt)^{1/3} = \kappa_{\max},
\]
so that
\[
\rho^{3,\ball}_{R_{1,\kappa}}=f_{1,\kappa}(\lambda_{*,1,\kappa}),
\]
where
\[
\lambda_{*,\kappa,\alpha} =
\lt[\frac{\beta+2}{\pi}\lt(\frac{\pi}{\beta}-\dfrac1{\kappa^\beta}\rt)\rt]^{1/2}
= \lt[\frac{5}{\pi}\lt(\frac{\pi}{3}-\lt(\frac{10}{11}\rt)^3\rt)\rt]^{1/2}
\]
and
\begin{equation}\label{riesztrunc:expf}
f_{1,\kappa}(\lambda)
= \frac{6\lambda}{\kappa} + 6\pi\kappa^{\beta-1}
\lt( \frac{\lambda^3}{3(\beta+2)} -\frac{\lambda}{\beta}+\frac{2}{3(\beta-1)}\rt)\\
= \frac{60}{11}\lambda + 6\pi\lt(\frac{11}{10}\rt)^{2}
\lt( \frac{\lambda^3}{15} -\frac{\lambda}{3}+\dfrac1{3}\rt).
\end{equation}
One can then check using a symbolic computation library (we used Python's SymPy library) that 
\[
\sigma_{R_{1,\kappa}}^{3,\cyl}(\kappa/2)<\rho^{3,\ball}_{R_{1,\kappa}}.
\]
Symbolic computation guarantees the truthness of the inequality (evaluating constants such as $\pi$
and the Catalan constant $C$ to the required precision): see the interactive Jupyter Notebook
provided in the supplementary material.
We obtain the following inclusions:
\begin{align*}
\sigma_{R_{1,\kappa}}^{3,\cyl}(\kappa/2)
&\in (6.60 - 10^{-2}, 6.60 + 10^{-2}),\\
\rho^{3,\ball}_{R_{1,\kappa}}
&\in (6.62-10^{-2},6.62+10^{-2}). 
\end{align*}
\end{proof}

From this we deduce:

\begin{cor}\label{cor:trunccoulomb}
Let $n=3$, $\kappa=11/10$. Then we have $\rho_{R_{1,\kappa}}^{3,\cyl} < \rho_{R_{1,\kappa}}^{3,\ball}$.
\end{cor}

This proves \cref{mainprp_prp1} and thus \cref{mainthm:thm1} in the case of the truncated Coulomb
potential thanks to \cref{prp_ecyleball_nonsphere}.

As a matter of fact, for an interval of $\kappa$ around~$11/10$, we observe numerically that for some $l=l(\kappa)$, we have
$\sigma^{3,\cyl}_{R_{1,\kappa}}(l(\kappa))< \rho_{R_{1,\kappa}}^{3,\ball}$ which implies $\rho_{R_{1,\kappa}}^{3,\cyl}<\rho_{R_{1,\kappa}}^{3,\ball}$ 
, see \cref{fig:trunc_coulomb}. We find  adequate values $l(\kappa)$
by optimizing approximately $l\mapsto \sigma^{3,\cyl}_{R_{1,\kappa}}(l)$. For this optimization, we
use Python's SciPy library and we use Simpson's quadrature formula to evaluate the integral in
\cref{fcyl_integral}.

 \begin{figure}[ht]
    \begin{minipage}{0.48\textwidth}
         \begin{center}
 		\captionsetup{width=.95\textwidth}
 		\includegraphics[width=.95\textwidth]{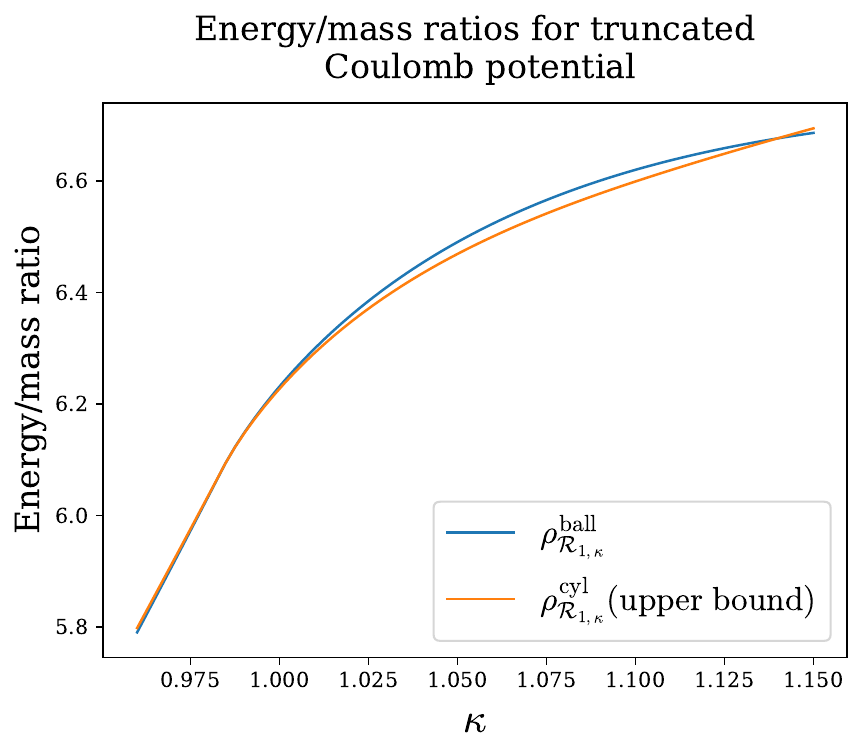}
 		\end{center}
 	\end{minipage}
    \begin{minipage}{0.5\textwidth}
         \begin{center}
 		\captionsetup{width=.95\textwidth}
 		\includegraphics[width=.95\textwidth]{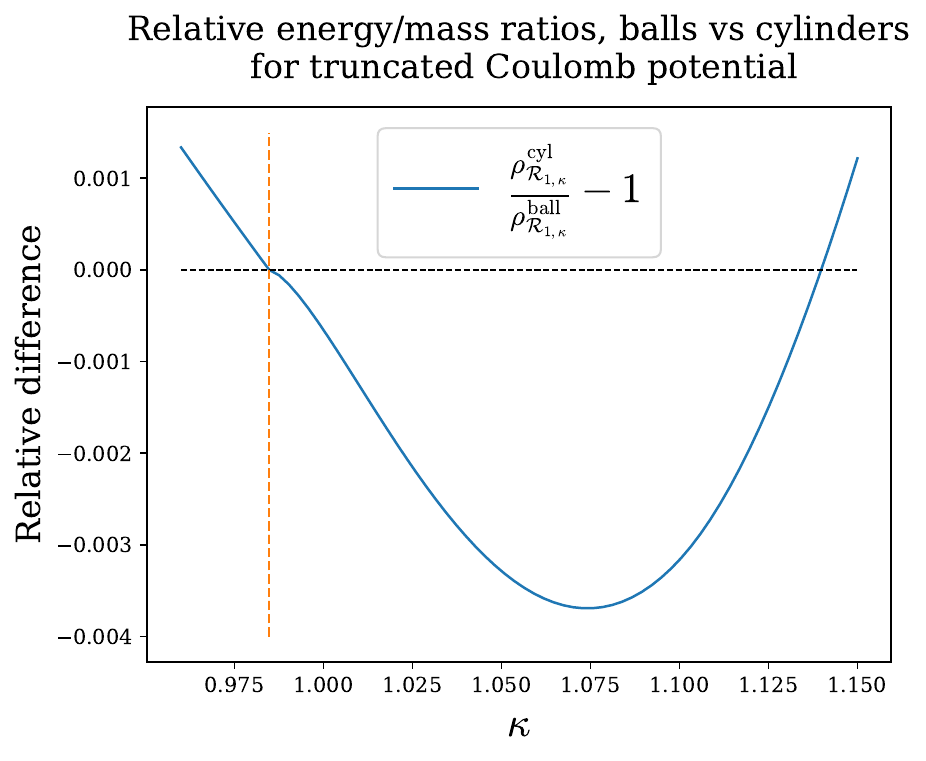}
 		\end{center}
 	\end{minipage}
 	\caption{\label{fig:trunc_coulomb} Comparison of the energy-per-mass ratios for balls and infinite cylinders with the truncated
	Coulomb potential for varying cut-length $\kappa$.
	We used $\sigma_{R_{1,\kappa}}^{3,\cyl}(l(\kappa))$ with appropriate $l(\kappa)$ as an upper bound
	for $\rho_{R_{1,\kappa}}^{3,\cyl}$. Convenient values $l(\kappa)$ are obtained using SciPy library to optimize $l\mapsto \sigma_{R_{1,\kappa}}^{3,\cyl}(l)$.\\ 
	The vertical dashed line in the right figure corresponds to $\kappa=\kappa_{\mathrm{min}}$ from \cref{prp_eballriesztrunc}, the
	threshold below which the infimum $\rho^{3,\mathrm{ball}}_{R_{1,\kappa}}$ is obtained by taking
	balls $B_R$ with $R\up\infty$.}
 \end{figure}

\subsection{Energy-per-mass ratio of core-shells versus balls}\label{subsec:nonradial}

This section is devoted to the proof of \cref{mainthm:thm2}, which reduces to the
computation of the optimal energy-per-mass ratio of core-shells $\rho^{3,\shell}_{R_{1,\kappa}}$, defined
by \cref{eq:defrhoshell}, and its comparison with $\rho^{3,\ball}_{R_{1,\kappa}}$.

\begin{prp}\label{prp:rieszcut_shellball}
Let $n=3$ and $\kappa=11/10$. Then we have $\rho^{3,\shell}_{R_{1,\kappa}} =
\rho^{3,\ball}_{R_{1,\kappa}}$.
\end{prp}

By \cref{cor:trunccoulomb} we have $\rho^{3,\cyl}_{R_{1,\kappa}} < \rho^{3,\ball}_{R_{1,\kappa}}$,
so that \cref{prp:rieszcut_shellball} implies \cref{mainprp_prp2}. Then, \cref{prp_ecyleshell_nonshell} gives \cref{mainthm:thm2}.

We first compute the truncated Coulomb interaction of core-shells, introducing the notation $A_{r,R} =
B_R\setminus B_r$.

\begin{lem}\label{lem:rieszcoreshell}
Let $n=3$ and $\kappa>0$. For any $r,R$ such that $0\leq r<R$ we have
\[
\mathtoolsset{multlined-width=0.5\displaywidth}
\begin{aligned}
\frac{\I^3_{R_{1,\kappa}}(A_{r,R})}{|A_{r,R}|}
= \frac{\pi}{20(R^3 - r^3)}
\begin{cases}
I_1(r,R,\kappa)
&\begin{multlined}[t]
\quad\text{ if } (r<R<3r \text{ and } 0 < \kappa
<R-r)\\[-10pt]\text{ or } (R\geq 3r \text{ and  }0 < \kappa < 2r)\end{multlined}\\
I_2(r,R,\kappa)
&\quad\text{ if } R-r \leq \kappa < 2r\\
I_3(r,R,\kappa)
&\quad\text{ if } 2r \leq \kappa < R-r\\
I_4(r,R,\kappa)
&\begin{multlined}[t]
\quad\text{ if } (r<R<3r \text{ and } 2r \leq \kappa < R+r)\\[-10pt]
\text{ or } (R\geq 3r \text{ and } R-r \leq \kappa < R+r)\end{multlined} \\
I_5(r,R,\kappa)
&\quad\text{ if } R+r \leq \kappa < 2R \\
I_6(r,R,\kappa)
&\quad\text{ if } \kappa \geq 2R
\end{cases}
\end{aligned}
\]
where
\[
\begin{aligned}
I_1(r,R,\kappa)&=2\kappa^5 - 20(R^2+r^2)\kappa^3 + 40(R^3-r^3)\kappa^2\\
I_2(r,R,\kappa)&=30(R^2-r^2)^2\,\kappa -8(R-r)^3(r^2+3rR+R^2)\\
I_3(r,R,\kappa)&=\kappa^5 - 20R^2\kappa^3 + 40R^3\kappa^2 + 32r^5 - 80r^3\kappa^2\\
I_4(r,R,\kappa)&)-\kappa^5 + 20r^2\kappa^3 - 40r^3\kappa^2 + 30(R^2-r^2)^2\kappa - 8R^5 + 40R^3r^2 - 40R^2r^3 + 40r^5\\
I_5(r,R,\kappa)&=\kappa^5 - 20R^2\kappa^3 + 40R^3\kappa^2 + 32r^5 - 16r^3(5R^2-r^2)\\
I_6(r,R,\kappa)&=16 \left( 2R^5 - 5R^2r^3 + 3r^5 \right).
\end{aligned}
\]
\end{lem}

Notice that there are two regimes that appear naturally: the regime of thick shells $r<R< 3r$ and
thin shells $R\geq 3r$. Let us also point out that the two cases $R-r \leq \kappa < 2r$ and $2r \leq
\kappa < R-r$ imply respectively $r<R< 3r$ and $R\geq 3r$.

\begin{proof}
Let us set
\[
\mathcal{J}_{\kappa}(E,F)
\coloneqq \int_{E\times F} R_{1,\kappa}(|x-y|)\,dxdy
\]
so that $\mathcal{J}_{\kappa}(E,E)=\I_{R_{1,\kappa}}^3(E)$ and we have the identity
\begin{equation}\label{rieszshell:eq1}
\I^3_{R_{1,\kappa}}(A_{r,R})
= \mathcal{J}_\kappa(B_R,B_R) + \mathcal{J}_\kappa(B_r,B_r) - 2\mathcal{J}_\kappa(B_r,B_R).
\end{equation}
Thus, we need only compute $\mathcal{J}_\kappa(B_r,B_R)$ for $0\leq r \leq R$.
Using the symmetries of the ball and of the kernel, by the coarea formula we have
\[
\begin{aligned}
\mathcal{J}_\kappa(B_r,B_R)
&= \int_0^{\min(r+R,\kappa)} t^{-1} \mathcal{H}^{5}\left(\left\{(x,y) \in B_r\times B_R~:~ |x-y|=t
\right\}\right)\, dt\\
&= \int_0^{\min(r+R,\kappa)} t^{-1} \mathcal{H}^{5}\left(\left\{(z,y)\in \R^3\times\R^3~:~ |z|=t \text{ and
} y\in B_R\cap B_r(z)\right\}\right)\, dt\\
&= \int_0^{\min(r+R,\kappa)} t^{-1} \mathcal{H}^2(\partial B_{t}) |B_R\cap B_r(t e_1)|\, dt\\
&= 4\pi\int_0^{\min(r+R,\kappa)} t|B_R\cap B_r(t e_1)|\, dt,
\end{aligned}
\]
where $e_1$ is the first vector in the canonical basis of $\R^3$. The volume of the intersection 
$B_R\cap B_r(t e_1)$ is the sum of the volumes of the spherical caps respectively of radii $R$, $r$ and
heights
\[
h_1 = \frac{(r-R+t)(r+R-t)}{2t},\qquad
h_2 = \frac{(R-r+t)(R+r-t)}{2t}.
\]
Recalling that the volume of a spherical cap of radius $r$ and height $h$ is $\frac{\pi
h^2}{3}(3r-h)$, after computation we find
\[
|B_R\cap B_r(t e_1)|
=
\begin{cases}
\frac{4}{3}\pi r^3 & \text{ if } t < R-r\\
\frac{\pi}{12t}(R+r-t)^2(t^2+2t(r+R)-3(R-r)^2) & \text{ if } t\in [R-r,R+r]\\
0 & \text{ otherwise.}
\end{cases}
\]
\[
\mathcal{J}_\kappa(B_r,B_R) 
= 4\pi \int_0^{\min(r+R,\kappa)} t
\]
As a consequence, there are three cases:
\begin{enumerate}[(i)]
\item if $\kappa<R-r$ then
\[
\mathcal{J}_\kappa(B_r,B_R) = \frac{8\pi^2}{3} r^3\kappa^2;
\]
\item if $\kappa\in[R-r,R+r]$, then
\[
J_{\kappa}(B_r,B_R)
= \frac{8\pi^2}{3} r^3 (R-r)^2
+\frac{\pi^2}{3}\int_{x}^{\kappa} (X-t)^2(t^2+2tX-3x^2)\,dt,
\]
where we have set $x=R-r$ and $X=R+r$ for simplicity.
We compute
\[
\begin{multlined}
\int_{x}^{\kappa} (X-t)^2(t^2+2tX-3x^2)\,dt\\
= \int_x^{\kappa}
t^4-3(x^2+X^2)t^2+2X(X^2+3x^2)t-3x^2X^2\,dt
= Q(\kappa)-Q(R-r)
\end{multlined}
\]
where we used
\[
x^2+X^2 = 2(R^2+r^2)
\quad\text{ and }\quad
X(X^2+3x^2) = 4(R^3+r^3)
\]
and set
\[
Q(t) = \frac{t^5}{5}-2(R^2+r^2)t^3+4(R^3+r^3)t^2-3(R^2-r^2)^2t.
\]
Hence in that case
\[
J_{\kappa}(B_r,B_R)
= \frac{8\pi^2}{3} r^3 (R-r)^2 + \frac{\pi^2}{3}\left(Q(\kappa)-Q(R-r)\right).
\]
\item if $\kappa>R+r$, then by the previous point we have
\[
J_{\kappa}(B_r,B_R)
= \frac{8\pi^2}{3} r^3 (R-r)^2 + \frac{\pi^2}{3}\left(Q(R+r)-Q(R-r)\right).
\]
We can simplify this expression. Since
\[
(R+r)^5-(R-r)^5 = 10 R^4 r+20R^2 r^3+2r^5,
\]
\[
(R+r)^3-(R-r)^3 = 6r R^2 + 2r^3,
\qquad\text{ and }\qquad
(R+r)^2-(R-r)^2 = 4 Rr,
\]
we deduce
\[
\begin{aligned}
Q(R+r)-Q(R-r)
&= 2R^4 r + 4R^2r^2 + \frac{2}{5} r^5
-2(R^2+r^2)(6rR^2+2r^3)\\
&\qquad\qquad+16rR(R^3+r^3)-6r(R^2-r^2)^2\\
&= \frac{16}{5} r^4 (5R-3r)
\end{aligned} 
\]
hence
$$
J_\kappa(B_r,B_R)= \frac{\pi^2}{3}
\left(8 r^3(R-r)^2+\frac{16}{5}r^4(5R-3r)\right)
= \frac{8\pi^2}{15} r^3(5R^2-r^2).
$$
\end{enumerate}
Recalling \cref{rieszshell:eq1}, in view of the three cases above, we now need to compare $\kappa$
with the values $2r$, $R-r$, $R+r$, $2R$. We obviously have $\max(2r,R-r) \leq R+r \leq 2R$ but we
need to distinguish between the cases $2r \leq R-r\iff R \geq 3r$ (thick shells) and $2r \geq R-r \iff
R<3R$ (thin shells). 
Summing up all the cases, using the definition of $Q$ and $|A_{r,R}| = \frac{4\pi}{3}(R^3-r^3)$
eventually yields the result.
\end{proof}

Throughout the rest of section, we fix $\kappa=11/10$.
Since $\cE_{R_{1,\kappa}}(A_{r,R}) = P(A_{r,R}) + \I^3_{R_{1,\kappa}}(A_{r,R})$ and $\kappa=11/10$,
in view of \cref{lem:rieszcoreshell} we set
\[
T_i(r,R) =\frac{P(A_{r,R})}{|A_{r,R}|}+I_i(r,R,11/10),\qquad\forall i\in\{1,\ldots,6\}.
\]
Recall that $\frac{P(A_{r,R})}{|A_{r,R}|}=\frac{3(r^2+R^2)}{R^3-r^3}$.
In order to prove \cref{prp:rieszcut_shellball}, there are no real shortcuts: we optimize each
$T_i$ on the relevant subset of $(r,R)\in\R_+^2$ according to \cref{lem:rieszcoreshell} and compare
it with $\rho^{3,\ball}_{R_{1,\kappa}}$. By continuity of $\I^3_{R_{1,\kappa}}$, except for the
constraint $r<R$, we may consider inclusive inequalities rather than strict inequalities.
This is done in the six next following lemmas, where we substitute $\kappa$ with its value $11/10$
only when it seems necessary to go on with the computation.

\begin{lem}\label{lem:T1}
We have
\begin{multline*}
e_1 \coloneqq \inf \Big\{ T_1(r,R) ~:~ 0\leq r<R \text{ and } \big[(R\leq 3r \text{ and } \kappa \leq R-r) \text{ or }
(R\geq 3r \text{ and } \kappa \leq 2r)\big]\Big\}\\
= 2\pi\kappa+\frac{2}{3\kappa}(3-\pi\kappa^3)
\geq 6.88.
\end{multline*}
\end{lem}

\begin{proof}
As we will see, the infimum $e_1$ is obtained by considering arbitrarily thin and large core-shells,
more precisely, $A_{r(t),R(t)}$ with $R(t)=\kappa/(1-t)$, $r(t)=tR(t)$ and $t\to 1^-$.\\
We write $r=tR$ with $t\in[0,1)$ so that with a slight abuse of notation
\[
e_1 = \inf \Big\{ T_1(t,R) ~:~ (t\in[1/3,1) \text{ and } \kappa \leq R(1-t)) \text{ or }
(t\in[0,1/3) \text{ and } \kappa \leq 2tR)\Big\}.
\]
with
\[
T_1(t,R) = -\frac{A}{R}\left(\frac{1+t^2}{1-t^3}\right)+\frac{B}{R^3(1-t^3)}+2\pi\kappa^2,
\]
where
\[
A\coloneqq \pi\kappa^3-3 \qquad\text{ and }\qquad
B\coloneqq \frac{\pi\kappa^5}{20}.
\]
Note that $A>0$ since $\kappa=11/10$.
It is easy to check, via a simple study of the function, that for any fixed $t\in [0,1)$, $R\mapsto
T_1(t,R)$ is nonincreasing on $[0,R_*(t)]$ and nondecreasing on $[R_*(t),1)$, where
\[
R_*(t)
= \sqrt{\frac{3B}{A(1+t^2)}}
\leq  \sqrt{\frac{3B}{A}} < 1.
\]
In the thin-shell case $t\in[1/3,1)$, since $\kappa\leq R(1-t)$ we get $R>3\kappa/2>R_*(t)$
so that $R\mapsto T_1(t,R)$ is nondecreasing on the admissible set for $R$. As a consequence, the
minimum in $R$ is reached at $R^{\mathrm{min}}(t) \coloneqq \kappa/(1-t)$.\\
In the thick-shell case $t\in[0,1/3]$, since $\kappa\leq 2tR$ we get as before $R>3\kappa/2>R_*(t)$
and $R\mapsto T_1(t,R)$ is nondecreasing. The minimum in $R$ is then reached in this case at $R^{\mathrm{min}}(t)
\coloneqq \kappa/(2t)$.\\
With $R^{\mathrm{min}}$ thus defined, for $t\in[0,1)$ fixed, the minimum $R\mapsto T_1(t,R)$ under
the constraints is
\[
T_1^{\mathrm{min}}(t) \coloneqq T_1(t,R^{\mathrm{min}}(t))
\]
and we need only minimize this quantity over $t\in[0,1)$.
We distinguish the thin and thick shell cases once again.
\begin{enumerate}[{Case} 1.]
\item For $t\in [1/3,1)$, we have
\[
\begin{multlined}
T_1^{\mathrm{min}}(t)
= T_1(t,\kappa/(1-t))
= -\frac{A}{\kappa}\left(1-\frac{t}{1+t+t^2}\right)
+\frac{B}{\kappa^3}\left(1-\frac{3t}{1+t+t^2}\right)+2\pi\kappa^2\\
=-\frac{A}{\kappa}+\frac{B}{\kappa^3}+2\pi\kappa^2-\left(\frac{3B}{\kappa^3}-\frac{A}{\kappa}\right)\frac{t}{1+t+t^2}.
\end{multlined}
\]
We notice that for $\kappa=11/10$,
\begin{equation}\label{minT1:eq1}
\frac{3B}{\kappa^3}-\frac{A}{\kappa}=\frac{3}{\kappa}-\frac{7\pi\kappa^2}{10} \geq 0.06 > 0
\end{equation}
and $t\mapsto t/(1+t+t^2)$ is nondecreasing on $[0,1)$. Thus, $T_1^{\mathrm{min}}$ is nonincreasing
on $[1/3,1)$ so that
\[
\inf_{[1/3,1)} T_1^{\mathrm{min}} = T_1^{\mathrm{min}}(1^-) = 2\pi\kappa
+\frac{2}{3\kappa}(3-\pi\kappa^3) > 6.88.
\]
\item For $t\in[0,1/3]$, we have
\[
\begin{multlined}
T_1^{\mathrm{min}}(t)
= T_1(t,\kappa/(2t))
= \frac{-2At}{\kappa}
\left(\frac{1+t^2}{1-t^3}\right)+\frac{8B}{\kappa^3}\left(\frac{t^3}{1-t^3}\right)
= \frac{(8b-2a)t^3-2at}{1-t^3}\\
=2a-8b+\frac{8b-2a-2at}{1-t^3},
\end{multlined}
\]
where we have set $a=A/\kappa$ and $b=B/\kappa^3$.
After computation, we find
\[
(T_1^{\mathrm{min}})'(t) = \frac{-4at^3+(24b-6a)t^2+2a}{(1-t^3)^2}
\eqqcolon \frac{Q(t)}{(1-t^3)^2}
\]
and then $Q'(t) = 4t(-3at+12b-3a)$.
Since $t\mapsto -3at+12b-3a$ is an affine decreasing function (recall $a=A/\kappa>0$), its minimum
on $[0,1/3]$ is reached at $t=1/3$, and is $4(3b-a)>0$ (see \cref{minT1:eq1}). Thus, $Q'$ is
nonnegative and $Q$ is nondecreasing.
Since after computation
\[
Q(1/3) = \frac{4}{27}(18b-19a) = \frac{4}{27\kappa}\left(57-\frac{181}{10}\pi\kappa^3\right)< -2.5<0
\]
we get that $Q$ is negative on $[0,1/3]$ and thus so is $(T_1^{\mathrm{min}})'$. Hence,
$T_1^{\mathrm{min}}$ is nonincreasing on this interval and 
\[
\inf_{[0,1/3]} T_1^{\mathrm{min}} = T_1^{\mathrm{min}}(1/3)
\geq T_1^{\mathrm{min}}(1^-).
\]
\end{enumerate}
We conclude the proof by gathering both cases.
\end{proof}

\begin{lem}\label{lem:T2}
We have
\[
e_2\coloneqq \inf \Big\{ T_2(r,R) ~:~ 0\leq r<R \text{ and } R-r\leq\kappa\leq 2r\Big\}
=2\pi\omega_*(2\kappa-\omega_*)
\]
where $\omega_*$ is the unique solution in $(0,\kappa)$ of
\[
\frac{2\pi}{3}\omega^3-\pi\kappa\omega^2+1=0.
\]
In addition, we have the estimate $e_2\geq 6.64$.
\end{lem}

\begin{proof}
In this case, numerical exploration suggests that $e_2$ is obtained by considering arbitrarily large
core-shells, but this time with an optimal positive thickness $\omega=R-r$.\\
For $\omega>0$ fixed, we thus compute
\[
T_2^{\mathrm{min}}(\omega)\coloneqq \lim_{r\to\infty} T_{2}(r,r+\omega)
\]
then prove 
\[
e_2 = \inf_{\omega\in(0,\kappa]} T_2^{\mathrm{min}}(\omega)
\]
and finally determine this infimum.
First, using $R^3-r^3=\omega(R^2+Rr+r^2)$, we write
\[
\begin{aligned}
T_2(r,R)
&=\frac{3(r^2+R^2)}{R^3-r^3}+\frac{3\pi\kappa}{2}\left(\frac{(R^2-r^2)^2}{R^3-r^3}\right) -\frac{2\pi}{5}\frac{(R-r)^3}{R^3-r^3}(r^2+3rR+R^2)\\
&=\frac{1}{R^2+Rr+r^2}\left(
\frac{3(r^2+R^2)}{\omega}+\frac{3\pi\kappa}{2}(R+r)^2 \omega -\frac{2\pi}{5}(r^2+3rR+R^2)\omega^2 \right).
\end{aligned}
\]
For $\omega$ fixed, $R=r+w$ and $r\to+\infty$, we have $R=r+o(1)$ and find, by definition of
$T_2^{\mathrm{min}}$,
\[
T_2^{\mathrm{min}}(\omega)
=\lim_{r\to 0}~ \frac{1}{3r^2} \left( 6r^2\omega + 6\pi\kappa r^2\omega -2\pi r^2\omega^2\right)
=  \frac{2}{\omega} + 2\pi\kappa\omega -\frac{2\pi}{3}\omega^2.
\]
Then we compute the difference
\[
\begin{aligned}
&T_2(r,r+\omega)-T_2^{\mathrm{min}}(\omega)\\
&\quad
= \frac{1}{\omega}\left[\frac{3(r^2+R^2)}{r^2+Rr+R^2}-2\right]
+\frac{\pi\kappa\omega}{2}\left[\frac{3(r+R)^2}{r^2+Rr+R^2}-4\right]
-\frac{2\pi\omega^3}{15}\left[\frac{3(r^2+3rR+R^2)}{r^2+Rr+R^2}-5\right]\\
&\quad
=\frac{1}{r^2+Rr+R^2}\left[
\frac{1}{\omega}(r^2-2rR+R^2)-\frac{\pi\kappa}{2}\omega(r^2-2rR+R^2)
-\frac{2\pi}{15}\omega^2(-2r^2+4Rr-2R^2) \right]\\
&\quad
=\frac{1}{\omega(r^2+Rr+R^2)}\left(1-\frac{\pi\kappa}{2}\omega^2+\frac{4\pi}{15}\omega^3\right)\\
&\quad
\eqqcolon \frac{\omega\phi(\omega)}{R^2+Rr+r^2}.
\end{aligned}
\]
We now check that $\phi$ is nonnegative. We have $\phi'(\omega)=\pi\omega(4\omega^2/5-\kappa)$.
Recall that the set of constraints gives $\omega=R-r\leq\kappa$, so that $4\omega^2/5\leq
4\kappa^2/5<\kappa$ for $\kappa=11/10$. Thus, $\phi'$ is nonnegative on the relevant interval, and
$\phi$ is nonincreasing. Since, $\omega\leq\kappa$, we deduce
\[
\phi(\omega)\geq \phi(\kappa) = 1-\frac{\pi}{30}\kappa^3 \geq 0.024 > 0.
\]
As a consequence $T_2(r,r+\omega)-T_2^{\mathrm{min}}(\omega) \geq 0$, and since $\lim_{r\to\infty}
T_2(r,r+\omega)=T_2^{\min}(\omega)$, we obtain 
\[
e_2 = \inf_{\omega\in(0,\kappa]} T_2^{\mathrm{min}}(\omega).
\]
Next, to determine this infimum, we compute
\[
(T_2^{\mathrm{min}})'(\omega) =
\frac{-2g(\omega)}{\omega^2}
\quad\text{ with }\quad
g(\omega)\coloneqq\frac{2\pi}{3}\omega^3-\pi\kappa\omega^2+1
\quad\text{ and }\quad
g'(\omega) = 2\pi\omega(\omega-\kappa)\leq 0.
\]
Thus $g$ is nonincreasing, and $g(0)=1$, $g(\kappa)=1-\frac{\pi}{3}\kappa^3 <0$ for $\kappa=11/10$.
As a consequence, there exists a unique $\omega_*\in(0,\kappa)$ such that
\begin{equation}\label{defomstar}
g(\omega_*)
=\frac{2\pi}{3}\omega_*^3-\pi\kappa\omega_*^2+1
=0,
\end{equation}
and $T_2^{\mathrm{min}}$ is decreasing on $(0,\omega_*)$ and increasing on
$[\omega_*,\kappa)$, hence
\[
e_2 = T_2^{\mathrm{min}}(\omega_*).
\]
Using the relation \cref{defomstar} satisfied by $\omega_*$, we may rewrite
\[
T_2^{\mathrm{min}}(\omega_*) 
=\frac{2}{\omega_*}\left(1+\pi\kappa\omega_*^2-\frac{\pi}{3}\omega_*^3\right)
=2\pi\omega_*(2\kappa-\omega_*)
\eqqcolon h(\omega_*).
\]
There remains to estimate this quantity.
Notice that the function $h$ is strictly concave and maximal in $\kappa$. In particular, it is
increasing on $(0,\kappa)$. One can verify that $\omega_*> 0.71$ just by checking that
$g'(0.71)<0$, thus
\[
e_2 \geq h(0.71) \geq 6.64.
\]
This is actually a crude lower bound: numerically we see that $e_2= 6.668\pm 10^{-3}$.
\end{proof}

\begin{lem}\label{lem:T3}
We have
\[
e_3 \coloneqq \inf \Big\{ T_3(r,R) ~:~ 0\leq r<R \text{ and }2r \leq \kappa\leq R-r\Big\}
=\frac{3}{\kappa}+\frac{21\pi}{20}\kappa^2
= 6.71 \pm 10^{-2}.
\]
\end{lem}

\begin{proof}
We have
\[
T_3(r,R) = \frac{1}{R^3-r^3}\left[ 3(R^2+r^2)+\frac{\pi}{20}(\kappa^5-20
R^2\kappa^3+40R^3\kappa^2+32r^5-80r^3\kappa^2)\right].
\]
We show that the infimum under the constraints is in fact
\[
L\coloneqq T_3(0,\kappa) = \frac{3}{\kappa}+\frac{21\pi}{20}\kappa^2.
\]
This is equivalent to showing $T_3(r,R)\geq L$, which rewrites
\[
F(r,R)\coloneqq 3(R^2+r^2)+\frac{\pi}{20}(\kappa^5-20 R^2\kappa^3+40R^3\kappa^2+32r^5-80r^3\kappa^2)
-L(R^3-r^3) \geq 0.
\]
Computations gives
\[
F(r,R) =
(2\pi\kappa^2-L)R^3+(3-\pi\kappa^3)R^2+\frac{8\pi}{5}r^5+(L-4\pi\kappa^2)r^3+3r^2+\frac{\pi}{20}\kappa^5.
\]
We show that for $r$ fixed, $R\mapsto F(r,R)$ is increasing.
Differentiating, we find
\[
\partial_R F(r,R) =
R\left[\left(\frac{57\pi}{20}\kappa^2-\frac{9}{\kappa}\right)R+(6-2\pi\kappa^3)\right]
\eqqcolon R g(R).
\]
We check that for $\kappa=11/10$, $\frac{57\pi}{20}\kappa^2-\frac{9}{\kappa}$ and $6-2\pi\kappa^3$
are positive, so that $g$ is nondecreasing. Since the constraints give $R\geq r+\kappa\geq\kappa$
and $g(\kappa)=17\pi\kappa^3-3\geq 0.54$, we find that $\partial_R F(r,R)>0$ and thus $R\mapsto
F(r,R)$ is increasing. As a consequence, recalling $R\geq \kappa+r$ on the admissible set of
$(r,R)$, we deduce
\[
F(r,R)\geq F(r,r+\kappa),
\]
and we need only study the function of one variable $r\mapsto F(r,r+\kappa)$ and show that it is
nonnegative to conclude. After computation, $F(r,\kappa+r)$ simplifies tremendously to
\[
F(r,\kappa+r)
= \frac{8\pi}{5}r^5-2\pi\kappa^2r^3+\left(\frac{37\pi}{20}\kappa^3-3\right)r^2+\left(\frac{17\pi}{20}\kappa^4-3\kappa\right)r
\eqqcolon r h(r).
\]
Since the constraint on $r$ gives $r\leq\kappa/2$, we have $2\pi\kappa^2r^2\leq \pi\kappa^3 r$, so
that
\begin{multline*}
h(r)=\frac{8\pi}{5}r^4-2\pi\kappa^2r^2+\left(\frac{37\pi}{20}\kappa^3-3\right)r+\frac{17\pi}{20}\kappa^4-3\kappa\\
\geq \frac{8\pi}{5}r^4+\left(\frac{37\pi}{20}\kappa^3-3-\pi\kappa^3\right)r+\frac{17\pi}{20}\kappa^4-3\kappa
= \frac{8\pi}{5}r^4+\left(\frac{17\pi}{20}\kappa^3-3\right)(r+\kappa) \geq 0,
\end{multline*}
where for the inequality, we used the fact that $\frac{17\pi}{20}\kappa3>3$ for $\kappa=11/10$.
Hence, $F(r,R)\geq F(r,\kappa+r)\geq 0$, which implies $T_3(r,R)\geq T_3(0,\kappa)$ and concludes
the proof.
\end{proof}

\begin{lem}\label{lem:T4}
We have
\begin{multline*}
e_4\coloneqq\inf \Big\{ T_4(r,R) ~:~ 0\leq r<R\\
\text{ and } \big[(R\leq 3r \text{ and } 2r\leq\kappa\leq R+r)
\text{ or } (3r\leq R \text{ and } R-r\leq\kappa\leq R+r)\big]\Big\}
=f_4(\theta_*),
\end{multline*}
where 
\[
f_4(\theta)\coloneqq \frac{6}{\kappa}\left(\frac{1+\theta^2}{\theta(3+\theta^2)}\right)
+\frac{\pi\kappa^2}{5} \left(\frac{\theta}{3+\theta^2}\right)(\theta^3+5\theta^2-15\theta+30)
\]
and $\theta_*$ is the unique zero in the interval $(0,1)$ of 
\[
\theta\mapsto 14\theta^5+30\theta^4+60\theta^3+\left(60-\frac{120\pi}{\kappa^3}\right)\theta^2-270\theta+180 .
\]
In addition we have the estimate $e_4\in[6.6967,6.6986]$.
\end{lem}

\begin{proof}
This case is the most complicated of the six and is divided into three steps.\\[2pt]
We notice that the constraints on $r,R$ equivalently simplify to
\[
\Big\{ (r,R) \in \R^2 ~:~ 0\leq r\leq\kappa/2 \text{ and } 0<R-r\leq\kappa\leq R+r \Big\}.
\]
Indeed, the converse implication $\impliedby$ is immediate.
For the direct implication $\implies$, we verify:
if $R\leq 3r$ and $2r\leq\kappa\leq R+r$, then $R-r\leq 3r-r=2r\leq\kappa$;
if $3r\leq R$ and $R-r\leq\kappa\leq R+r$, then $2r=3r-r\leq R-r\leq\kappa$.\\
Looking at the function on the domain graphically seems to indicate that the minimum is reached on
the edge $R+r=\kappa$ of the domain, for a particular value of $R-r$. The change of
variables $\rho\coloneqq R+r$, $\omega\coloneqq R-r$ is thus particularly well-suited, and the set
of the admissible domain rewrites
\[
\Big\{ (\omega,\rho) \in \R^2 ~:~ 0<\omega\leq\kappa \text{ and } \kappa\leq\rho\leq\kappa+\omega
\Big\}.
\]
We rather work with fractions of $\kappa$, and thus write instead $\omega=\theta\kappa$ and
$\rho=x\kappa$, on the admissible set
\[
\Big\{ (\theta,x) \in \R^2 ~:~ 0<\theta\leq 1 \text{ and } 1\leq x\leq 1+\theta\Big\}.
\]
After computation, using in particular $R^3-r^3=\omega(R^2+Rr+r^2)$, we find
\[
\begin{aligned}
\mathtoolsset{multlined-width=0.85\displaywidth}
T_4(r,R)
&=\begin{multlined}[t]
\frac{3(R^2+r^2)}{R^3-r^3}
+\frac{\pi}{20(R^3-r^3)}\big[-\kappa^5+20R^2\kappa^3 -40 r^3\kappa^2\\
+30(R^2-r^2)\kappa-8R^5+40R^3r^2-40R^2r^3+40r^5\big]
\end{multlined}\\
&=\begin{multlined}[t]
\frac{1}{5\omega(3\rho^2+\omega^2)}\big[30(\rho^2+\omega^2)
+\pi (\rho^5-5\rho^4\omega+5\rho\omega^4+\omega^5+10\rho^2\omega^2(\rho-2\omega+3\kappa)\\
+5\kappa^2(\rho-\omega)^2(\kappa+\omega-\rho)-\kappa^5) \big]\end{multlined}\\
&=\begin{multlined}[t]
\frac{1}{5\theta(3x^2+\theta^2)}\left[
\frac{30}{\kappa}(x^2+\theta^2)
+\pi\kappa^2\big[x^5-5x^4\theta+5x\theta^4+\theta^5
+10x^2\theta^2(x-2\theta+3)\right.\\\left.\vphantom{\frac{30}{\kappa}(x^2+\theta^2)}
+5(x-\theta)^2(1+\theta-x)-1\big]\right].
\end{multlined}
\end{aligned}
\]
Abusing notation, we write $T_4(\theta,x)$ this quantity, and we now optimize it on the admissible
set. As remarked above, graphically, the minimum is supposedly reached for $x=1$ and some
$\theta\in(0,1)$. The rest of the proof is divided into steps.\\[2pt]
\textit{Step 1.} For fixed $\theta$, we show that $x\mapsto T_4(\theta,x)$ is nondecreasing and then
nonincreasing on $x\in[1,1+\theta]$, which implies
\[
\min_{x\in [1,1+\theta]} T_4(\theta,x) = \min(T_4(\theta,1),T_4(\theta,1+\theta)).
\]
In order to prove this, we study the sign of $\partial_x T_4$, and we have, after computation
\begin{equation}\label{T4:eq1}
\begin{aligned}
\mathrm{sign}\,(\partial_xT_4)
= \mathrm{sign} [ 5\theta (3x^2+\theta^2)^2\partial_xT_4]\\
= \pi\kappa^2 \left (P_\theta(x)- cx\theta^2 \right)
\end{aligned}
\end{equation}
where $c\coloneqq\frac{120}{\pi\kappa^3}$ and
\begin{multline*}
P_\theta(x)\coloneqq 9x^6-30\theta x^5+(35\theta^2-15)x^4
-20\theta^3x^3+15\theta(\theta^3+2\theta+2)x^2\\
+(6-20\theta^2+60\theta^4-46\theta^5)x
+5\theta^6-15\theta^4-10\theta^3.
\end{multline*}
We look for the sign of $P_\theta-c_\theta$ on $[1,1+\theta]$. We rather study a function on $[0,1]$
and thus write $x=(1+\theta)-\theta u$ (equivalently $u=\frac{1+\theta-x}{\theta}$), where $u\in[0,1]$.
After substitution, abusing notation we find
\begin{multline*}
P_\theta(u)
= \theta^2\Big( 45u^2+30u^2(3-4u)\theta
+60(1+u^2-3u^3+2u^4)\theta^2\\
+(28-60u-80u^3+120u^4-54u^5)\theta^3
+(-32+32u+20u^4-24u^5+9u^6)\theta^4\Big).
\end{multline*}
Notice the $\theta^2$ appearing in $P_\theta$ and in front of $c$ in \cref{T4:eq1}, so that
\begin{equation}\label{T4:eq2}
\begin{aligned}
\mathrm{sign}\,(\partial_xT_4)
&=\mathrm{sign}\, \Big[ \big[45u^2+30u^2(3-4u)\theta
+60(1+u^2-3u^3+2u^4)\theta^2\\
&\qquad\qquad
+(28-60u-80u^3+120u^4-54u^5)\theta^3\\
&\qquad\qquad
+(-32+32u+20u^4-24u^5+9u^6)\theta^4\big]
- c(1+\theta-\theta u)\Big]\\
&\eqqcolon \mathrm{sign}\, Q_\theta(u).
\end{aligned}
\end{equation}
Sorting $Q_\theta$ in ascending powers of $u$, we have
\begin{align*}
Q_\theta(u)
&= 60\theta^2+28\theta^3-32\theta^4
+-(60\theta^3-32\theta^4)u
+(45+90\theta+60\theta^2)u^2
-(120\theta+180\theta^2+80\theta^3) u^3\\
&\qquad\qquad
+(120\theta^2+120\theta^3+20\theta^2)u^4
-(54\theta^3+24\theta^4)u^5
+9\theta^4u^6
- c(1+\theta-\theta u)\\
&\eqqcolon R_\theta(u)-c(1+\theta-\theta u).
\end{align*}
We claim the following:
\begin{enumerate}[(i)]
\item We have $Q_\theta(0)<0$.\\Indeed, for $\theta=1$, we have $Q_1(0) = 56-2c
=56-\frac{240}{\pi\kappa^3} \leq -1.4<0$ for $\kappa=11/10$, and we write
\[
Q_\theta(0) = Q_\theta(0)-Q_1(0)+Q_1(0)
=(\theta-1)p(\theta)+Q_1(0)
\]
where $p(\theta)\coloneqq -32\theta^3-4\theta^2+56\theta+56-c$. Notice that
$p'(\theta)=-96\theta^2-8\theta+56$ is a trinomial with $p'(0)=56$, $p'(1)=-48$, so that $p'$ is
positive and then negative on $[0,1]$. Consequently, $p$ is increasing and then decreasing on
$[0,1]$ thus $p(\theta)\geq \min (p(0),p(1)) = \min(56-c,66-c)>0$. Since $\theta-1<0$ on $(0,1)$, from the
previous equation we deduce $Q_\theta(0) \leq Q_1(0)< 0$.
\item We have $Q_\theta(1)>0$.\\Indeed, we compute $Q_1(\theta)=45-c-30\theta+60\theta^2-46\theta^3+5\theta^4$,
and we check $Q_1(1) = 34-c = 34-\frac{120}{\pi\kappa^3}\geq 5.3$. Then we write
\[
Q_1(\theta)=Q_1(\theta)-Q_1(1)+Q_1(1)
= (\theta-1)q(\theta)+Q_1(1).
\]
where $q(\theta)\coloneqq 5\theta^3-41\theta^2+19\theta-11$. Notice that
$q'(\theta)=15\theta^2-82\theta+19$ is a trinomial which is minimal in $\theta=82/30>1$, so that
$q'$ is decreasing on $[0,1]$. We check $q'(0)=19$ and $q'(1)=-48$ so that $q'$ is positive and then
negative on $[0,1]$. Thus $q$ is decreasing then increasing on $[0,1]$ and $q(\theta)\leq
\max(q(0),q(1))=\max(-11,28)<0$. Since $\theta-1<0$ on $(0,1)$, from the previous equation we deduce
$Q_\theta(1)\geq Q_1(1)>0$.
\item There exists a unique $u_*=u_*(\theta)\in(0,1)$ such that $Q_\theta(u_*)=0$.\\The existence comes
from the two previous points, so there remains to check the uniqueness. Recall that
$Q_\theta(u)=R_\theta(u)-c(1+\theta-\theta u)$ so that, setting $S_\theta(u) \coloneqq
\frac{R_\theta(u)}{1+\theta-\theta u}$, we have
\[
Q_\theta(u) = 0 \iff S_\theta(u) = c.
\]
We have seen that $S_\theta(0)<c$ and $S_\theta(1)>c$, and we know study the variations of
$S_\theta$. Differentiating, after computation we obtain
\[
\begin{aligned}
S_\theta'(u)
&= R_\theta'(u)(1+\theta-\theta u)+\theta R_\theta(u)\\
&=u\Big[
90+(270-405 u)\theta+(300-990u+720u^2)\theta^2\\
&\qquad\qquad
+(120-840u+1320u^2-630u^3)\theta^3\\
&\qquad\qquad
+(-240u+720u^2-750u^3+270u^4)\theta^4+(80u^2-180u^3+150u^4-45u^5)\theta^5 \Big]\\
&\eqqcolon u U_\theta(u).
\end{aligned}
\]
Differentiating and regrouping by powers of $\theta$, we find
\begin{align*}
U_\theta'(u)
&= \theta\big[-405+(1440u-990)\theta +(-1890u^2+2640u-840)\theta^2\\
&\qquad\qquad
+(1080u^3-2250u^2+1440u-240)\theta^3
+(-225u^4+600u^3-540u^2+160u)\theta^4\big]\\
&\eqqcolon \theta V_\theta(u).
\end{align*}
Differentiating once more, we find
\[
\mathtoolsset{multlined-width=0.8\displaywidth}
\begin{aligned}
V_\theta'(u)
&=\begin{multlined}[t]
1440\theta+(2640-3780u)\theta^2
+(3240u^2-4500u+1440)\theta^3\\
+(-900u^3+1800u^2-1080u+160)\theta^4
\end{multlined}\\
&\eqqcolon A\theta + B\theta^2+C\theta^3+D\theta^4.
\end{aligned}
\]
Since $u\in[0,1]$, we estimate from below $B\geq -1140$. For $C$, we use the fact that it is a
trinomial with positive quadratic coefficient, minimal at $u_0=4500/6480$, which gives
\[
C\geq 3240 u_0^2-4500 u+1440 = \frac{245}{2}>0.
\]
Eventually, the derivative of $D$ vanishes when $15u^2+20u-6=0$, a trinomial whose roots are
$u_{\pm} = \frac{-10\pm\sqrt{90}}{15}$. Since $u_-<0$ and $u_+ \simeq 0.252\in (0,1)$, in the
interval $[0,1]$, $D$ is minimal at $u_+$ from which we compute the lower bound $D\geq -44$. Using
the fact the $\theta^p\leq\theta$ for $p>1$ and discarding the positive term $C$, we deduce
\[
V_\theta'(u) \geq (1440-1140-44)\theta > 0.
\]
Thus, $V_\theta$ is increasing and
\[
V_\theta(1) = -405+450\theta-90\theta^2+30\theta^3-5\theta^4.
\]
Differentiating w.r.t. $\theta$ we find
\[
\partial_\theta V_\theta(1) = -20\theta^3+90\theta^2-180\theta+450 \geq 450-180-80 > 0,
\]
and thus $V_\theta(1) \leq V_1(1) = -20$. Since $V_\theta$ is increasing on $[0,1]$ this implies
that it is in nonpositive, and recalling that $U_\theta' =\theta V_\theta$, so is $U_\theta'$. As a
consequence, $U_\theta$ is nonincreasing. We then compute
\[
U_\theta(0) = 90+270\theta+300\theta^2+120\theta^3 >0.
\]
There are two cases then:
\begin{itemize}
\item Either $U_\theta(1)\geq 0$, in which case $U_\theta$ is nonnnegative on $[0,1]$, and consequently
$S_\theta'(u)=u U_\theta(u)\geq 0$. This implies that $S_\theta$ is nondecreasing, and thus it reaches
the value $c$ exactly once on the interval $(0,1)$ (recall $S_\theta(0)<c$ and $S_\theta(1)>c$).
\item Or $U_\theta(1)< 0$, in which case $U_\theta$ is nonnegative and then nonpositive on $[0,1]$, and
similarly for $S_\theta'$. This implies that $S_\theta$ is nondecreasing and then nonincreasing on
$[0,1]$. Recalling $S_\theta(0)<c$ and $S_\theta(1)>c$, this also shows that $S_\theta$ reaches the
value $c$ exactly once on $(0,1)$.
\end{itemize}
In both cases, we get that there exists a unique $u_*\in(0,1)$ such that $S_\theta(u_*)=c$, that is,
$Q_\theta(u_*)=0$.
\end{enumerate}
Recalling \cref{T4:eq2} and $x=(1+\theta)-\theta u$, by the three previous points we see that
$\partial_x T_4$ is negative at $u=0$ (that is, $x=1+\theta$), positive at $u=1$ (that is, $x=1$), and
vanishes only once in the interval $(1,1+\theta)$. By consequence, $x\mapsto T_4(\theta, x)$ is
increasing then decreasing on this interval, which concludes Step~1.\\[2pt]
\textit{Step 2.} We determine and estimate $\inf_{\theta\in (0,1]} T_4(\theta,1)$.
After computation, we find
\[
T_4(\theta,1) = \frac{6}{\kappa}\left(\frac{1+\theta^2}{\theta(3+\theta^2)}\right)
+\frac{\pi\kappa^2}{5} \left(\frac{\theta}{3+\theta^2}\right)(\theta^3+5\theta^2-15\theta+30).
\]
Differentiating in $\theta$ we get
\[
\begin{aligned}
\partial_\theta T_4(\theta,1)
&= \frac{\pi\kappa^2}{5\theta^2(3+\theta^2)^2}
\left(\theta^2
(2\theta^5+5\theta^4+12\theta^3+15\theta^2-90\theta+90)-\frac{30}{\pi\kappa^3} (\theta^4+3)\right)\\
&\eqqcolon \frac{\pi\kappa^2 W(\theta)}{5\theta^2(3+\theta^2)^2}
\end{aligned}
\]
and, setting $a\coloneqq\frac{120\pi}{\kappa^3}$, we have
\[
W'(\theta) = \theta
\left(14\theta^5+30\theta^4+60\theta^3+60\theta^2-270\theta+180
-a \theta^2
\right)
\eqqcolon \theta\big(M(\theta) -a\theta^2\big).
\]
We claim that $W'$ is nonnegative.  As one can check that $a\leq 29$, it is enough to show that
\[
N(\theta)\coloneqq M(\theta)-29\theta^2
=14\theta^5+30\theta^4+60\theta^3+31\theta^2-270\theta+180 \geq 0
\]
for all $\theta\in[0,1]$.
For $\theta \in [0,2/3]$, we verify $31\theta^2-270\theta+180>0$, so that
\[
N(\theta)\geq 14\theta^5+30\theta^4+60\theta^3\geq 0.
\]
For $\theta\in[2/3,1]$, we let $y=1-\theta\in[0,1/3]$ and compute
\[
N(\theta)=45-162y+531y^2-320y^3+100y^4-14y^5.
\]
Since $y\in[0,1/3]$, we have $-320y^3-14y^5\geq -\frac{320}{3}y^2-\frac{14}{27}y^2$, which yields
\[
N(\theta)\geq 45-162y+\left(531-\frac{320}{3}-\frac{14}{27}\right)y^2+100y^4
\geq 45-162y+423y^2 \geq 0.
\]
Thus, $W'$ is nonnegative and $W$ is nondecreasing. Since $W(0)=-90\pi/\kappa^3$ and
$W(1)=34-120\pi/\kappa^3\geq 22.1$, we deduce that $W$ there exists a unique $\theta_*\in(0,1)$ such
that $W(\theta_*)=0$. Thus, $\theta\mapsto T_4(\theta,1)$ is nonincreasing on $(0,\theta_*]$ and
nondecreasing on $[\theta_*,1]$, so that it is minimal at $\theta_*$. Hence, we have
\begin{equation}\label{T4_min1}
\inf_{\theta\in (0,1]} T_4(\theta,1) = T_4(\theta_*,1)
\end{equation}
where $\theta_*$ is the unique solution in $(0,1)$ of
\begin{equation}\label{T4_defthetastar}
W(\theta_*) \overset{\mathrm{def}}{=} 14\theta_*^5+30\theta_*^4+60\theta_*^3+60\theta_*^2-270\theta_*+180
-\frac{120\pi}{\kappa^3} \theta_*^2 = 0.
\end{equation}
Numerically, we find $\theta_*\simeq 0.84381$, $T_4(\theta_*,1)\simeq 6.6977$ and it is easy to find lower and upper bounds.
Indeed, since $W$ is monotonous, we can check that $W(0.8438)\cdot W(0.8439)<0$ so that
$\theta_*\in(0.8438,0.8439)\eqqcolon(\underline{\theta},\overline\theta)$.
Then we may use basic interval arithmetics to deduce
\[
T_4(\theta_*,1) \leq \frac{6}{\kappa}\left(\frac{1+\overline\theta^2}{\underline\theta(3+\underline\theta^2)}\right)
+\frac{\pi\kappa^2}{5} \left(\frac{\overline\theta}{3+\underline\theta^2}\right)\left(\overline\theta^3+5\overline\theta^2
+15(1-\underline\theta) +15\right)
\leq 6.6986.
\]
and similarly, swapping $\overline\theta$ and $\underline\theta$, $T_4(\theta_*,1)\geq 6.6967$.
~\\This concludes Step 2.\\[2pt]
\textit{Step 3.} We show $\inf_{\theta\in (0,1]} T_4(\theta,1+\theta) > \inf_{\theta\in (0,1]} T_4(\theta,1)$.\\
After computation, we find
\[
T_4(\theta,1+\theta)=\frac{6}{\kappa}\left(\frac{2\theta^2+2\theta+1}{4\theta^2+6\theta+3}\right)
+\frac{\pi\kappa^2}{5}\left(\frac{-8\theta^3+10\theta^2+50\theta+30}{4\theta^2+6\theta+3}\right).
\]
In view of the previous step, we need only establish $T_4(1+\theta,\theta)\geq
6.7$. Using the lower bound $\frac{\pi\kappa^2}{5}\geq \frac{19}{25}$ for $\kappa=11/10$, we find
\[
\begin{aligned}
T_4(1+\theta,\theta)-\frac{67}{100}
&\geq
\frac{60}{11\theta}\left(\frac{2\theta^2+2\theta+1}{4\theta^2+6\theta+3}\right)
+\frac{19}{29}\left(\frac{\theta(30+50\theta+10\theta^2-8\theta^3)}{4\theta^2+6\theta+3}\right)\\
&= \frac{N(\theta)}{550\theta(4\theta^2+6\theta+1)},
\end{aligned}
\]
where
\[
N(\theta)\coloneqq 3000-5055\theta-3570\theta^2+6160\theta^3+4180\theta^4-3344\theta^5.
\]
We crudely estimate $N(\theta)\geq 100 Q(\theta)$ with $Q(\theta)\coloneqq
30-51\theta-36\theta^2+61\theta^3+41\theta^4-34\theta^5$.
On the interval $[0,1/2]$ we study the derivative
\[
Q'(\theta) = -170\theta^4+164\theta^3+183\theta^2-72\theta-51
\]
and we estimate, using $\theta\leq 1/2$,
\[
Q'(\theta) \leq \left(\frac{164}{4}+\frac{183}{2}-72\right)\theta-51
= \frac{121}{2}\theta -51\leq \frac{121}{4}-51<0.
\]
As a consequence $Q$ is nonincreasing on $[0,1/2]$ but $Q(1/2)=\frac{4625}{1000}$ is still
positive, so that $Q$ is positive on $[0,1/2]$.
To estimate $Q$ on $[1/2,1]$, we check that it is convex, this time, using $\theta\geq1/2$:
\[
Q''(\theta) = -680\theta^3+492\theta^2+366\theta-72
\geq (492-680)\theta^2 + 366\theta-72
\geq -172\theta+366\theta-72
\geq  27.
\]
Hence, the graph of $Q$ on $[1/2,1]$ lies above its tangents. Taking the tangent at $\theta=7/10$,
we find, for $\theta\geq 1/2$,
\[
Q(\theta)
\geq Q\left(\frac{7}{10}\right)+Q'\left(\frac{7}{10}\right)\left(\theta-\frac{7}{10}\right)
\geq Q\left(\frac{7}{10}\right)+Q'\left(\frac{7}{10}\right)\left(\frac{1}{2}-\frac{7}{10}\right)
\geq 0.97.
\]
In the end, we deduce that $Q$ is nonnegative on $[0,1]$, which implies $N\geq 0$ and consequently
\[
T_4(1+\theta,\theta)\geq 6.7.
\]
This concludes Step~3.\\[2pt]
\textit{Conclusion.}
By Step~1, we have $e_4\geq \inf_{\theta\in(0,1]} \min(T_4(\theta,1),T_4(1,1+\theta))$. Step~2 gives
the upper bound $\inf_{\theta\in(0,1]} T_4(\theta,1)< 6.7$ while Step~3 gives the lower bound
$\inf_{\theta\in(0,1]} T_4(1,1+\theta)\geq 6.7$, thus
\[
e_4 = \inf_{\theta\in(0,1]} T_4(\theta,1).
\]
Step~2 describes this minimum in \cref{T4_min1}, \cref{T4_defthetastar} and gives $e_4\in
[6.6967,6.6986]$. This concludes the proof.
\end{proof}

\begin{lem}\label{lem:T5}
We have
\[
e_5\coloneqq \inf \Big\{ T_5(r,R) ~:~ 0\leq r<R\text{ and }R+r \leq \kappa \leq 2R\Big\}
=T_5(0,R_*)
\]
where $R_*=\kappa/(2\lambda_{*,1,\kappa})$ and
\[
\lambda_{*,1,\kappa} = \sqrt{\frac{5}{3}-\frac{5}{\pi\kappa^3}}
\]
is given by \cref{prp_eballriesztrunc}.
In addition, we have the approximation $e_5= 6.619\pm 10^{-3}$.
\end{lem}

\begin{proof}
In that case, we will see that the infimum is reached for $r=0$ and some positive $R_*$, that is, for
a ball. We have
\[
T_5(r,R)=\frac{3(R^2+r^2)}{R^3-r^3}+\frac{\pi}{20(R^3-r^3)}\big[\kappa^5-20R^2\kappa^3+40R^3\kappa^2+32r^5
-16r^3(5R^2-r^2)\big].
\]
Since the constraints give $\kappa/2\leq R$ and $r<R$, we write $\kappa/2=\lambda R$ and $r=tR$, with
$t\in [0,1)$ and $\lambda\in[1/2,1]$ such that $2\lambda\geq 1+t$ (coming from $\kappa\geq R+r$). Abusing notation, we compute
\[
\begin{aligned}
T_5(t,\lambda)
&= \frac{3}{R}\left(\frac{1+t^2}{1-t^3}\right)+\frac{4\pi R^2}{5(1-t^3)}
\big[2\lambda^5-10\lambda^3+10\lambda^2+3t^5-5t^3\big]\\
&=A\lambda\left(\frac{1+t^2}{1-t^3}\right)+\frac{B}{\lambda^2(1-t^3)}
\big[2\lambda^5-10\lambda^3+10\lambda^2+3t^5-5t^3\big].
\end{aligned}
\]
where we have set $A\coloneqq 6/\kappa$ and $B\coloneqq \pi\kappa^2/5$.
We claim that $T_5(t,\lambda)\geq T_5(0,\lambda)$ for all admissible $t$, $\lambda$.
Indeed, after computation we find
\[
T_5(t,\lambda)\geq T_5(0,\lambda)
\iff
t^2\left[A\lambda(1+t)+Bt\left(2\lambda^3-10\lambda+10-\frac{5}{\lambda^2}\right)+\frac{3Bt^3}{\lambda^2}\right]
\geq 0.
\]
Dismissing the nonnegative term $\frac{3Bt^3}{\lambda^2}$, it is enough to show
\[
P(t,\lambda)\coloneqq A\lambda(1+t) +Bt\left(2\lambda^3-10\lambda+10-\frac{5}{\lambda^2}\right)\geq 0.
\]
Observe that $P$ is an affine function in $t\in[0,2\lambda-1)$, with $P(0,\lambda) = A\lambda\geq 0$
and, setting $\alpha=A/B = 30/(\pi\kappa^3)$,
\[
P(2\lambda-1,\lambda)
=\frac{B}{\lambda^2} \big[
4\lambda^6-2\lambda^5+2(\alpha-10)\lambda^4
+30\lambda^3-10\lambda^2-10\lambda+5\big]
\eqqcolon \frac{BQ(\lambda)}{\lambda^2}.
\]
Discarding some high order nonnegative terms and then using $(\alpha-10)\lambda^4\geq
(\alpha-10)\lambda^3$ since $\alpha-10<0$, we estimate
\[
Q(\lambda) \geq 2(\alpha-10)\lambda^4+30\lambda^3-10\lambda^2-10\lambda+5
\geq (10+2\alpha)\lambda^3-10\lambda^2-10\lambda+5
\eqqcolon S(\lambda).
\]
Differentiating, we have $S'(\lambda)=6(5+\alpha)\lambda^2-20\lambda-10$, a trinomial with real roots
\[
\lambda_1 = \frac{20-\sqrt{1600+240\alpha}}{60+12\alpha} < 0
\quad\text{ and }\quad
\lambda_2 = \frac{20+\sqrt{1600+240\alpha}}{60+12\alpha} \in (1/2,1).
\]
Thus, $S$ is nonincreasing on $(1/2,\lambda_2)$, nondecreasing on $(\lambda_2,1)$ and thus minimal in
$\lambda_2$, where $S(\lambda_2)\geq 0.51$. As a consequence, $Q(\lambda)$ is nonnegative and so is
$P(2\lambda-1,\lambda)$. Since $t\mapsto P(t,\lambda)$ is affine and positive in $0$ and $2\lambda-1$,
it is positive for all $t\in [0,2\lambda-1]$, hence $T_5(t,\lambda)\geq T_5(0,\lambda)$.
Next, we compute
\[
T_5(0,\lambda) = B(2\lambda^3+(\alpha-10)\lambda+10),
\]
Its derivative in $\lambda$ is a trinomial with roots $\lambda_{\pm}=\pm
\sqrt{\frac{10-\alpha}{6}}\simeq \pm 0.69$, thus on $(1/2,1)$, $\lambda\mapsto T_5(0,\lambda)$ is
minimal at $\lambda_+$.
Notice that
\[
\lambda_+^2 = \frac{5}{3}-\frac{5}{\pi\kappa^3} =\lambda_{*,1,\kappa}^2,
\]
where $\lambda_{*,1,\kappa}$ is given by Proposition 3.3 of the paper.
In the end, recalling that $t=0$ and $\lambda=\lambda_+=\lambda_{*,1,\kappa}$ correspond to inner radius $r=0$ and
outer radius $R_*\coloneqq\kappa/(2\lambda_{*,1,\kappa})$ we find
\[
e_5 = T_5(0,R_*) = \frac{\mathcal{E}_{R_{1,\kappa}}(B_{R_*})}{|B_{R_*}|} =
\rho^{3,\ball}_{R_{1,\kappa}}.
\]
We may evaluate $e_5$ numerically to provide the approximation of the statement. This concludes the
proof.
\end{proof}

\begin{lem}\label{lem:T6}
We have
\[
e_6\coloneqq \inf \Big\{ T_6(r,R) ~:~ 0\leq r<R\text{ and }\kappa\geq 2R\Big\}
= T_6(0,\kappa/2) = \frac{6}{\kappa}+\frac{2\pi\kappa^2}{5} = 6.97\pm 10^{-2}.
\]
\end{lem}

\begin{proof}
In this case the infimum is also reached for some $r=0$ and positive $R$.
We have
\[
T_6(r,R) = \frac{3(R^2+r^2)}{R^3-r^3}+\frac{4\pi R^2}{5}\big(2R^5-5R^2r^3+3r^5\big).
\]
We let $r=tR$ with $t\in[0,1)$, and with a slight abuse of notation, we write
\[
T_6(t,R) = \frac{3}{R} A(t) + \frac{4\pi}{5} R^2 B(t)
\]
where
\[
A(t)\coloneqq \frac{1+t^2}{1-t^3} >0
\qquad\text{ and }\qquad
B(t) = \frac{3t^5-5t^3+2}{1-t^3}.
\]
Setting $N(t)=3t^5-5t^3+2$, we see that $N'(t) = 15t^2(t^2-1)\leq 0$, so that $N$ is nonincreasing. Since
$N(1)=0$, $N$ is positive on $[0,1)$ and so is $B$.
Since $A$ and $B$ are nonnegative, for $t$ fixed, we see that $R\mapsto T_6(t,R)$ is a convex
function, minimal in $R=R(t)$ such that
\[
-\frac{3}{R(t)^2} A(t)+\frac{8\pi}{5} R(t)^3 B(t) = 0
\iff R(t) = \left(\frac{15}{8\pi} \frac{A(t)}{B(t)}\right)^{1/3}.
\]
However, recall that we minimize $T_6$ under the constraint $R\leq\kappa/2$.
Notice that $A$ is an increasing function. We claim that $B$ is nonincreasing. Indeed, we have
\[
B'(t) = \frac{3t^2(1-t)(2t^4+2t^3+2t^2-3t-3)}{(1-t^3)^2} \eqqcolon \frac{3t^2(1-t)P(t)}{(1-t^3)^2},
\]
and we estimate
\[
P(t) \leq 6t^2-3t-3 = 3(t-1)(2t+1)\leq 0.
\]
Since $A$ is nondecreasing and $B$ is nonincreasing, $R$ is nondecreasing. Thus,
for all $t\in[0,1)$, $R(t)\geq R(0) = \left(\frac{15}{8\pi} \frac{A(0)}{B(0)}\right)^{1/3} =
\left(\frac{15}{16\pi}\right)^{1/3} >\kappa/2$.
Since $R\mapsto T_6(t,R)$ is convex and minimal in $R(t)$, it is nonincreasing on
$(0,\kappa/2)\subset (0,R(t))$, hence 
\[
\inf_{R\in [0,\kappa/2]} T_6(t,R) = T_6(t,\kappa/2)
= \frac{6}{\kappa} A(t) +\frac{\pi\kappa^2}{5} B(t).
\]
Differentiating, since the denominator $(1-t^3)^2$ is positive, we have
\[
\mathrm{sign} \partial_t T_6(t,\kappa/2)
=\mathrm{sign}\big[10(t^3+3t+2)-\pi\kappa^3 (2t^5-5t^2+3)\big]
\eqqcolon\mathrm{sign} Q(t).
\]
On $[0,1)$, $S(t) \coloneqq 2t^5-5t^2+3\leq 3$ since $2t^5-5t^2=t^2(2t^3-5)\leq 0$, thus using the
fact that $20-3\pi\kappa^3> 0$ for $\kappa=11/10$,
\[
Q(t) \geq 10(t^3+3t+2)-3\pi\kappa^3
= 10t^3+30t+(20-3\pi\kappa^3) > 0.
\]
As a consequence, $t\mapsto T_6(t,\kappa/2)$ is increasing, hence
\[
e_6 = T_6(0,\kappa/2) = \frac{6}{\kappa}+\frac{2\pi\kappa^2}{5} = 6.97\pm 10^{-2}.
\]
This concludes the proof.
\end{proof}

Gathering the six previous lemmas, we see that for $\kappa=11/10$, we have
\[
\rho^{3,\mathrm{shell}}_{R_{1,\kappa}}
= \min_{1\leq i\leq 6} e_i = e_5 = \rho^{3,\mathrm{ball}}_{R_{1,\kappa}}
\]
which proves \cref{prp:rieszcut_shellball}. As pointed out at the beginning this section, this
implies \cref{mainthm:thm2}.


\section{Yukawa potentials}\label{S_yukawa}

In this section we treat the Yukawa case of \cref{mainthm:thm1}. We assume that $n=3$.

\subsection{Energy-per-mass ratio of balls}

We begin by computing the interaction energy on 1D slices as defined in~\cref{eq_defSG}.

\begin{lem}\label{lem_SYukawatrunc}
For every $\kappa>0$ and every $L>0$, we have
\[
S_{Y_{1,\kappa}}(L)
=2\kappa^3 \lt(\frac{L}{\kappa}-2+\lt(\frac{L}{\kappa}+2\rt)e^{-L/\kappa}\rt).
\]
\end{lem}

\begin{proof}
We start from the definition~\cref{eq_defSG} of $S_{Y_{1,\kappa}}$ and we compute
\begin{multline*}
S_{Y_{1,\kappa}}(L) 
=\int_0^L \int_0^L |s-t|e^{-|s-t|/\kappa}\,ds\, dt
=2\int_0^L  \int_{0}^{t} |r|e^{-|r|/\kappa}\,dr\,dt\\
=2\kappa^3 \int_0^{L/\kappa}1-e^{-t}(1+t)\, dt
=2\kappa^3 \lt(\frac{L}{\kappa}-2+\lt(\frac{L}{\kappa}+2\rt)e^{-L/\kappa}\rt).
\end{multline*}
One can check that $S_{Y_{1,\kappa}}(L)\stackrel{L\dw0}{\sim} L^3/3$, which is consistent with
$S_{R_{1}}(L)=L^3/3$.
\end{proof}

We deduce the Yukawa energy of balls.

\begin{lem}\label{lem_Yukawa_nrj_of_balls}
For every $R>0$, we have
\[
\I_{Y_{1,\kappa}}(B_R)
= R^52^6\pi^2 \lambda^2\left(
\frac{1}{3}-\lambda(1-4\lambda^2)-\lambda(4\lambda^2+4\lambda+1)e^{-1/\lambda} \right),
\]
where $\lambda=\kappa/(2R)$.
\end{lem}

\begin{proof}
By \cref{cor:rieszball} we have
\begin{align*}
\I_{Y_{1,\kappa}}(B_R)
&= 4\pi^2\int_0^R S_{Y_{1,\kappa}}\big(2\sqrt{R^2-r^2}\big)r\,dr\\
&= 8\pi^2\kappa^3 \int_0^R 
\left[\frac{\sqrt{1-(r/R)^2}}{\kappa/(2R)}-2+\left(\frac{\sqrt{1-(r/R)^2}}{\kappa/(2R)}+2\right)e^{-\frac{\sqrt{1-(r/R)^2}}{\kappa/(2R)}}\right]
r\,dr\\
&= R^52^6\pi^2 \int_0^1 
\left(\frac{\kappa}{2R}\right)^3\left[\frac{\sqrt{1-r^2}}{\kappa/(2R)}
-2+\left(\frac{\sqrt{1-r^2}}{\kappa/(2R)}+2\right)e^{-\frac{\sqrt{1-r^2}}{\kappa/(2R)}}\right] r\,dr\\
&= R^5 2^6\pi^2\int_0^1 
\lambda^3\left[\frac{\sqrt{1-r^2}}{\lambda}-2+\left(\frac{\sqrt{1-r^2}}{\lambda}+2\right)e^{-\frac{\sqrt{1-r^2}}{\lambda}}\right]
r\,dr
\end{align*}
where we have set $\lambda=\kappa/(2R)$.
We compute
\[
\int_0^1 r\sqrt{1-r^2} \,dr = \left[-\frac{1}{3}(1-r^2)^{3/2}\right]_0^1=\frac{1}{3},
\]
then
\begin{align*}
\int_0^1 r e^{-\frac{\sqrt{1-r^2}}{\lambda}}\,dr
=\int_0^{\pi/2} (\sin\theta\cos\theta) e^{-\frac{\cos\theta}{\lambda}}\,d\theta
= \left[\lambda e^{-\frac{\cos\theta}{\lambda}}(\lambda+\cos\theta)\right]_0^{\pi/2}
=\lambda^2-\lambda(1+\lambda) e^{-\frac{1}{\lambda}}
\end{align*}
and finally using the change of variable $r=\sin\theta$ we compute 
\begin{multline*}
\int_0^1 r\sqrt{1-r^2}e^{-\frac{\sqrt{1-r^2}}{\lambda}}\,dr
=\int_0^{\pi/2} (\sin\theta\cos^2\theta)e^{-\frac{\cos\theta}{\lambda}}\,d\theta\\
=\left[\lambda\left(2\lambda^2+2\lambda\cos\theta+\cos^2\theta\right)e^{-\frac{\cos\theta}{\lambda}}\right]_0^{\pi/2}
=2\lambda^3-\lambda(2\lambda^2+2\lambda+1)e^{-\frac{1}{\lambda}}.
\end{multline*}
These identities lead to
\begin{align*}
\I_{Y_{1,\kappa}}(B_R)
&=R^52^6\pi^2
\lambda^3\left(
\frac{1}{3\lambda}-1+2\lambda^2-2\lambda(1+\lambda)e^{-1/\lambda}+2\lambda^2-(2\lambda^2+2\lambda+1)e^{-1/\lambda}
\right)\\
&=R^52^6\pi^2
\lambda^2\left( \frac{1}{3}-\lambda(1-4\lambda^2)-\lambda(4\lambda^2+4\lambda+1)e^{-1/\lambda}
\right),
\end{align*}
as claimed.
\end{proof}
\begin{rmk}
One can check that 
\[
\left( \frac{1}{3}-\lambda(1-4\lambda^2)-\lambda(4\lambda^2+4\lambda+1)e^{-1/\lambda} \right)
\stackrel{\lambda\up\oo}\sim \frac{1}{30\lambda^2}.
\]
We deduce 
\[
\lim_{\kappa\to\infty} \I_{Y_{1,\kappa}}(B_R) = R^5 \left(\frac{2^5\pi^2}{15}\right),
\]
which is consistent with
\[
\mathcal{R}_1(B_R)
= R^5\left(\frac{2^3 4\pi 2\pi}{6}\right)\int_0^{\pi/2} \cos^4\theta \sin\theta\,d\theta
= R^5\left(\frac{2^5\pi^2}{15}\right).
\]
\end{rmk}

\begin{prp}\label{prp_eballyukawa}
For any $\kappa>0$, we have
\be\label{prp_eballyukawa_1}
\rho^{3,\ball}_{Y_{1,\kappa}}= f_\kappa(\lambda_*)
\ee
where $f_\kappa:[0,+\infty)\to (0,+\infty)$ is defined by
\[
f_\kappa(\lambda)\coloneqq
\begin{cases}
\displaystyle
\frac{6\lambda}{\kappa}+4\pi
\kappa^2\big(1-3\lambda(1-4\lambda^2)-3\lambda(4\lambda^2+4\lambda+1)e^{-1/\lambda}\big)&
\quad \text{ if } \lambda>0,\\
\qquad 4\pi\kappa^2&\quad \text{ if }\lambda=0
\end{cases}
\]
and
\[
\lambda_* \coloneqq \argmin_{\lambda\in [0,+\infty)} f_\kappa(\lambda).
\]
In addition, $\lambda_*=0$ if $\kappa\leq (2\pi)^{-1/3}$ and $\lambda_*>0$ otherwise.
\end{prp}

\begin{proof}
Recall that $P(B_R)/|B_R|=3/R=6\lambda/\kappa$ and that from \cref{lem_Yukawa_nrj_of_balls} we have
\[
\frac{\I_{Y_{1,\kappa}}(B_R)}{|B_R|}
=12\pi \kappa^2\left( \frac{1}{3}-\lambda(1-4\lambda^2)-\lambda(4\lambda^2+4\lambda+1)e^{-1/\lambda}
\right),
\]
where $\lambda=\kappa/(2R)$.
We then have,
\[
\frac{\cE_{Y_{1,\kappa}}(B_R)}{|R|}
= \frac{6\lambda}{\kappa}+4\pi
\kappa^2\big(1-3\lambda(1-4\lambda^2)-3\lambda(4\lambda^2+4\lambda+1)e^{-1/\lambda}\big)\\
\eqqcolon f_\kappa(\lambda).
\]
There holds
\[
\lim_{\lambda\to 0^+} f_\kappa(\lambda) = 4\pi^2
\]
and extending $f_\kappa$ by continuity at $0$, $f_\kappa$ is smooth on $[0,+\oo)$. Moreover,
\[
f_\kappa(\lambda) \stackrel{\lambda\up\oo}\sim \frac{6\lambda}{\kappa}.
\]
We deduce that $f_\kappa$ admits a minimum. 
Next, we  compute
\[
\begin{multlined}
f_\kappa'(\lambda)
=\frac{6}{\kappa} - 12\pi \kappa^{2}\left( 1 - 12\lambda^{2} + \frac{e^{-1/\lambda}}{\lambda}\left(12\lambda^{3} +
12\lambda^2 + 5\lambda + 1\right) \right)\\
\end{multlined}
\]
and
\[
f''_\kappa(\lambda)
=12\pi \kappa^2\lt( 24 \lambda - \frac{e^{-1/\lambda}}{\lambda^3} \lt(24\lambda^4 +
24\lambda^3+12\lambda^{2}+4\lambda+1\rt)\rt).
\]
Notice that 
\[
24 \lambda^4 - e^{-1/\lambda} \lt(24\lambda^4 + 24^3+12\lambda^{2}+4\lambda+1\rt)>0
\iff
e^{1/\lambda} >
1+\lambda^{-1}+\frac{\lambda^{-2}}{2!}+\frac{\lambda^{-3}}{3!}+\frac{\lambda^{-4}}{4!},
\]
which is obviously true, thus $f_\kappa$ is strictly convex.
In particular,  $f_\kappa$ admits a unique minimizer $\lambda_*\in[0,+\infty)$.
In addition, we have 
\[
f_{\kappa}'(0) = \frac{6}{\kappa}-12\pi\kappa^2
\]
and we deduce that  $\lambda_*=0$ if $\kappa \leq (2\pi)^{-\frac{1}{3}}$, and $\lambda_*>0$
otherwise.
\end{proof}

\subsection{Energy-per-mass ratio of infinite cylinders}

\begin{prp}\label{prp_ecylyukawa}
Let $\kappa,l>0$ and $\lambda=\frac{\kappa}{2l}$. Setting
\[
I(s) \coloneqq \frac{l(2l-s)}{2}\arccos\lt(\frac{s}{2l}\rt)
+ls\arcsin\lt(\sqrt{\frac{2l-s}{4l}}\rt)
-\frac{s}{4}\sqrt{(2l)^2-s^2},
\]
we have
\[
\sigma_{Y_{1,\kappa}}^{3,\mathrm{cyl}}(l) =\frac2l+\frac{8}{l^2}\int_0^{2l} s K_0(s/\kappa)I(s)\,ds,
\]
where $K_0$ is a modified Bessel function of the second kind, defined by
\[
K_0(\sigma)= \int_{0}^\infty e^{-\sigma \cosh(\xi)} \, d\xi.
\]
\end{prp}

\begin{proof}
Recall that $B'_l$ denotes the disk of radius $l$ in dimension $2$ and let us recall the following formula from \cref{lem_cylinf},
\[
\sigma_{Y_{1,\kappa}}^{3,\cyl}(l)
=  \frac{2}{l}+\dfrac1{|B_l'|}\int_{B_l'\times B_l'} \int_\R Y_{1,\kappa}(x-y,t)\,dt\,dx\,dy.
\]
We see that we have to estimate the integral term:
\[
\cC_{1,\kappa}(l)
\coloneqq \int_{B'_l}\int_{B'_l} \int_\R Y_{1,\kappa}(x-y,t) \,dt\, dx\,dy.
\]
Unfortunately, the inner integral:
\[
S_{Y_{\kappa,\alpha}}^{3,\mathrm{cyl}} \coloneqq \int_\R Y_{1,\kappa}(x-y,t)\, dt 
\]
does not have an analytic expression and this prevents us from using \cref{cor:rieszcyl} to get a numerically computable formula for
$\sigma_{Y_{\kappa,1}}^{3,\mathrm{cyl}}(l)$.
Instead, using the change of variable  $y=r\tau$ with $r>0$, $\tau\in\Sp^1$ and using the fact that $Y_{1,\kappa}$ is radial, we compute
\begin{align*}
\cC_{1,\kappa}(l)=\int_{\Sp^1}\int_0^l r^{} \int_{B'_l} \int_\R
Y_{1,\kappa}(x-r\tau,t)\, dx \,dt\,dr\,d\h^1(\tau).
\end{align*}
For a fixed $\tau\in\Sp^1$, there exists a rotation $R$ in $\R^2$ such that $\tau=Re_1$. Performing the change of variable $x =R\tilde x$ and using the symmetry $Y_{1,\kappa}(Rz,t)=Y_{1,\kappa}(z,t)$ for $z\in\R^2$, we obtain
\begin{align*}
\cC_{1,\kappa}(l)
&=|\Sp^{1}|\ \int_0^l r^{} \int_{B'_l} \int_\R
Y_{1,\kappa}(x-re_1,t)\, dx \,dt\,dr\\
&=2\pi\int_0^l r^{} \int_{B'_l} \int_\R
\dfrac{e^{-\sqrt{|x-re_1|^2+t^2}/\kappa}}{(|x-re_1|^2+t^2)^{1/2}} \,dt\, dx\,dr.
\end{align*}
Decomposing the integration over $B_l'$ over the circular arcs $B_l'\partial B'_s(re_1)$, $s>0$, we have by Fubini,
\begin{align*}
\cC_{1,\kappa}(l)
&=2\pi\int_0^l r \int_0^{l+r}  \int_{B_l'\cap\{|x-re_1|=s\}} \int_\R
\dfrac{e^{-\sqrt{s^2+t^2}/\kappa}}{(s^2+t^2)^{1/2}}\,dt\,d\mathcal{H}^1(x)\, ds\,dr\\
&=2\pi\int_0^l r \int_0^{l+r}  C(r,s,l)\lt(  \int_\R \dfrac{e^{-\sqrt{s^2+t^2}/\kappa}}{(s^2+t^2)^{1/2}}\,dt\rt)\, ds\,dr,
\end{align*}
where $C(r,s,l)$ is the circle-length
\[
C(r,s,l)=|\partial B'_s(re_1)\cap B'_l| =\mathcal{H}^{1}(\partial B'_s(re_1)\cap B'_l).
\]
We slightly simplify the expression of $\cC_{1,\kappa}(l)$ by using the change of variable $t=su$ in the inner integral to get 
\begin{equation}\label{ecylyukawa:eq1}
\cC_{1,\kappa}(l)
=2\pi\int_0^l r \int_0^{l+r}  C(r,s,l)\lt( \int_\R \dfrac{e^{-(s/\kappa)\sqrt{1+u^2}}}{(1 + u^2)^{1/2}}\,du\rt)\, ds\,dr.
\end{equation}
For $s<l-r$, we have a full circle, that is $C(r,s,l)=2\pi l$. For $s>l+r$, $\partial
B'_s(re_1)\cap B'_l=\void$ and $C(r,s,l)=0$.
For $l-r<s<l+r$, we have $\partial B'_s(re_1)\cap \partial B'_l=\{(r+s\cos\gamma,\pm\sin\gamma)\}$
where $\gamma=\gamma(r,s,l)\in (0,\pi)$ satisfies
\[
(r+s\cos\gamma)^2 + (s\sin\gamma)^2=l^2 \iff \gamma=\arccos\lt(\dfrac{l^2-r^2-s^2}{2rs}\rt).
\]
Hence, setting
\[
\gamma(r,s,l)=
\begin{cases}
\qquad0 &\text{if }s<l-r\\
\arccos\lt(\dfrac{l^2-r^2-s^2}{2rs}\rt)&\text{if } l-r<s<l+r\\
\qquad\pi&\text{if }s>l+r,
\end{cases}
\]
we have
\[
C(r,s,l) = 2(\pi-\gamma(r,s,l)).
\]
Inserting this into \cref{ecylyukawa:eq1}, we obtain
\[
\cC_{1,\kappa}(l) =4\pi\int_0^l r\int_0^{l+r} s(\pi-\gamma(r,s,l))\lt(  \int_\R \dfrac{e^{-(s/\kappa)\sqrt{1+u^2}}}{(1 +
u^2)^{1/2}}\,du\rt)\, ds\,dr.
\]
Using Fubini, this rewrites as
\be\label{proof_prp_ecylyukawa}
\cC_{1,\kappa}(l)
=4\pi \int_0^{2l}s\lt[
\lt(  \int_\R \dfrac{e^{-(s/\kappa)\sqrt{1+u^2}}}{(1 + u^2)^{1/2}}\,du\rt) \lt(\int_{(s-l)_+}^l r
\varphi(r,s,l)\, dr\rt) \rt]\,ds,
\ee
where we have set 
\[
\begin{aligned}
\varphi(r,s,l)
\coloneqq&\pi-\gamma(r,s,l)\\
=&\begin{cases}
\qquad\pi &\text{if } r <l-s,\\
\arccos\lt(\dfrac{r^2+s^2-l^2}{2rs}\rt)&\text{if }\big[s>l\text{ and } r>s-l\big]\text{ or }\big[s<l\text{ and }r>l-s\big]\\
\qquad0&\text{if }s>l\text{ and } r<s-l.
\end{cases}
\end{aligned}
\]
Now we compute the two inner integrals in~\cref{proof_prp_ecylyukawa} in terms of well-known special
functions. We start with the integral with respect to $u$. Using $u = \sinh(\xi)$, we find
\begin{equation}\label{eballyukawa:modBessel}
\int_{\mathbb{R}} \frac{e^{-\sigma \sqrt{1+u^2}}}{\sqrt{1+u^2}} \, du
= \int_{\mathbb{R}} e^{-\sigma \cosh(\xi)} \, d\xi
= 2K_0(\sigma),
\end{equation}
where $K_0$ is a modified Bessel function of the second kind (see, e.g., \cite[p.~181,~(5)]{Wat1995}). Recall that $K_0$ is a smooth integrable function on $(0,+\infty)$ with 
\[
K_0(x)\stackrel{x\dw 0^+}{\sim}
\log(1/x).
\]
Next, we compute the integral with respect to $r$ in~\cref{proof_prp_ecylyukawa}.
If $s<l$, we have
\begin{equation}\label{cutvarphi:eq1}
\int_{(s-l)_+}^l r \varphi(r,s,l)\, dr
=\pi\dfrac{(l-s)^2}2 + \int_{l-s}^l r\arccos\lt(\dfrac{r^2+s^2-l^2}{2rs}\rt)\, dr.
\end{equation}
If $s>l$, we have
\begin{equation}\label{cutvarphi:eq2}
\int_{(s-l)_+}^l r \varphi(r,s,l)\, dr
= \int_{s-l}^lr\arccos\lt(\dfrac{r^2+s^2-l^2}{2rs}\rt)\, dr.
\end{equation}
Let us set $a=l-s$ and $b=l+s$. We search for a primitive function of
$r\arccos\big(\tfrac{r^2+s^2-l^2}{2rs}\big)$ on $(|a|,l)\subset(|a|,b)$.
Integrating by parts, we have
\begin{multline}\label{arccosint:eq1}
\int r\arccos\lt(\dfrac{r^2+s^2-l^2}{2rs}\rt)\, dr
= \frac{r^2}{2}\arccos\lt(\dfrac{r^2+s^2-l^2}{2rs}\rt)
+\frac{1}{2}\int \frac{r(r^2+ab)}{\sqrt{b^2-r^2}\sqrt{r^2-a^2}}\,dr.
\end{multline}
Notice that
\begin{equation}\label{arccosint:eq2}
\int \frac{r}{\sqrt{b^2-r^2}\sqrt{r^2-a^2}}\,dr
=\frac{1}{2}\arcsin\lt(\frac{2r^2-a^2-b^2}{b^2-a^2}\rt).
\end{equation}
Then we have
\begin{multline}\label{arccosint:eq3}
\int \frac{r^3}{\sqrt{b^2-r^2}\sqrt{r^2-a^2}}\,dr
= \int \frac{r(r^2-a^2)+a^2 r}{\sqrt{b^2-r^2}\sqrt{r^2-a^2}}\,dr\\
= \int r\frac{\sqrt{r^2-a^2}}{\sqrt{b^2-r^2}}\,dr+ \frac{a^2}{2}\arcsin\lt(\frac{2r^2-a^2-b^2}{b^2-a^2}\rt).
\end{multline}
Eventually, setting $\tau=(a/b)^2$ and using the successive change of variables $t=r/b$, $y=t^2$, and $t=\tau+(1-\tau)\sin^2\theta$,
we compute
\begin{equation}\label{arccosint:eq4}
\begin{aligned}
\int_{x_0}^x r\frac{\sqrt{r^2-a^2}}{\sqrt{b^2-r^2}}\,dr
&=b^2\int_{x_0/b}^{x/b} t\frac{\sqrt{t^2-(a/b)^2}}{\sqrt{1-t^2}}\,dt\\
&=\frac{b^2}{2}\int_{(x_0/b)^2}^{(x/b)^2} \frac{\sqrt{y-\tau}}{\sqrt{1-y}}\,dy\\
&=(b^2-a^2)\int_{\arcsin\lt(\sqrt{\frac{x_0^2-a^2}{b^2-a^2}}\rt)}^{\arcsin\lt(\sqrt{\frac{x^2-a^2}{b^2-a^2}}\rt)}
\sin^2\theta\,d\theta\\
&=\frac{(b^2-a^2)}{2}\bigg[\theta-\sin\theta\cos\theta
\bigg]_{\arcsin\lt(\sqrt{\frac{x_0^2-a^2}{b^2-a^2}}\rt)}^{\arcsin\lt(\sqrt{\frac{x^2-a^2}{b^2-a^2}}\rt)}\\
&=\frac{(b^2-a^2)}{2}\lt[\arcsin\lt(\sqrt{\frac{x^2-a^2}{b^2-a^2}}\rt)-\sqrt{\frac{x^2-a^2}{b^2-a^2}}\sqrt{\frac{b^2-x^2}{b^2-a^2}}\rt]
+\text{ constant}\\
&=\frac{1}{2}(b^2-a^2)\arcsin\lt(\sqrt{\frac{x^2-a^2}{b^2-a^2}}\rt)-\frac{1}{2}\sqrt{(x^2-a^2)(b^2-x^2)}+\text{ constant}.
\end{aligned}
\end{equation}
Gathering \cref{arccosint:eq1}, \cref{arccosint:eq2}, \cref{arccosint:eq3} and
\cref{arccosint:eq4}, we find the primitive function
\begin{equation}\label{arccosint:exp}
\begin{aligned}
&\int r\arccos\lt(\dfrac{r^2+s^2-l^2}{2rs}\rt)\, dr\\
&\quad= \frac{r^2}{2}\arccos\lt(\dfrac{r^2+s^2-l^2}{2rs}\rt)
+\frac{l(l-s)}{2}\arcsin\lt(\frac{r^2-l^2-s^2}{2ls}\rt)\\
&\qquad\qquad\qquad+ls\arcsin\lt(\sqrt{\frac{r^2-(l-s)^2}{4ls}}\rt)-\frac{1}{4}\sqrt{(r^2-(l-s)^2)((l+s)^2-r^2)}\\
&\quad\eqqcolon \psi(r,s,l).
\end{aligned}
\end{equation}
Summarizing, from \cref{proof_prp_ecylyukawa}, \cref{eballyukawa:modBessel}, \cref{cutvarphi:eq1}
and \cref{cutvarphi:eq2}, we deduce
\begin{equation}\label{eballyukawa:summary}
\cC_{1,\kappa}(l)
=8\pi\int_0^{2l} sK_0(s/\kappa)I(s)\,ds,
\end{equation}
with
\[
I(s)=
\begin{cases}
\dfrac\pi2 (l-s)^2+(\psi(l,s,l)-\psi(l-s,s,l))&\qquad\text{ if }s<l\\
\psi(l,s,l)-\psi(s-l,s,l) &\qquad\text{ if }s>l.
\end{cases}
\]
For $s\neq l$, we compute
\[
\psi(|s-l|,s,l) = \frac{(l-s)^2}{2}\arccos(\mathrm{sign}(s-l))-\frac{l(l-s)}{2}\frac{\pi}{2}
\]
and
\[
\psi(l,s,l) = \frac{l^2}{2}\arccos\lt(\frac{s}{2l}\rt)-\frac{l(l-s)}{2}\arcsin\lt(\frac{s}{2l}\rt)
+ls\arcsin\lt(\sqrt{\frac{2l-s}{4l}}\rt)
-\frac{s}{4}\sqrt{(2l)^2-s^2}.
\]
In both cases $s<l$ or $s>l$ for $I$ we find
\[
\begin{aligned}
I(s)
&= \frac{l^2}{2}\arccos\lt(\frac{s}{2l}\rt)-\frac{l(l-s)}{2}
\lt[\frac{\pi}{2}-\arcsin\lt(\frac{s}{2l}\rt)\rt]
+ls\arcsin\lt(\sqrt{\frac{2l-s}{4l}}\rt)
-\frac{s}{4}\sqrt{(2l)^2-s^2}\\
&= \frac{l(2l-s)}{2}\arccos\lt(\frac{s}{2l}\rt)
+ls\arcsin\lt(\sqrt{\frac{2l-s}{4l}}\rt)
-\frac{s}{4}\sqrt{(2l)^2-s^2}.
\end{aligned}
\]
We conclude the proof by recalling \cref{eballyukawa:summary} and 
\[
\sigma_{Y_{1,\kappa}}^{3,\cyl}(l)
= \frac{2}{l}+\frac{\cC_{1,\kappa}(l)}{\pi l^2}.
\]
\end{proof}

\subsection{Energy-per-mass ratio of balls versus infinite cylinders}\label{subsec:yukawaballvscyl}

Now we compare the energy-per-mass of balls and long cylinders in the case of the Yukawa potential
$Y_{1,\kappa}$. We establish the following.

\begin{lem}\label{lem_ecylball_y}
Let $\kappa=0.56$. There holds $\rho_{Y_{1,0.56}}^{3,\cyl} < \rho_{Y_{1,0.56}}^{3,\ball}$.
\end{lem}

\begin{proof}
For this proof, we frequently rely on symbolic computations. Further details are provided in the
Jupyter Notebook included as supplementary material.\\
First, we look for a lower bound for $\rho^{3,\mathrm{ball}}_{Y_{1,\kappa}}$.
Using Newton method to find a zero of $f_\kappa'$ numerically, we observe that the minimum is
likely reached at
\[
\lambda_* \subset (\alpha,\beta),
\]
where $\alpha=8.84 \times 10^{-2}$ and $\beta=\alpha+10^{-4}$.
A posteriori, we ensure that this holds true by using the strict convexity of $f_\kappa$ and
checking with a symbolic computation library (we used Python's SymPy) that
\[
f_\kappa'(\alpha) f_\kappa'\lt(\beta\rt) < 0.
\]
By convexity again, the graph of $f_\kappa$ lies above its tangent lines. In
particular, setting $T_\alpha(x)=f_\kappa(\alpha)+f_\kappa'(\alpha)(x-\alpha)$ and
$T_\beta(x)=f_\kappa(\beta)+f_\kappa'(\beta)(x-\beta)$, we have the lower bound
\begin{equation}\label{yukawalowerbound}
\rho^{3,\mathrm{ball}}_{Y_{1,\kappa}}
\geq \min (T_\alpha(\beta),T_\beta(\alpha))
\geq 3.8755.
\end{equation}
The above numerical lower bound is guaranteed by SymPy, which handles error propagation to evaluate
expressions to the desired precision.
Second, we search for an upper bound of $\rho^{3,\mathrm{cyl}}_{Y_{1,\kappa}}$. By definition, any
$\sigma^{3,\cyl}_{1,\kappa}(l)$ is an upper bound.
Recall that by \cref{prp_ecylyukawa}, $\sigma^{3,\mathrm{cyl}}_{Y_{1,\kappa}}(l)$ is given by
\[
\sigma_{Y_{1,\kappa}}^{3,\cyl}(l)=\frac{2}{l}+\frac{8}{l^2}\mathcal{J}_{\kappa}(l),
\]
where we have set
\[
\mathcal{J}_\kappa(l) = \int_0^{2l} F_\kappa(s)\,ds
\quad\text{ and }\quad
F_\kappa(s)=s K_0(s/\kappa)I(s).
\]
Here $K_0$ is a modified Bessel function of the second kind, and
\[
I(s) = \frac{l(2l-s)}{2}\arccos\lt(\frac{s}{2l}\rt) +ls\arcsin\lt(\sqrt{\frac{2l-s}{4l}}\rt)
-\frac{s}{4}\sqrt{(2l)^2-s^2}.
\]
From the integral expression of $K_0$ given in \cref{eballyukawa:modBessel}, we see that $K_0$ is
nonincreasing. In addition, after computation we find
\begin{equation}\label{derivI}
I'(s) = 
-\frac{\sqrt{(2l)^2-s^2}}{2}
-\frac{l}{2}\arccos\lt(\frac{s}{2l}\rt)
+l\arcsin\lt(\sqrt{\frac{2l-s}{4l}}\rt)
\quad\text{ and }\quad
I''(s) = \frac{s}{2\sqrt{(2l)^2-s^2}}.
\end{equation}
We see that $I''(s)$ is positive, thus $I'$ is increasing on $(0,2l)$. Since $I'(2l)=0$, it
follows that $I'$ is non-positive and thus $I$ is nonincreasing.
Hence, for any segment $[a,b]\subset (0,2l]$, we have
\begin{equation}\label{Fubound1}
F_\kappa(s) \leq sI(a)K_0(a/\kappa),\qquad\forall s\in[a,b].
\end{equation}
In addition, by \cite[Corollary~3.3]{Gau2014} we have
\[
K_0(x) < e^{-x}\sqrt{\frac{\pi}{2x}},\qquad\forall x>0,
\]
so that
\begin{equation}\label{Fubound2}
F_\kappa(s) \leq I(0)\sqrt{\frac{\pi\kappa s}{2}},\qquad\forall s\in(0,b).
\end{equation}
Splitting the interval $(0,2l)$ into the regular grid of size step  $h={2l}/{N}$ for some $N\ge 1$,
we get the estimate
\begin{equation}\label{boundJk}
\begin{aligned}
\mathcal{J}_\kappa(l)
&= \sum_{k=0}^{N-1} \int_{kh}^{(k+1)h} F_\kappa(s)\,ds\\
&\leq I(0)\sqrt{\frac{\pi\kappa}{2}}\int_0^{h}s\,ds
+\sum_{k=1}^{N-1} I(kh)K_0\lt(\frac{kh}{\kappa}\rt)\int_{kh}^{(k+1)h} s\,ds\\
&= \frac{\sqrt{2}}{3}I(0)\sqrt{\pi\kappa}|h|^{3/2}+\sum_{k=1}^{N-1}
\frac{|h|^2}{2}\lt(k+\frac{1}{2}\rt)I(kh)K_0\lt(\frac{kh}{\kappa}\rt).
\end{aligned}
\end{equation}
Here we used \cref{Fubound2} to estimate $F_\kappa$ on the interval $(0,h)$, and \cref{Fubound1} on the
intervals $[kh,(k+1)h]$ with $k\geq 1$.
Python's SymPy library can compute the above sum (and thus $\sigma_{Y_{1,\kappa}}^{3,\cyl}(l)$) to
desired precision by handling errors propagation along the computations.
Choosing $l=2.09$ and $N=3\cdot 10^4$, SymPy guarantees the upper bound
\begin{equation}\label{yukawaupperboundfcyl}
\sigma_{Y_{1,0.56}}^{3,\cyl}(2.09) \leq 3.8747,
\end{equation}
which implies
\begin{equation}\label{yukawaupperbound}
\rho_{Y_{1,0.56}}^{3,\cyl} \leq 3.8747.
\end{equation}
This concludes the proof.
\end{proof}
\begin{rmk}
We used the upper bound~\cref{boundJk} and SymPy to obtain the upper bound~\cref{yukawaupperboundfcyl} for safety, but a numerical
evaluation of the integral using a more precise quadrature formula (see the discussion further below) actually gives
\[
\sigma_{Y_{1,\kappa}}^{3,\cyl}(2.09) \simeq 3.8730.
\]
\end{rmk}

\cref{lem_ecylball_y} implies \cref{mainprp_prp1} and thus \cref{mainthm:thm1} in the case of the Yukawa potential, in
view of \cref{prp_ecyleball_nonsphere}.

\vspace{10pt}

In fact, numerical computations indicate that $\rho_{Y_{1,\kappa}}^{3,\cyl} \leq
\rho_{Y_{1,\kappa}}^{3,\cyl}$ holds on a whole interval of values of $\kappa$ around $0.56$, see
\cref{fig:yukawa}.
\begin{figure}[ht]
    \begin{minipage}{0.71\textwidth}
         \begin{center}
 		\captionsetup{width=.95\textwidth}
 		\includegraphics[width=.95\textwidth]{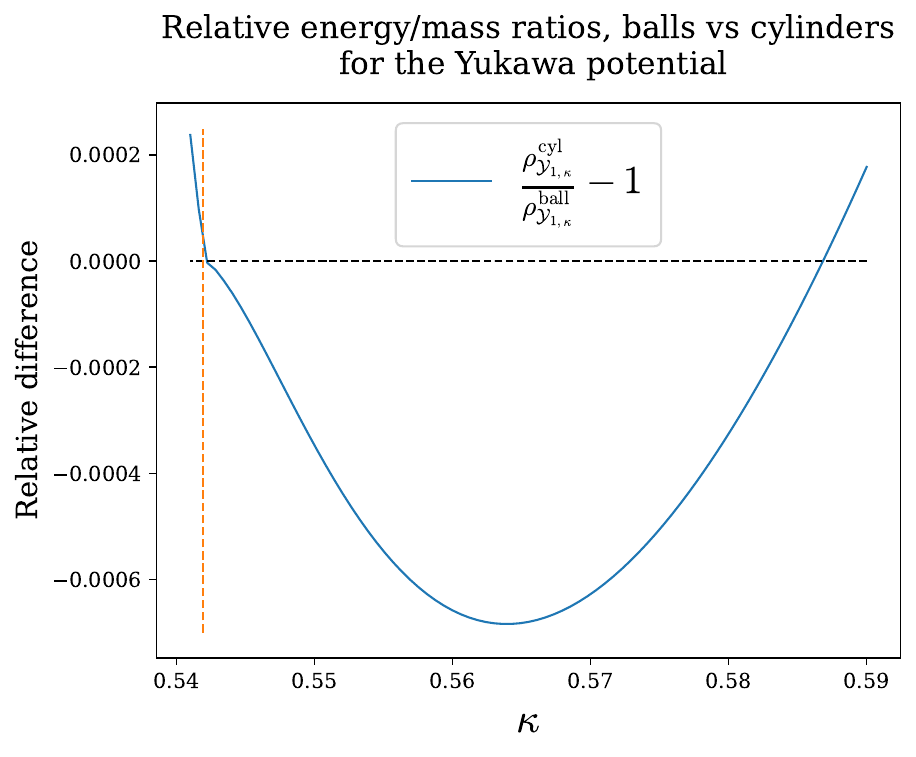}
 		\end{center}
 	\end{minipage}
 	\caption{\label{fig:yukawa}  Comparison of the energy-per-mass ratio for balls and infinite  cylinders with the Yukawa
	potential for varying cut-length $\kappa$. We used $\sigma_{Y_{1,\kappa}}^{3,\cyl}(l(\kappa))$ with
	appropriate $l(\kappa)$ as an upper bound for $\rho_{Y_{1,\kappa}}^{3,\cyl}$. As in the previous section, such $l(\kappa)$
	are obtained using SciPy's library to optimize $l\mapsto
	\sigma_{Y_{1,\kappa}}^{3,\cyl}(l)$.\\ 
	The vertical dashed line corresponds to
	$\kappa=(2\pi)^{-\frac{1}{3}}$, the threshold below which the infimum
	$\rho^{3,\mathrm{ball}}_{Y_{1,\kappa}}$ is obtained by taking balls $B_R$ with $R\to\infty$.
	}
\end{figure}
%
%
%
%

Let us detail how we  compute accurate numerical approximations of $\sigma_{Y_{1,\kappa}}^{3,\cyl}(l)$ in order to produce~\cref{fig:yukawa}.
We have to integrate numerically the function $F_\kappa(s)=s\mapsto s K_0(s/\kappa)I(s)$.
We would like to use Simpson's quadrature formula, which gives an error in $O(|h|^4)$ for
\textsl{smooth functions}, where $h$ is the step size of the grid. However, the effective order of
convergence depends on the regularity of the integrand, and it turns out that  $F_\kappa$  is not
smooth in $0$ nor in $2l$. 
\\
Let us study the regularity of  $F_\kappa$. We see that $I$ is smooth on $[0,2l)$ and $K_0$ is smooth on $(0,+\infty)$, so that $F_\kappa$ is smooth on
$(0,2l)$. The only issues appear in $0$ and in $2l$.\\[5pt]
\underline{Behavior in $0$}.
We have (see \cite[Chapter~III,~\string{3.71,~(14)\string} and \string{3.7,~(2)\string}]{Wat1995})
\[
K_0(s)\stackrel{s\dw 0}{=} -\log(s)+\log 2-\gamma+\frac{s^2}{4}\lt(\gamma-1-\log 2\rt)-\frac{s^2}{4}\log(s)
+O(s^4\log(s)),
\]
where $\gamma$ is Euler's constant.
Hence, $F_\kappa$ is continuous in zero with a singularity of the form $s\log(1/s)$. We have the
asymptotic expansion
\begin{multline*}
F_\kappa(s)
= s\lt(-\log(s/\kappa)+\log 2-\gamma+\frac{1}{4}\lt(\frac{s}{\kappa}\rt)^2(\gamma-1-\log 2)
-\frac{1}{4}\lt(\frac{s}{\kappa}\rt)^2\log(s/\kappa)+O(s^4\log(s))\rt)\\
\times\lt(I(0)+I'(0)s+\frac{I''(0)}{2}s^2+O(s^3)\rt).
\end{multline*}
Subtracting the more singular terms
\[
 G_\kappa^{(1)} (s)
=-\log(s/\kappa)\lt[
\kappa I(0)\lt(\frac{s}{\kappa}\rt)
+\kappa^2I'(0)\lt(\frac{s}{\kappa}\rt)^2
\rt]
\]
to $F_\kappa$, we get that the non-smooth leading order term of $F_\kappa(s)-G_\kappa^{(1)} (s)$ is of the form $s^3\log(s)$ near $s=0$. It follows that using Simpson's quadrature formula to integrate $F_\kappa-G_\kappa^{(1)}$ on the interval $(0,l)$ yields an error in $O(|h|^4)$.\\[5pt]
\underline{Behavior in $2l$}.
Since $I(2l)=0$, we deduce from~\cref{derivI} the equivalent
\[
F(s)~ \stackrel{s\up 2l}{\sim}~ \frac{4}{3}l^{3/2}K_0\lt(\frac{2l}{K}\rt)(2l-s)^{3/2} \eqqcolon
G^{(2)}_\kappa(s).
\]
The contribution of such term in the quadrature error  using
Simpson's quadrature formula is in $O(|h|^{5/2})$ (see~\cite{CN2002}). To gain an order $1$, we subtract $ G_\kappa^{(2)} $ to $F_\kappa$. The first term in the
asymptotic expansion of $F_\kappa- G_\kappa^{(2)} $ at $2l$ is then of the form $(2l-s)^{5/2}$, which gives a
contribution in $O(|h|^{7/2})$ to the error.\\[5pt]

Subtracting the leading singular parts at 0 and $2l$, we define 
\[
F_\kappa^{\mathrm{reg}} :=F_\kappa- G_\kappa^{(1)} - G_\kappa^{(2)},
\]
and we decompose
\[
\int_0^{2l} F_\kappa(s)\,ds = \int_0^{2l} F_\kappa^{\mathrm{reg}}(s)\,ds+\int_0^{2l}  G_\kappa^{(1)} (s)+ G_\kappa^{(2)} (s)\,ds.
\]
On the one hand, we use the following closed formula for the integration of the singular parts:
\begin{multline*}
\int_0^{2l}  G_\kappa^{(1)} (s)+ G_\kappa^{(2)} (s)\,ds\\
=
l^2 I(0)\lt(1-2\log\lt(\frac{2l}{\kappa}\rt)\rt)+\frac{8}{9}l^3I'(0)\lt(1-3\log\lt(\frac{2l}{\kappa}\rt)\rt)
+\frac{32\sqrt{2}}{15}l^4K_0\lt(\frac{2l}{\kappa}\rt).
\end{multline*}
On the other hand we use  Simpsons's quadrature
formula to integrate $F_\kappa^{\mathrm{reg}}$.  The above considerations indicate that the
quadrature error is in $O(|h|^{7/2})$ which is efficient enough to reach machine precision with a
low numerical cost. In practice, we observe a rate of convergence in $O(|h|^4)$,
see~\cref{fig:cv_num_integ_yukawa}, due to the small prefactor $sK_0(s/\kappa)\simeq 10^{-3}$ in
$s=2l$ and the fact that machine precision is reached for $|h|\simeq 10^{-3}$.

\begin{figure}[ht]
\captionsetup{width=.95\textwidth}
\includegraphics[width=.65\textwidth]{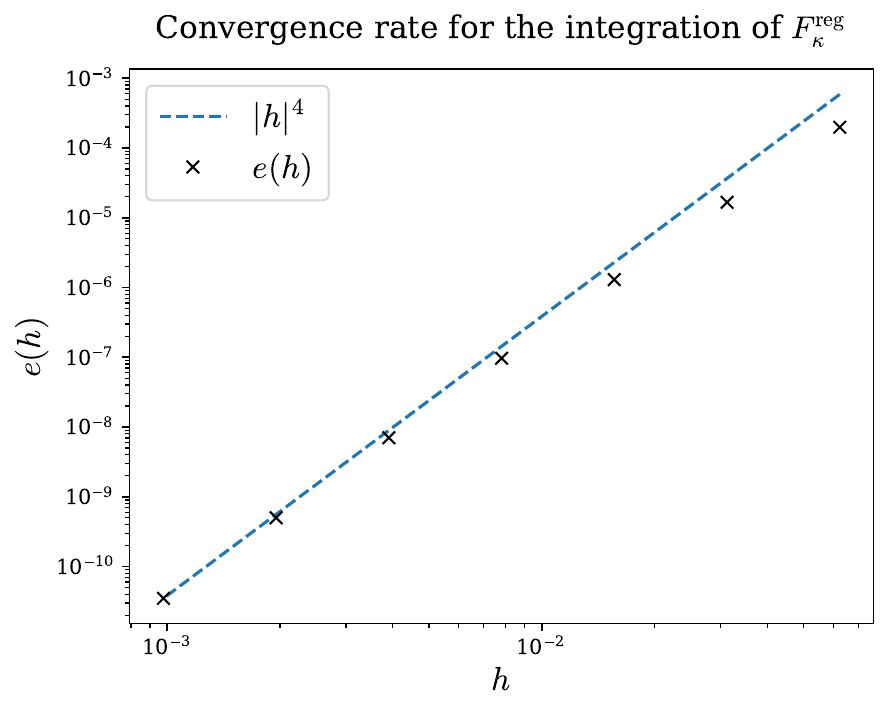}
\vspace{-10pt}
\caption{\label{fig:cv_num_integ_yukawa} Convergence rate for the numerical integration of
$F^{\mathrm{reg}}_\kappa$ on $(0,2l)$ using Simpson's quadrature formula on a regular grid. Here $\kappa=0.56$, $l=2.09$ and $e(h)$ is the 
 error using step size $h$. For the ``exact'' value we use  Simpson's quadrature formula on a very  fine grid with step size $2l/2^{14}$.}
\end{figure}

\begin{rmk}\label{rmk:yukawa}
The same strategy would yield \cref{conj:yukawa}, that is, existence of non-radial minimizers in the
Yukawa case, provided one can show that
$\rho^{3,\cyl}_{Y_{1,\kappa}}<\rho^{3,\shell}_{Y_{1,\kappa}}$ for some $\kappa>0$, which we expect
by analogy with the truncated Coulomb case. We were however not able to compute
$\rho^{3,\shell}_{Y_{1,\kappa}}$. Contrary to the truncated Coulomb case, the Yukawa energies are
not given by expressions which can easily be optimized by hand: Bessel functions appear throughout,
and already the computation of $\rho^{3,\ball}_{Y_{1,\kappa}}$ required symbolic computation.
\end{rmk}
~

%
%

\noindent
{\bf Acknowledgments.}

L. Bronsard is supported by an NSERC Discovery Grant.
B. Merlet and M. Pegon are partially supported by the ANR projects STOIQUES (ANR-24-CE40-2216) and
ESSTOS (ANR-25-CE40-4799).
They also acknowledge support from the CDP C2EMPI and are grateful to the French
State under the France 2030 programme, the University of Lille, the Initiative
d'excellence de l'Université de Lille, and the Métropole Européenne de Lille for
funding the R-CDP-24-004-C2EMPI project.

\printbibliography

\end{document}